\documentclass[10pt,reqno]{article}

\usepackage[b5paper,top=1.2in,left=0.9in]{geometry}
\usepackage{verbatim}
\usepackage{amssymb}
\usepackage{amsmath}
\usepackage{graphicx}
\usepackage{appendix}
\usepackage{color}
\usepackage{amsthm}
\usepackage{tikz}
\usepackage{ifthen}
\usepackage{float}
\usepackage{pdflscape}
\usepackage{rotating}
\usepackage{hyperref}
\usepackage{courier}
\usepackage{multirow}
\usetikzlibrary{arrows,positioning,decorations.pathmorphing,  decorations.markings}

\renewcommand{\d}{\mathrm{d}}
\newcommand{\D}{\mathrm{D}}

\newcommand{\e}{\mathrm{e}}

\newtheorem{Thm}{Theorem}[section]
\newtheorem{Lem}[Thm]{Lemma}
\newtheorem{Prop}[Thm]{Proposition}
\newtheorem{Cor}[Thm]{Corollary}
\newtheorem{Con}[Thm]{Conjecture}
\newtheorem*{MainThm}{Main Theorem}

\theoremstyle{definition}
\newtheorem{Def}[Thm]{Definition}
\newtheorem{Rem}[Thm]{Remark}
\newtheorem{Nota}[Thm]{Notation}
\newtheorem{Ex}[Thm]{Example}

\newtheoremstyle{named}{}{}{\itshape}{}{\bfseries}{.}{.5em}{#1 #3}
\theoremstyle{named}

\def\R{\mathbb{R}}
\def\Q{\mathbb{Q}}
\def\N{\mathbb{N}}
\def\C{\mathbb{C}}
\def\Z{\mathbb{Z}}

\def\fD{\mathfrak{D}}

\def\fb{\mathfrak{b}}
\def\sl{\mathfrak{sl}}
\def\g{\mathfrak{g}}

\def\fh{\mathfrak{h}}

\def\cO{\mathcal{O}}

\def\cC{\mathcal{C}}

\def\cE{\mathcal{E}}

\def\cH{\mathcal{H}}

\def\cK{\mathcal{K}}

\def\cM{\mathcal{M}}

\def\cP{\mathcal{P}}
\def\cR{\mathcal{R}}

\def\cU{\mathcal{U}}
\def\cV{\mathcal{V}}
\def\cW{\mathcal{W}}
\def\cX{\mathcal{X}}

\def\a{\alpha}
\def\b{\beta}
\def\e{\varepsilon}
\def\G{\Gamma}
\def\c{\gamma}
\def\D{\Delta}

\def\d{\delta}

\def\k{\kappa}
\def\h{\theta}
\def\l{\lambda}
\def\L{\Lambda}

\def\w{\omega}

\def\bA{\mathbf{A}}

\def\bb{\mathbf{b}}
\def\bc{\mathbf{c}}
\def\bC{\mathbf{C}}
\def\bD{\mathbf{D}}
\def\bu{\mathbf{u}}

\def\bx{\mathbf{x}}

\def\be{\mathbf{e}}
\def\bE{\mathbf{E}}
\def\bf{\mathbf{f}}
\def\bF{\mathbf{F}}

\def\bi{\mathbf{i}}

\def\bK{\mathbf{K}}
\def\bL{\mathbf{L}}

\def\bp{\mathbf{p}}
\def\bQ{\mathbf{Q}}

\def\bT{\mathbf{T}}

\def\=>{\Longrightarrow}
\def\inj{\hookrightarrow}
\def\corr{\longleftrightarrow}

\def\to{\longrightarrow}

\def\ox{\otimes}
\def\o+{\oplus}
\def\bo+{\bigoplus}
\def\x{\times}
\def\<{\langle}
\def\>{\rangle}

\def\cong{\equiv}
\def\^{\wedge}
\def\+{\dagger}

\def\inv{^{-1}}
\def\half{\frac12}
\def\dis{\displaystyle}
\def\dd[#1,#2]{\frac{d#1}{d#2}}
\def\del[#1,#2]{\frac{\partial #1}{\partial #2}}
\def\:{\;:\;}
\def\ol{\mathring{\l}}
\def\rank{\mathrm{rank}}

\def\tab{\;\;\;\;\;\;}

\newcommand{\til}[1]{\widetilde{#1}}
\newcommand{\what}[1]{\widehat{#1}}

\newcommand{\mat}[1]{\begin{pmatrix}#1\end{pmatrix}}
\newcommand{\bin}[1]{\begin{bmatrix}#1\end{bmatrix}}

\newcommand{\case}[2][ll]{\left\{\begin{array}{#1}#2 \\ \end{array}\right.}
\newcommand{\Eq}[1]{\begin{align}#1\end{align}}
\newcommand{\Eqn}[1]{\begin{align*}#1\end{align*}}
\renewcommand{\vec}[1]{\overrightarrow{#1}}
\renewcommand{\over}[1]{\overline{#1}}

\tikzset{>=latex}
\tikzstyle{vthick}=[line width=1.8pt]

\newcommand\drawpath[2]{%
  \foreach \too [count=\c from 1] in {#1}
  {
  \ifthenelse{\c=1}
  {\xdef\from{\too}}
  {\path (\from) edge [->, #2] (\too);
    \xdef\from{\too}}
  };
}
\begin{document}
\title{Positive Representations with Zero Casimirs}

\author{  Ivan Chi-Ho Ip\footnote{
	  Department of Mathematics, Hong Kong University of Science and Technology\newline
	  Email: ivan.ip@ust.hk\newline		
	  Email: rman@connect.ust.hk
          }, Ryuichi Man\footnotemark[1]
}
\maketitle

\numberwithin{equation}{section}

\begin{abstract}
In this paper, we construct a new family of generalization of the positive representations of split-real quantum groups based on the degeneration of the Casimir operators acting as zero on some Hilbert spaces. It is motivated by a new observation arising from modifying the representation in the simplest case of $\cU_q(\sl(2,\R))$ compatible with Faddeev's modular double, while having a surprising tensor product decomposition. For higher rank, the representations are obtained by the polarization of Chevalley generators of $\cU_q(\g)$ in a new realization as universally Laurent polynomials of a certain skew-symmetrizable quantum cluster algebra. We also calculate explicitly the Casimir actions of the maximal $A_{n-1}$ degenerate representations of $\cU_q(\g_\R)$ for general Lie types based on the complexification of the central parameters. 
\end{abstract}
\vspace{3mm}

\vspace{3mm}


\tableofcontents
\section{Introduction}\label{sec:intro}

\subsection*{Motivation}
\emph{Positive representations} were introduced in \cite{FI} to study the representation theory of \emph{split real quantum groups} $\cU_{q}(\g_\R)$ associated to semisimple Lie algebra $\g$, as well as its \emph{modular double} $\cU_{q\til{q}}(\g_\R)$ introduced by \cite{Fa1, Fa2} in the regime where $|q|=1$. These representations are natural generalizations of a special family of representations of $\cU_q(\sl(2,\R))$ classified in \cite{Sch} and studied in detail by Teschner \emph{et al.} \cite{BT, PT1, PT2} from the physics point of view of quantum Liouville theory, in which they are characterized by the actions of the Chevalley generators $\<\be, \bf, \bK\>$ as \emph{positive self-adjoint operators} on the Hilbert space $L^2(\R)$. In general, this family of representations of $\cU_q(\g_\R)$, which we will refer to as the \emph{standard positive representations}, has been constructed explicitly for all Lie types \cite{FI, Ip2, Ip3}, and has since been given a cluster realization \cite{Ip7, SS1} as well as a geometric meaning in terms of the quantization of potential functions \cite{GS}, associated to moduli spaces of certain framed $G$-local systems \cite{FG1}.

In this paper, we discover a new family of representations of $\cU_q(\g_\R)$ which does not lie in the original family of the standard positive representations, but yet the Chevalley generators of the quantum group still act by positive operators. This is based on a simple observation in the $\sl_2$ case, where the generators can be re-expressed in terms of the Casimir element 
\Eq{
\bC=\bf \be-q\bK -q\bK\inv.
}
Let $q=e^{\pi \bi b^2}$ where $\bi:=\sqrt{-1}$ and $0<b<1$ such that $|q|=1$. In the family of the standard positive representation $\cP_\l$, $\bC$ acts as multiplication on $L^2(\R)$ by a positive real scalar
\Eq{
\pi_\l(\bC)=e^{2\pi b \l}+e^{-2\pi b\l}\geq 2} for a real parameter $\l\in \R_{\geq0}$. We observe that however, if we require that $\bC$ acts by zero instead, the resulting representation is still positive, since we can rewrite formally
\Eq{
\bf=\be\inv(\bC+q\bK+q\bK\inv) \leadsto q\be\inv \bK+q\be\inv \bK\inv
}
which is a positive expression. We call this the \emph{degenerate positive representation}, denoted by $\cP^0$. 

From another point of view, formally this can be obtained by setting the \emph{real} parameter $\l$ to certain special \emph{complex} values $\l_0=\pm\frac{\bi}{4b}$, so that we may consider $\cP^0:=:\cP_{\l_0}$ as some kind of analytic continuation of the standard positive representations $\cP_\l$. It turns out that $\cP^0$ also behaves surprisingly well under taking tensor product since it decomposes into a direct integral of the standard family $\cP_\l$ again. In some sense this is reminiscent of the \emph{complementary series} of the unitary representations of the split real group $SL(2,\R)$.

Finally, the action of the Casimir by zero is very special, in the sense that we have an embedding of (a quotient of) $\cU_q(\sl_2)$ into a skew-symmetrizable quantum cluster algebra $\cO_q(\cX)$. In particular, the image of the Chevalley generators are universally Laurent polynomials, such that $\cP^0$ is realized as a certain polarization of such embedding. This also induces a dual construction which is compatible with the modular double structure.

\subsection*{Degenerate positive representations}

Generalizing the motivation in $\sl_2$, we construct a new family of representations of $\cU_q(\sl(n,\R))$ by formally setting the generalized Casimir elements $\bC_k=0$ in the standard positive representations $\cP_\l$ in an appropriate cluster chart. The construction is based on the cluster realization of $\cU_q(\sl(n,\R))$ and its symmetric folding by cluster mutations due to \cite{SS2}. This leads to the consideration of a new, skew-symmetrizable cluster variety $\cX^0$, and we obtain the following main result (Theorem \ref{mainthmAn}). 

Let $\fD_q(\g)$ denote the Drinfeld's double of the Borel subalgebra of $\cU_q(\g)$.
\begin{MainThm}
There is an embedding of $\fD_q(\sl_{n})/\<\bC_k=0\>$ into a skew-symmetrizable quantum cluster algebra $\cO_q(\cX^0)$, such that the image of the Chevalley generators are universally Laurent polynomials.

Passing to a polarization, we have an irreducible representation $\cP^0$ of $\cU_q(\sl(n,\R))$ acting on a Hilbert space as positive operators, such that all the generalized Casimir operators $\bC_k$ act by zero.
\end{MainThm}

By reversing the multipliers of the symmetric folding, one also induces from $\cP^0$ a representation $\til{\cP}^0$ compatible with the modular double counterpart (Corollary \ref{cormainAn}). Hence in fact we have constructed two different new embeddings of $\cU_q(\sl(n,\R))$ with specialized Casimir actions.

Next, we proceed to discuss the representations for $\cU_q(\g_\R)$ for general Lie type, where $\rank(\g)=n$. If it is not of type $A_n$, the symmetric folding construction may not work. However, we can consider a parabolic subgroup $W_J\subset W$ of the Weyl group associated to a subset of the Dynkin index $J\subset I$. In \cite{Ip8}, we have constructed the \emph{parabolic positive representations} $\cP_\l^J$ of $\cU_q(\g_\R)$ based on $W_J$. Using the same argument presented in \cite{Ip8}, we proved that one can do the symmetric folding construction on different type $A$ parabolic parts, and obtain a new family of representations (Theorem \ref{mainthm}) which is again referred to as the \emph{degenerate positive representations}. 

\begin{MainThm} Given a parabolic subgroup $W_J\subset W$ of type $A$, there exists a new family of irreducible representations $\cP_\l^{0,J}$ of $\cU_q(\g_\R)$ parametrized by $\l\in \R^{n-|J|}$, such that the Chevalley generators are positive operators realized by a polarization of universally Laurent polynomials in a skew-symmetrizable quantum cluster algebra.
\end{MainThm}
Again the construction also yields another representation $\cP_{\til{\l}}^{0,J}$ compatible with the modular double counterpart (Theorem \ref{mainthmmod}).

\subsection*{Computation of generalized Casimirs}
We observed that the symmetric folding construction yielding $\cP^0$ can also be obtained formally by setting the parameters $\l$ to certain special complex values (which we call the \emph{general solution} of a symmetric equation). Since the generalized Casimirs $\bC_k$ of the original representation $\cP_\l$ act by scalars in terms of $\l$ only, one obtain the corresponding Casimir actions for $\cP_\l^{0,J}$ by substituting the specialized parameter $\l\in\C^n$ with appropriate \emph{complex shifts}. In the parabolic case, when $|J|=n-1$, the resulting representations $\cP_\l^{0,J}$, which we refer to as \emph{maximal degenerate representations}, are parametrized by a single number $\l\in\R$. 

Using the Weyl-type character formula developed in \cite{Ip5}, together with the explicit presentation of the weight spaces of the fundamental representations of $\g$, as well as some technical calculations involving the central characters of the folded quantum torus algebra, we compute explicitly all the actions of the generalized Casimirs of $\cP_\l^{0,J}$ in the case $W_J$ is of type $A_{n-1}$. This is summarized in Theorem \ref{thmother}--\ref{thmother2}.

\subsection*{Regular positive representations}
The original standard positive representations $\cP_\l$ \cite{FI, Ip2, Ip3}, the parabolic positive representations $\cP_\l^J$\cite{Ip8}, the degenerate representations $\cP_{\l}^{0,J}$ as well as their modular double counterpart $\cP_{\til{\l}}^{0,J}$ considered in this paper, all share the same important cluster theoretic properties. Namely, we have a homomorphism of the Drinfeld's double quantum group $\fD_q(\g)$ into a quantum torus algebra, such that the image of the Chevalley generators are universally Laurent polynomials. In other words, we have a homomorphism
$$\fD_q(\g)\to \cO_q(\cX)$$
to the quantum algebra of regular functions of a cluster variety $\cX$, or equivalently, the quantum upper cluster algebra of $\cX$. Furthermore, the representations are recovered from a polarization of any cluster chart of $\cO_q(\cX)$ as positive operators.

This motivates the definition of \emph{regular positive representations} (Definition \ref{regposrep}), and the new goal is to classify all the irreducible regular positive representations up to unitary equivalence. We propose in Conjecture \ref{mainconj} that these are classified by the 4 types of positive representations above, as well as their appropriate mixtures.

\subsection*{Outline} 

The paper is organized as follows. In Section \ref{sec:pre}, we set the notations and recall the basic construction of the positive representations of $\cU_q(\g_\R)$ via the polarization of its cluster embedding into a certain quantum torus algebra. We also recall some results on the calculation of the generalized Casimir operators. In Section \ref{sec:sl2}, we focus on the case of $\cU_q(\sl(2,\R))$ and discuss the main observations and results that motivate the general construction in higher rank. In Section \ref{sec:An}, we give the symmetric folding construction in type $A_n$, where the mutation sequence and the change of central parameter are outlined in Appendix \ref{App:mutate}. In Section \ref{sec:higher}, we state the main results for general Lie type by parabolic folding, and explain the computation of the Casimir action for the maximal degenerate representations. Finally, in Section \ref{sec:class}, we discuss the classification of the regular positive representations, and illustrate with an example in type $A_2$.

\subsection*{Acknowledgment} We would like to thank J\"{o}rg Teschner for some insightful discussions made many years ago regarding the motivations in the case $\cU_q(\sl(2,\R))$ when this project is still in its infancy. We would also like to thank Gus Schrader for some helpful comments. The first author is supported by the Hong Kong RGC General Research Funds ECS \#26303319.
\section{Prerequisites}\label{sec:pre}
\subsection{Root systems}\label{sec:pre:root}
Let $\g$ be a finite-dimensional semisimple Lie algebra over $\C$. Let $I$ be the root index of the Dynkin diagram of $\g$ such that 
\Eq{
|I|=n=\rank(\g).}
Let $\Phi$ be the set of roots of $\g$. Let $\Pi_+ := \{ \a_i \}_{i \in I}$ be the set of simple roots and $\D_+$ the set of positive roots. Let $W=\<s_i\>_{i\in I}$ be the Weyl group generated by the simple reflections $s_i:=s_{\a_i}$. We write  
\Eq{
N=l(w_0)} to be the length of the longest element of $W$.

\begin{Def} Let $(-,-)$ be a $W$-invariant inner product of the root lattice. We define
\Eq{
a_{ij}:=\frac{2(\a_i,\a_j)}{(\a_i,\a_i)},\tab i,j\in I
}
such that $A:=(a_{ij})$ is the \emph{Cartan matrix}. 

We normalize $(-,-)$ as follows: we choose the symmetrization factors (also called the \emph{multipliers}) such that for any $i\in I$,
\Eq{\label{di}d_i:=\frac{1}{2}(\a_i,\a_i)=\case{1&\mbox{$i$ is long root or in the simply-laced case,}\\\frac{1}{2}&\mbox{$i$ is short root of type $B,C,F$,}\\\frac{1}{3}&\mbox{$i$ is short root of type $G$,}}}
and $(\a_i,\a_j)=-1$ when $i,j \in I$ are adjacent in the Dynkin diagram, such that
\Eq{
d_ia_{ij}=d_ja_{ji}.
}
\end{Def}
\begin{Def} We denote the \emph{simple coroots} by \Eq{
H_i:=\a_i^\vee:=\frac{2\a_i}{(\a_i,\a_i)}\in\fh,}
the \emph{fundamental weights} dual to $H_i$ by
\Eq{
w_i:=\sum_j (A\inv)_{ji}\a_j\in \fh_\R^*,}
and the \emph{fundamental coweights} dual to $\a_i$ by
\Eq{\label{Wieq}W_i:=\sum_j(A\inv)_{ij}H_j\in\fh_\R.}
We also let
\Eq{
\rho:=\frac12\sum_{\a\in \D_+}\a = \sum_i w_i = \sum_i d_i W_i.
}
be the half sum of positive roots.
\end{Def}
\begin{Def}
The Weyl group $W$ acts on the fundamental coweights by
\Eq{\label{weylact}
s_i\cdot W_j = W_j - \d_{ij}\a_j^\vee=W_j-\d_{ij}\sum_{k=1}^n a_{jk}W_k.
}
\end{Def}

\begin{Def} Let $w_0\in W$ be the longest element of the Weyl group. The \emph{Dynkin involution}
\Eq{
I&\to I\nonumber\\
i&\mapsto i^*
}
is defined by
\Eq{w_0 s_i w_0 = s_{i^*}.}
Equivalently, we have
\Eq{
w_0(\a_i) = -\a_{i^*},\tab \a_i\in \Pi_+ .
}
\end{Def}

\subsection{Quantum groups $\cU_q(\g)$ and $\fD_q(\g)$}\label{sec:pre:Uq}
For any finite dimensional complex semisimple Lie algebra $\g$, Drinfeld \cite{Dr} and Jimbo \cite{Ji1} associated to it a remarkable Hopf algebra $\cU_q(\g)$ known as the \emph{quantum group}, which is a certain deformation of the universal enveloping algebra. We follow the notations used in \cite{Ip7} for $\cU_q(\g)$ as well as the Drinfeld's double $\fD_q(\g)$ of its Borel part.

In the following, we assume $\g$ is of simple Dynkin type, with straightforward modification for the semisimple case.
\begin{Def} \label{qi} Let $d_i$ be the multipliers \eqref{di}. We define
\Eq{
q_i:=q^{d_i},
} which we will also write as
\Eq{
q_l&:=q,\\ 
q_s&:=\case{q^{\frac12}&\mbox{if $\g$ is of type $B, C, F$},\\q^{\frac13}&\mbox{if $\g$ is of type $G$},}
}
for the $q$ parameters corresponding to long and short roots respectively.
\end{Def}

\begin{Def} We define $\fD_q(\g)$ to be the $\C(q_s)$-algebra generated by the elements
$$\{\bE_i, \bF_i,\bK_i^{\pm1}, \bK_i'^{\pm1}\}_{i\in I}$$ 
subject to the following relations (we will omit the obvious relations involving $\bK_i^{-1}$ and ${\bK_i'}^{-1}$ below for simplicity):
\Eq{
\bK_i\bE_j&=q_i^{a_{ij}}\bE_j\bK_i, &\bK_i\bF_j&=q_i^{-a_{ij}}\bF_j\bK_i,\label{KK1}\\
\bK_i'\bE_j&=q_i^{-a_{ij}}\bE_j\bK_i', &\bK_i'\bF_j&=q_i^{a_{ij}}\bF_j\bK_i',\label{KK2}\\
\bK_i\bK_j&=\bK_j\bK_i, &\bK_i'\bK_j'&=\bK_j'\bK_i', &\bK_i\bK_j' = \bK_j'\bK_i,\label{KK3}\\
&&[\bE_i,\bF_j]&= \d_{ij}\frac{\bK_i-\bK_i'}{q_i-q_i\inv},\label{EFFE}
}
together with the \emph{Serre relations} for $i\neq j$:
\Eq{
\sum_{k=0}^{1-a_{ij}}(-1)^k\frac{[1-a_{ij}]_{q_i}!}{[1-a_{ij}-k]_{q_i}![k]_{q_i}!}\bE_i^{k}\bE_j\bE_i^{1-a_{ij}-k}&=0,\label{SerreE}\\
\sum_{k=0}^{1-a_{ij}}(-1)^k\frac{[1-a_{ij}]_{q_i}!}{[1-a_{ij}-k]_{q_i}![k]_{q_i}!}\bF_i^{k}\bF_j\bF_i^{1-a_{ij}-k}&=0,\label{SerreF}
}
where $\dis [k]_q:=\frac{q^k-q^{-k}}{q-q\inv}$ is the \emph{$q$-number}, and $\dis [n]_q!:=\prod_{k=1}^n [k]_q$ is the \emph{$q$-factorial}.
\end{Def}

The algebra $\fD_q(\g)$ is a Hopf algebra with coproduct
\Eq{
\D(\bE_i)&=1\ox \bE_i+\bE_i\ox \bK_i,&\D(\bK_i)&=\bK_i\ox \bK_i,\\
\D(\bF_i)&=\bF_i\ox 1+\bK_i'\ox \bF_i,&\D(\bK_i')&=\bK_i'\ox \bK_i',
}
We will not need the counit and antipode in this paper.

\begin{Def}
The \emph{quantum group} $\cU_q(\g)$ is defined as the quotient
\Eq{
\cU_g(\g):=\fD_q(\g)/\<\bK_i\bK_i'=1\>_{i\in I},
}
and it inherits a well-defined Hopf algebra structure from $\fD_q(\g)$. 
\end{Def}
\begin{Rem} $\fD_q(\g)$ is the \emph{Drinfeld's double} of the quantum Borel subalgebra $\cU_q(\fb)$ generated by $\bE_i$ and $\bK_i$.    
\end{Rem}

\begin{Nota}\label{notbi}
In the split real case with $q\in \C$, we require $|q|=1$ and write
\Eq{q:=e^{\pi \bi  b^2}} 
where $\bi=\sqrt{-1}$ and $0<b<1$. We assume $b^2\notin\Q$. We also write 
\Eq{b_i:= \sqrt{d_i}b\label{bi}}
such that $q_i = e^{\pi \bi b_i^2}$ as in Definition \ref{qi}. We will also write $q_s=e^{\pi \bi b_s^2}$.
\end{Nota}

\begin{Def}We define the rescaled generators by
\Eq{
\be_i &:=\left(\frac{\bi }{q_i-q_i\inv}\right)\inv \bE_i,&\bf_i &:=\left(\frac{\bi }{q_i-q_i\inv}\right)\inv \bF_i.\label{rescaleFF}
}
We also denote by $\fD_q(\g)$ the $\C(q_s)$-algebra generated by 
\Eq{\{\be_i, \bf_i, \bK_i, \bK_i'\}_{i\in I}
} and the corresponding quotient by $\cU_q(\g)$. The generators satisfy all the defining relations above except \eqref{EFFE} which is modified to
\Eq{\label{EFFE2}
[\be_i, \bf_j]=\d_{ij} (q_i-q_i\inv)(\bK_i'-\bK_i).
}
\end{Def}
\begin{Def}
We define $\cU_q(\g_\R)$ to be the real form of $\cU_q(\g)$ induced by the star structure
\Eq{\be_i^*=\be_i,\tab \bf_i^*=\bf_i,\tab \bK_i^*=\bK_i,}
with $q^*=\over{q}=q\inv$, making it a Hopf-* algebra.
\end{Def}

\subsection{Quantum torus algebra}\label{sec:pre:qtorus}
In this subsection we recall some definitions and properties concerning the quantum torus algebra and their cluster realizations.
\begin{Def} A \emph{cluster seed} is a datum 
\Eq{
\bQ=(Q, Q_0, B, D),
} where $Q$ is a finite set, $Q_0\subset Q$ is a subset called the \emph{frozen subset}, $B=(\e_{ij})_{i,j\in Q}$ a skew-symmetrizable $\frac12\Z$-valued matrix called the \emph{exchange matrix}, and $D=\mathrm{diag}(d_j)_{j\in Q}$ is a diagonal $\Q_{>0}$-matrix called the \emph{multiplier}, such that
\Eq{W:=DB=-B^TD} is a skew-symmetric $\Q$-matrix. The rank of $\bQ$ is defined to be the rank of the matrix $B$.
\end{Def}
In the following, we will consider only the case where there exists a \emph{decoration} 
\Eq{\label{deco}\eta:Q\to I
} to the root index of a simple Dynkin diagram, such that $D=\mathrm{diag}(d_{\eta(j)})_{j\in Q}$ where $(d_i)_{i\in I}$ are the multipliers given in \eqref{di}.

Let $\L_\bQ$ be a $\Z$-lattice with basis $\{\vec{e_i}\}_{i\in Q}$, and let $d=\min(d_{\eta(j)})_{j\in Q}$. Also let
\Eq{\label{wij}w_{ij}=d_i\e_{ij}=-w_{ji}.}
We define a skew symmetric $d\Z$-valued form $(-,-)$ on $\L_\bQ$ by 
\Eq{
(\vec{e_i}, \vec{e_j}):=w_{ij}.
}

\begin{Def} Let $q$ be a formal parameter. We define the \emph{quantum torus algebra}\footnote{We abuse the notation here for convenience. More precisely it should be written as $\cO_q(\cX^\bQ)$ where $\cX^\bQ$ denote the cluster Poisson tori associated to the seed $\bQ$. See also Definition \ref{regOq}.} $\cX_q^\bQ$ associated to a cluster seed $\bQ$ to be the associative algebra over $\C[q^{\pm d}]$ generated by $\{X_i^{\pm 1}\}_{i\in Q}$ subject to the relations
\Eq{\label{XiXj}
X_iX_j=q^{-2w_{ij}}X_jX_i,\tab i,j\in Q.
}
The generators $X_i\in \cX_q^\bQ$ are called the \emph{quantum cluster variables}, and they are called \emph{frozen} if $i\in Q_0$.

Alternatively, $\cX_q^\bQ$ is generated by $\{X_{\l}\}_{\l\in \L_\bQ}$ with $X_0:=1$ subject to the relations
\Eq{q^{(\l,\mu)}X_{\l}X_{\mu} = X_{\l+\mu},\tab \mu,\l\in\L_\bQ.}

Finally, we define $\bT_q^\bQ$ to be the \emph{fraction field} of the quantum torus algebra $\cX_q^\bQ$, which is well defined since $\cX_q^\bQ$ is an Ore domain.
\end{Def}

\begin{Nota}\label{xik}
Under this realization, we shall write 
\Eq{
X_i=X_{\vec{e_i}},
} and define the monomials (allowing the indices to repeat)
\Eq{X_{i_1,...,i_k}:=X_{\vec{e_{i_1}}+...+\vec{e_{i_k}}},}
or more generally for $n_1,...,n_k\in \R$,
\Eq{X_{i_1^{n_1},...,i_k^{n_k}}:=X_{n_1\vec{e_{i_1}}+...+n_k\vec{e_{i_k}}}.}

A collection of monomials is said to be \emph{independent} if the underlying vectors of the indices are linearly independent over $\R$.
\end{Nota}
\begin{Def}\label{quiver} 
We associate to each cluster seed $\bQ=(Q,Q_0,B,D)$ with decoration $\eta$ a quiver, denoted again by $\bQ$, with vertices labeled by $Q$ and adjacency matrix $C=(c_{ij})_{i,j\in Q}$, where
\Eq{
c_{ij}:=\case{\e_{ij}d_id_j\inv&\mbox{if }d_j>d_i,\\\e_{ij}&\mbox{otherwise.}}
}
An arrow $i\to j$ represents the algebraic relation 
\Eq{\label{XiXj}
X_iX_j=q_*^{-2}X_jX_i,}
where $*=\case{i&\mbox{if }d_i\geq d_j,\\j&\mbox{if }d_i\leq d_j.}$
\end{Def}
Note that $c_{ij}$ is skew-symmetric, so the quiver is well-defined. Obviously one can recover the cluster seed and the exchange matrix $B$ from the quiver and the multipliers by
\Eq{
\e_{ij}=\case{c_{ij}d_jd_i\inv&\mbox{if }d_j>d_i,\\c_{ij}&\mbox{otherwise.}}
}
\begin{Nota}\label{thick}
We will use squares to denote frozen nodes $i\in Q_0$ and circles otherwise. We will also use dashed arrows if $|c_{ij}|=\half$, which only occur between frozen nodes. For display convenience, we will represent the algebraic relations \eqref{XiXj} by thick or thin arrows (see for example Figure \ref{fig-X2fold}) to indicate the power of $q$ when we rewrite $q_*$ in terms of $q$ in the commutation relation \eqref{XiXj}. However, thickness is \emph{not} part of the data of the quiver.
\end{Nota}
   
\begin{Nota}\label{ol}
Let $\eta: \N\to I$ be a decoration. For any symbol $x_k$, $k\in \N$, we denote the \emph{rescaled symbol} by
\Eq{
\mathring{x}_k:= b_{\eta(k)} x_k,
}
where $b_i\in\R$ is defined in \eqref{bi}.
\end{Nota}

\begin{Def}\label{posrepX} A \emph{polarization} $\pi$ of the quantum torus algebra $\cX_q^\bQ$ on a Hilbert space $\cH=L^2(\R^M)$ is an assignment
\Eq{
X_i\mapsto e^{2\pi \mathring{L}_i},\tab i\in Q,
}
where $\mathring{L}_i(\mathring{u}_k, \mathring{p}_k,\ol_k)$ is a linear combination of the (rescaled) position and momentum operators  $\{u_k, p_k\}_{k=1}^M$ satisfying the Heisenberg relations
\Eq{
[u_j, p_k]=\frac{\d_{jk}}{2\pi\bi },
}
together with real parameters $\l_k \in \R$, such that they satisfy algebraically
\Eq{[\mathring{L}_i,\mathring{L}_j]=\frac{b^2w_{ij}}{2\pi\bi }.} 
Each generator $X_i$ acts as a positive essentially self-adjoint operator on $\cH$, and altogether these give a representation of $\cX_q^\bQ$ on $\cH$.
\end{Def}
\begin{Rem}
The domains of these unbounded operators are discussed in detail in e.g. \cite{FG3, Ip1, PT2}. In this paper, we will only deal with the algebraic relations among the cluster variables, and assume that their polarizations are well-defined acting on an appropriate dense subspace $\cP\subset \cH$ which contains the subspace $\cW$ of entire rapidly decreasing functions of the form
$$\cW=\{e^{-\bu^T \bA \bu+\bb\cdot \bu}P(\bu)\vert \bb\in \C^n, \bA\in M_{n\x n}(\C):\mbox{positive definite}, P:\mbox{polynomial}\}$$
which forms the core of essential self-adjointness of $\pi(X_i)$.
\end{Rem}
\begin{Nota}\label{ELnote}We will simplify notations and write
\Eq{\label{EL} 
e(L):=e^{\pi \mathring{L}}
}
for $L$ a linear combination of position, momentum operators and scalars as above, and $\mathring{L}$ rescales the corresponding variables with index $k$ by $b_{\eta(k)}$.
\end{Nota}
\begin{Def} Assume the polarization of a monomial is of the form
\Eq
{e^{\pi(\sum \a_k\mathring{u}_k + \sum \b_k\mathring{p}_k + \sum\c_k\ol_k)}.}
We call $\sum \a_k\mathring{u}_k +\sum\b_k\mathring{p}_k$ the \emph{Weyl part}, and $\sum\c_k\ol_k$ the \emph{central parameter} of the polarization. 
\end{Def}

\begin{Def} A Laurent monomial $C\in \cX_q^\bQ$ is called a \emph{central monomial} if it commutes with every cluster variable $X_i$, $i\in Q$. The center of $\cX_q^\bQ$ is generated by $|Q|-\rank(\bQ)$ independent central monomials.

A polarization is \emph{irreducible} if every central monomial acts as multiplication by scalars, i.e. their Weyl part is trivial. In this case we refer to the action $\pi(C)$ as the \emph{central character}.
\end{Def}

\begin{Lem}\label{polaru} Assume the rank of $\bQ$ is $2M$. Then there exists an irreducible polarization $\pi_\l$ of $\cX_q^\bQ$ on $\cH=L^2(\R^M)$ parametrized by the central characters, i.e. the central parameters $\l$ of the independent central monomials.

Any polarization of $\cX_q^\bQ$ on $\cH$ with the same central character is unitarily equivalent to $\pi_\l$ by an $Sp(2M)$ action on the lattice $\L_\bQ$ (known as the \emph{Weil representation} \cite{GS}).
\end{Lem}

Next we recall the notion of quantum cluster mutations.
\begin{Def}\label{qmut} Given a cluster seed $\bQ=(Q,Q_0,B, D)$ and an element $k\in Q\setminus Q_0$, a \emph{cluster mutation in direction $k$} is another seed $\mu_k^q(\bQ):=\bQ'=(Q', Q_0', B', D')$ with $Q':=Q$, $Q_0':=Q_0$, $D':=D$ and
\Eq{
\e'_{ij} &:= \case{-\e_{ij}&\mbox{if $i=k$ or $j=k$},\\ \e_{ij}+\frac{\e_{ik}|\e_{kj}|+|\e_{ik}|\e_{kj}}{2}&\mbox{otherwise}.}
}

The cluster mutation in direction $k$ induces an isomorphism $\mu_k^q:\bT_q^{\bQ'}\to\bT_q^{\bQ}$ called the \emph{quantum cluster mutation}, defined by
\Eq{
\mu_k^q(X_i'):=\case{X_k\inv&\mbox{if $i=k$},\\ \dis X_i\prod_{r=1}^{|\e_{ki}|}(1+q_i^{2r-1}X_k)&\mbox{if $i\neq k$ and $\e_{ki}\leq 0$},\\\dis X_i\prod_{r=1}^{\e_{ki}}(1+q_i^{2r-1}X_k\inv)\inv&\mbox{if $i\neq k$ and $\e_{ki}\geq 0$},}
}
where we denote by $X_i'$ the quantum cluster variables of $\bT_q^{\bQ'}$.
\end{Def}
\begin{Def}\label{regOq} We denote by $\cO_q(\cX)$ the quantum algebra of regular functions of the cluster variety $\cX$. More precisely, the elements of $\cO_q(\cX)$ consists of all elements $f\in \cX_q^\bQ$ which remain Laurent polynomials over $\C[q^{\pm d}]$ under any quantum cluster mutations. Equivalently, $\cO_q(\cX)$ is the quantum upper cluster algebra of $\cX$.

We will also refer to elements of $\cO_q(\cX)$ as \emph{universally Laurent polynomials}\footnote{This terminology usually refers to the classical $q=1$ setting.}.
\end{Def}

A useful criterion is the following Lemma.
\begin{Lem}\cite{GS}\label{stdmon} A cluster monomial $X_{i_1,...,i_s} \in \cX_q^\bQ$ is a \emph{standard monomial} if it is a sink with respect to mutable vertices, in the sense that
\Eq{
\sum_{k=1}^s \e_{i_k,j}\geq 0,\tab \forall j\in Q\setminus Q_0.
}
An element $f\in \cX_q^\bQ$ belongs to $\cO_q(\cX)$, i.e. a universally Laurent polynomial, if it can be cluster mutated to a standard monomial in some cluster seed.
\end{Lem}

Finally, we also recall that the monomial part of the quantum cluster mutation induces a change in polarization as follows.
\begin{Prop}\label{monopolar} Let $k\in Q\setminus Q_0$ and $\bQ':=\mu_k^q(\bQ)$. If $\pi$ is a polarization of $\cX_q^\bQ$, then
\Eq{\pi'(X_i):=\case{\pi(X_k)\inv&\mbox{if $i=k$,}\\
\pi(X_i)&\mbox{if $i\neq k$ and $\e_{ki}\leq 0$,}\\
\pi(X_{e_i+\e_{ki}k_i})&\mbox{if $i\neq k$ and $\e_{ki}\geq 0$}
}}
gives a polarization of $\cX_q^{\bQ'}$.
\end{Prop}

\subsection{Positive representations and cluster realization of $\cU_q(\g)$}\label{sec:pre:pos}
The family of the standard \emph{positive representations} $\cP_\l$ of $\cU_q(\g_\R)$ is constructed in \cite{FI, Ip2, Ip3} where the Chevalley generators of the quantum group are represented by positive essentially self-adjoint operators on the Hilbert space $L^2(\R^N)$ where $N=l(w_0)$. The representations are parametrized by $\l\in\R_{\geq 0}^n$.  In \cite{Ip7}, they are realized by an embedding of $\fD_q(\g)$ into a certain quantum torus algebra $\cX_q^\bD$ and taking the \emph{group-like polarization}. 

\begin{Thm}\cite{Ip7} Given a reduced word $\bi_0$ of the longest element of the Weyl group, one can construct the \emph{basic quiver} $\bD(\bi_0)$ of rank $2N+2n$ and its associated quantum torus algebra $\cX_q^{\bD(\bi_0)}$ such that 
\begin{itemize}
\item There exists an embedding of the Drinfeld's double\footnote{Throughout this paper, we will use bold letter to denote the generators of $\fD_q(\g)$ or $\cU_q(\g)$, while unbolded Roman letters denote their images in a quantum torus algebra.}
\Eq{\iota: \fD_q(\g)&\inj \cX_q^{\bD(\bi_0)}\nonumber\\
(\be_i, \bf_i, \bK_i, \bK_i') &\mapsto (e_i, f_i, K_i, K_i'),\label{dqxq}
}
where $K_i$ and $K_i'$ are cluster monomials, such that $K_iK_i'$, $i\in I$ are $n$ independent central monomials of $\cX_q^{\bD(\bi_0)}$.
\item In particular we have an embedding
\Eq{\iota: \cU_q(\g)&\inj \cX_q^{\bD(\bi_0)}/\<K_iK_i'=1\>_{i\in I}. \label{uqxq}}
\item There exists a polarization $\pi_\l$ of $\cX_q^{\bD(\bi_0)}$ where $\pi_\l(K_iK_i') = 1$ and the other $n$ independent central monomials act by $e(4\l_i) \in\R_{>0}$, such that the composition with the embedding \eqref{uqxq} coincides with the standard positive representations $\cP_\l$.
\item The representation $\cP_\l$ is irreducible, in the sense that the only operators strongly commuting with the action of the Chevalley generators are multiplication by scalars. 
\item The basic quivers associated to different reduced words $\bD(\bi_0')$ are mutation equivalent, and so the resulting expressions of the positive representations $\cP_\l$ are unitarily equivalent.
\end{itemize}
\end{Thm}
Here we say that an operator $X$ \emph{strongly commutes} with a positive operator $Y$ if $X$ commutes with the spectral projection of $Y$, or in other words, $X$ commutes with the bounded unitary operators $Y^{\bi t}$ for all $t\in\R$.

We omit the detailed construction of the basic quiver, see \cite{GS,Ip8} for more details and examples.
\begin{Def} The central elements $K_iK_i' \in \cX_q^{\bD(\bi_0)}$ are called the \emph{Cartan monomials}. A polarization $\pi$ is \emph{group-like} if $\pi(K_iK_i')=1$ for all $i\in I$.
\end{Def}

By the explicit expressions of $\cP_\l$ given in \cite{FI, Ip2}, it is parametrized by $\l=(\l_i)_{i\in I}$ with the identity decoration \eqref{deco}, where the central parameters of the polarization of $\cP_\l$ can be chosen as in Figure \ref{fig-A4}, in such a way that one side of the frozen variables carry $e(-2\l_i)$, while the variables along one half of the middle column carry $e(4\l_i)$. The polarizations of all other remaining variables have trivial central parameters.

\begin{Nota}\label{path}
The embedding of $\fD_q(\g)$ can sometimes be represented by telescopic sums in some cluster chart, described as paths on the quiver $\bD(\bi_0)$ (see Figure \ref{fig-A4} for an example). Following the convention used in \cite{Ip7}, we will use {\color{blue}\textbf{blue paths}} to denote the image of $f_i$ and $K_i'$ in $\cX_q^{\bD(\bi_0)}$ as follows. For a path $v_1\to v_2\to \cdots\to v_S$ on the quiver, the embedding is given by
\Eq{\label{patheq}
f_i&=X_{v_1}+X_{v_1,v_2}+\cdots + X_{v_1,...,v_{S-1}},\\
K_i'&=X_{v_1,v_2,...,v_S}.
}
We will use other colors to denote the embedding of $e_i$ and $K_i$ in a similar way.
\end{Nota}

\subsection{Casimir operators}\label{sec:pre:cas}
For the following definitions, we require an extension $\what{\cU}_q(\g):=\cU_q(\g)[\bK_i^{\pm\frac{1}{h}}]$ of the quantum group, where $h$ is the Coxeter number of $\g$, in order to allow fractional powers of the Cartan generators $\bK_i$.
\begin{Thm}\cite{Ip5} The center of $\what{\cU}_q(\g)$ is generated by the $n$ generalized Casimir elements
\Eq{\bC_k:=(1\ox \mathrm{Tr}|_{V_k}^q)(RR_{21}),\label{caseq}}
where 
\begin{itemize}
\item $V_k$ is the $k$-th fundamental representation of $\cU_q(\g)$, $k=1,...,n$.
\item the quantum trace $\mathrm{Tr}|_V^q$ of $x\in \cU_q(\g)$ is given by
\Eq{\mathrm{Tr}|_V^q(x):=\mathrm{Tr}|_V(xu\inv)
}
where
\Eq{\label{ueq} u:=\bK_{2\rho}\til{\bK}_{2\rho}:=\prod_i q_i^{2W_i}\prod_i q_i^{\frac{2W_i}{b_i^2}}}
and $W_i$ are the fundamental coweights \eqref{Wieq}.
\item $R\in \cU_q(\g)\what{\ox}\cU_q(\g)$ is the universal $R$-matrix.
\end{itemize}
\end{Thm}
\begin{Rem} In fact $\bC_V:=(1\ox \mathrm{Tr}|_V^q)(RR_{21})$ also lie in the center of $\what{\cU}_q(\g)$ for any finite dimensional representation $V$ of $\cU_q(\g)$. Hence one may refer to $\bC_k$ defined above as the generalized Casimirs with respect to the fundamental representations.
\end{Rem}

\begin{Ex} For $\cU_q(\sl_2)$, the Casimir element is given by
\Eq{
\bC&=\be\bf  - q\inv\bK - q \bK\inv\\
&=\bf\be-q\bK-q\inv \bK\inv\nonumber.
}
For $\cU_q(\sl_3)$, the two generalized Casimir elements are given by
{\small\Eq{
\bC_1=&\bK(q^{-2} \bK_1\bK_2+\bK_1\inv \bK_2 + q^2 \bK_1\inv \bK_2\inv -q\inv \bK_2 \be_1\bf_1-q\bK_1\inv \be_2\bf_2 + \be_{21}\bf_{12}),\label{Cas3a}\\
\bC_2=&\bK\inv(q^{2} \bK_1\inv \bK_2\inv+\bK_1 \bK_2\inv + q^{-2} \bK_1 \bK_2 -q \bK_2\inv \be_1\bf_1-q\inv \bK_1 \be_2\bf_2 + \be_{12}\bf_{21}),\label{Cas3b}
}}
where $\bK=\bK_1^{\frac{1}{3}}\bK_2^{-\frac{1}{3}}$, and
\Eq{\be_{ij}:=\frac{q^{\half}\be_j\be_i - q^{-\half}\be_i\be_j}{q-q\inv},\tab \bf_{ij}:=\frac{q^{\half}\bf_j\bf_i - q^{-\half}\bf_i\bf_j}{q-q\inv}} are the images of the Lusztig's isomorphism extended to positive generators \cite{Ip4}.
\end{Ex}

\begin{Rem} One can modify the Cartan part of the universal $R$-matrix, or more explicitly, replace $\bK_i\inv$ by $\bK_i'$ in the formula above to obtain the central elements for $\fD_q(\g)$.
\end{Rem}

Since the generalized Casimirs involve only fractional powers of the Cartan generators, their actions by the positive representations are still well defined. As the standard positive representations $\cP_\l$ are irreducible, the generalized Casimirs act as multiplication by scalars. The choice of the element $u$ in \eqref{ueq} which is compatible with the modular double, ensures that the Casimir operators are positive self-adjoint, acting with spectrum $\pi_\l(\bC_k)\geq \dim V_k$. Their actions by the standard positive representations are computed as follows.

\begin{Thm}\cite{Ip5}\label{CasAct} Let $\mu_{\cV}$ denote the weight of the weight space $\cV\subset V_k$. The generalized Casimir operators $\bC_k$ acts on $\cP_\l$ by the scalar
\Eq{\label{CasActEq}\pi_\l(\bC_k)=\sum_{\cV\subset V_k} \exp\left(-4\pi \mu_\cV(\vec{\l_\fh})\right),}
where the sum is taken over all the weight spaces of $V_k$, and 
\Eq{
\vec{\l_\fh}=\sum_{i\in I} \ol_i W_i\in\fh_\R.
}
\end{Thm}
Another way of computing the action is by the following Weyl character formula.
\begin{Cor}\cite{Ip5}\label{CasWeyl} We have
\Eq{\label{CasWeylEq}\pi_\l(\bC_k)= \frac{\sum_{w\in W} \mathrm{sgn}(w)e^{-4\pi (w_k+\rho)(w\cdot\vec{\l_\fh})}}{\prod_{\a\in\D^+}(e^{-2\pi\a(\vec{\l_\fh})}-e^{2\pi\a(\vec{\l_\fh})})},
}
where $w_k$ is the $k$-th fundamental weight, which is also the highest weight of $V_k$. The Weyl group acts on $\vec{\l_\fh}$ by \eqref{weylact}.
\end{Cor}

In type $A_n$, the fundamental representations $V_k$, $k=1,...,n$, are given by the exterior product $\L^k V$ of the standard representation $V=V_1$. Using Theorem \ref{CasAct}, we have an explicit formula for the action of $\bC_k$ in terms of the \emph{elementary symmetric polynomials} in $n+1$ variables:
\Eq{
\cE_k(\bx_0,...,\bx_n):= \sum_{i_1<i_2<\cdots<i_k} \bx_{i_1}\cdots \bx_{i_k}.
}

\begin{Prop}\cite{Ip5} Let 
\Eq{
\varpi_i:=\frac{1}{n+1}\sum_{k=1}^n k\l_k -\sum_{j=n+1-i}^n \l_j,\tab i=0,...,n.
}
Note that 
\Eq{
\varpi_0+\varpi_1+ \cdots+\varpi_n=0.} 
Then the actions of $\bC_k$ on $\cP_\l$ are given by
\Eq{\label{casact}
\pi_\l(\bC_k) = \cE_k\left(e^{4\pi b\varpi_0},..., e^{4\pi b\varpi_n}\right).
}
\end{Prop}

In particular the action of $\bC_1$ is given by
\Eq{\label{casw}
\pi_\l(\bC_1)=\sum_{i=0}^n e^{4\pi b\varpi_i}.
}
More generally, the Casimir actions can be expressed by a generating polynomial
\Eq{
p(t)&=(t+e^{4\pi b \varpi_0})\cdots (t+e^{4\pi b\varpi_n})\\
&=t^n+\pi_\l(\bC_1) t^{n-1}+\pi_\l(\bC_2) t^{n-2}+\cdots + \pi_\l(\bC_n) t + 1.
}

\begin{Lem}
The solution to the system of equation
\Eq{\label{symeq}
\case{\cE_k(\bx_0,...,\bx_n)=0,& k=1,...,n,\\
\cE_{n+1}(\bx_0,...,\bx_n):=\bx_0\cdots \bx_n=1,}
}
is given by the the set of $(n+1)$-th roots of $-1$, in other words
\Eq{
\{\bx_0,...,\bx_n\} = \{e^{\frac{(n-2k)}{n+1}\pi \bi }\}_{k=0,...,n}.
}
\end{Lem}
\begin{proof}
By the generating polynomial, the system of equations is equivalent to identifying
$$t^{n+1}+1 = (t+\bx_0)\cdots (t+\bx_n),$$
hence $\{\bx_i\}$ is any permutation of the $(n+1)$-th root of $-1$.
\end{proof}

\begin{Def}\label{gensol} We call an ordered scalars $(\l_1,...,\l_n)\in \C^n$ a \emph{general solution} to \eqref{symeq} if for any $k=0,...,n$, the roots $\bx_k=e^{4\pi b\varpi_k}$ and lying within the range 
\Eq{
-\pi \leq 4\pi \bi b\varpi_k\leq \pi.}
In particular, there are $(n+1)!$ of them, each of which corresponds to a distinct permutation of the $(n+1)$-th roots of $-1$.

The solution corresponding to the following choice of $(n+1)$-th roots of $-1$,
\Eq{
4\pi b\varpi_k = \frac{n-2k}{n+1}\pi\bi ,\tab k=0,...,n
}
is given by
\Eq{
\l_1=\cdots = \l_n = \frac{\bi}{2(n+1)b},
}
which is refered as the \emph{standard solution} of \eqref{symeq}.
\end{Def}

As a direct consequence,
\begin{Cor}
As a complex function of $\l_1,...,\l_n$,
\Eq{
\pi_\l(\bC_k)=0
} if we substitute $(\l_1,...,\l_n)$ to be any general solution of \eqref{symeq}.
\end{Cor}
There is also a ``modular double counterpart" which will be useful later. 

Recall the \emph{$q$-binomial coefficients} defined by
\Eq{\bin{n\\k}_q:=\frac{[n]_q!}{[k]_q![n-k]_q!}.}
\begin{Cor}
As a complex function of $\l_1,...,\l_n$, if we substitute $(b^2\l_1,...,b^2\l_n)$ into \eqref{casw}, where $(\l_1,...,\l_n)$ is any general solution of \eqref{symeq}, we obtain
\Eq{
\pi_\l(\bC_k)=\bin{n+1\\k}_{q^{\frac{1}{n+1}}},\tab k=1,...,n.}
\end{Cor}
\begin{proof}
By this substitution, the following sets for $k=0,...,n$ become
$$\{e^{4\pi b\varpi_k}\}=\{e^{\frac{(n-2k)}{n+1}\pi \bi b^2}\}=\{q^{\frac{n-2k}{n+1}}\}.$$
Hence the generating polynomial of the Casimirs becomes
\Eqn{
p(t)&=(t+e^{4\pi b \varpi_0})\cdots (t+e^{4\pi b\varpi_n})\\
&=(t+q^{\frac{n}{n+1}})(t+q^{\frac{n-2}{n+1}})\cdots (t+q^{-\frac{n}{n+1}})\\
&=\sum_{k=0}^{n+1} \bin{n+1\\k}_{q^{\frac{1}{n+1}}} t^{n+1-k}
}
as required.
\end{proof}
\section{Motivation: $\cU_q(\sl(2,\R))$}\label{sec:sl2}
In this section, we explain the motivation of the construction of the degenerate positive representations from the simplest case of $\cU_q(\sl(2,\R))$ and discuss some of its properties.
\subsection{Cluster realization of $\cU_q(\sl(2,\R))$}\label{sec:sl2:pos}
Recall that we have the cluster realization of the Drinfeld's double $\fD_q(\sl_2)$ into the quantum torus algebra $\cX_q^{\mathrm{std}}$ represented by the quiver in Figure \ref{fig-A1} with all multipliers $d_i=1$,
\begin{figure}[htb!]
\centering
\begin{tikzpicture}[every node/.style={inner sep=0, minimum size=0.5cm, thick, fill=white, draw}, x=2cm, y=1cm]
\node (1) at (0,1) {$1$};
\node (2) at (1,0) [circle]{$2$};
\node (3) at (2,1) {$3$};
\node (4) at (1,2)  [circle]{$4$};
\drawpath{1,2,3}{blue}
\drawpath{3,4,1}{red}
\end{tikzpicture}
\caption{The quiver for $\cX_q^{\mathrm{std}}$.}\label{fig-A1}
\end{figure}
such that we have an embedding of the Drinfeld's double $\fD_q(\sl_2)=\<\be,\bf, \bK,\bK'\>$ given by\footnote{The choice of indices follows the convention used in the higher rank, cf. Section \ref{sec:An}.}
\Eq{\label{sl2emb}
\be&\mapsto X_3+X_{3,4}\nonumber\\
\bf&\mapsto X_1+X_{1,2}\\
\bK&\mapsto X_{3,4,1}\nonumber\\
\bK'&\mapsto X_{1,2,3}\nonumber
}
using Notation \ref{xik}. The Casimir element of $\fD_q(\sl_2)$ is given by
\Eq{
\bC&:=\bf\be-q\bK-q\inv \bK'\\
&\mapsto X_{1,3}+X_{1,2,3,4}
}
which lies in the center of $\cX_q^{\mathrm{std}}$.

Let us perform a quantum cluster mutation at vertex $2$, such that we obtain the quiver (with minor rearrangements) in Figure \ref{fig-A1sym} corresponding to the quantum torus algebra denoted by $\cX_q^{\mathrm{sym}}$. 

\begin{figure}[htb!]
\centering
\begin{tikzpicture}[every node/.style={inner sep=0, minimum size=0.5cm, thick, fill=white, draw}, x=2cm, y=1cm]
\node (1) at (0,0) {$1$};
\node (2) at (1,1) [circle]{$2$};
\node (3) at (2,0) {$3$};
\node (4) at (1,2)  [circle]{$4$};
\drawpath{1,3}{blue}
\drawpath{3,4,1}{red}
\drawpath{3,2,1}{red}
\end{tikzpicture}
\caption{The quiver for $\cX_q^{\mathrm{sym}}$.}\label{fig-A1sym}
\end{figure}

In this cluster chart, the embedding of the Drinfeld's double becomes
\Eq{\label{sl2embsym}
\be&\mapsto X_3+X_{3,2}+X_{3,4}+X_{3,2,4}\nonumber\\
\bf&\mapsto X_1\\
\bK&\mapsto X_{3,2,4,1}\nonumber\\
\bK'&\mapsto X_{1,3}\nonumber
}
and the Casimir element is given by
\Eq{
\bC&\mapsto X_{1,2,3}+X_{1,3,4}\nonumber\\
&=X_{1,3}(X_2+X_4),
}
where we note that both $X_2$ and $X_4$ commute with $X_{1,3}$ in this cluster chart.

Now the main observation is that the variables $X_2$ and $X_4$ are symmetric in this quiver, which allows us to perform a certain kind of ``folding''. Another observation is that only the upper Borel generators $\be$ and $\bK$ depend on these two variables. In fact, we can rewrite $\be$ as
\Eq{\label{sl2e}
\be&\mapsto X_3+qX_3(X_2+X_4)+X_{3,2,4}\nonumber\\
&=X_3+X_1\inv X_{1,3}(X_2+X_4)+X_{3,2,4}\\
&=X_3+X_1\inv \bC +X_{3,2,4}\nonumber.
}
Since $\bC$ lies in the center of $\cX_q^{\mathrm{sym}}$, the commutation relations among the Chevalley generators remain invariant if we set $$\bC=0.$$ 
In other words, we have a homomorphic image of $\fD_q(\sl_2)$ given by
\Eq{
\be&\mapsto X_3+X_{3,2,4}\nonumber\\
\bf&\mapsto X_1\\
\bK&\mapsto X_{3,2,4,1}\nonumber\\
\bK'&\mapsto X_{1,3}\nonumber
}
or equivalently, an embedding of $\fD_q(\sl_2)/\<\bC=0\>$ into $\cX_q^{\mathrm{sym}}$.

Since $X_2$ and $X_4$ are symmetric in these expressions, we further identify them by means of \emph{folding}. 
\begin{Def} We define the \emph{symmetric folding} $\cX_q^{0}$ to be the skew-symmetrizable quantum torus algebra by combining the variables $X_2$ and $X_4$ of $\cX_q^{\mathrm{sym}}$.  We write $\cO_q(\cX^0)$ to denote the quantum algebra of regular functions of the corresponding cluster variety. More precisely, $\cX_q^0$ is generated by
\Eq{
X_1^0:= X_1, \tab X_2^0 := X_{2,4},\tab X_3^0:=X_3
}
with the new multipliers defined to be $(d_1,d_2,d_3):=(1,2,1)$.
\end{Def}
It follows that the $q$-commutation relations of the cluster variables of $\cX_q^0$ are given by
\Eq{
X_{1}^0X_{2}^0&=q^4 X_{2}^0X_{1}^0\nonumber\\
X_{2}^0X_{3}^0&=q^4 X_{3}^0X_{2}^0\\
X_{3}^0X_{1}^0&=q^2 X_{1}^0X_{3}^0\nonumber
}
and they can be presented by the quiver as in Figure \ref{fig-X2fold} according to Notation \ref{thick}.

\begin{figure}[htb!]
\centering
\begin{tikzpicture}[every node/.style={inner sep=0, minimum size=0.5cm, thick, fill=white, draw}, x=2cm, y=1cm]
\node (1) at (0,0) {$1$};
\node (2) at (1,2) [circle]{$2$};
\node (3) at (2,0) {$3$};
\drawpath{1,3}{blue}
\drawpath{3,2,1}{red,vthick}
\end{tikzpicture}
\caption{The quiver for the symmetric folding $\cX_q^0$.}\label{fig-X2fold}
\end{figure}
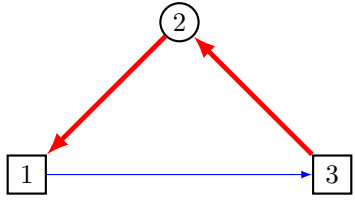
\begin{Thm} There is an embedding from $\fD_q(\sl_2)/\<\bC=0\>$ to the symmetric folding $\cX_q^0$. Furthermore, the Chevalley generators $\{\be, \bf, \bK, \bK'\}$ are realized as universally Laurent polynomials in $\cO_q(\cX^0)$.
\end{Thm}

This is the main construction that allows us to generalize the zero Casimir representation to higher rank in the next section.

\begin{proof}

The homomorphism of $\fD_q(\sl_2)$ to this quantum torus algebra is given explicitly by
\Eqn{
\be&\mapsto X_3^0+X_{3,2}^0\\
\bf&\mapsto X_1^0\\
\bK&\mapsto X_{3,2,1}^0\\
\bK'&\mapsto X_{1,3}^0.
}
One checks explicitly that $$\bC=\be\bf  - q\inv\bK - q \bK\inv \mapsto 0.$$

The image of $\bf, \bK, \bK'$ are standard monomials in this seed by Lemma \ref{stdmon}: the only mutable variable $X_2^0$ commutes with $K,K'$, while $e X_2^0=q^4 X_2^0 e$. For the generator $\be$, we mutate at vertex $2$ (taking into account the fact that the multiplier $d_2=2$) and obtain the new cluster chart.
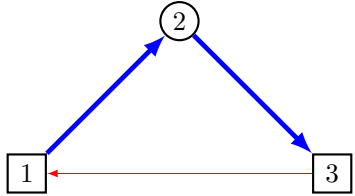
\begin{figure}[htb!]
\centering
\begin{tikzpicture}[every node/.style={inner sep=0, minimum size=0.5cm, thick, fill=white, draw}, x=2cm, y=1cm]
\node (1) at (0,0) {$1$};
\node (2) at (1,2) [circle]{$2$};
\node (3) at (2,0) {$3$};
\drawpath{1,2, 3}{blue, vthick}
\drawpath{3,1}{red}
\end{tikzpicture}
\caption{The quiver for $\mu_2^q(\cX_q^0)$.}\label{fig-A1mut}
\end{figure}
The generator transforms as $$\be\mapsto {X_3^0}'$$ in the new cluster chart in which it is again a standard monomial, as it $q^4$-commutes with the mutable variable ${X_2^0}'$. Hence the conclusion follows.
\end{proof}

\subsection{Positive representation with zero Casimir}\label{sec:sl2:cas}
The standard positive representations $\cP_\l$ of $\cU_q(\sl(2,\R))$ and its modular double, parametrized by $\l\geq0$, were first studied in \cite{Fa1, PT1, PT2} in the context of quantum Liouville theory. They can be obtained by choosing a polarization of the quantum torus algebra $\cX_q^{\mathrm{std}}$ such that $\iota(\bK\bK')=1$ (i.e. identifying $K'=K\inv$) and the central monomial $X_{2,4}$ acts by the central parameter $e(4\l)=e^{4\pi b\l}$.

Specifically, we choose the polarization
\Eq{
\pi_\l(X_1)&=e(u-2p)\nonumber\\
\pi_\l(X_2)&=e(-2u)\\
\pi_\l(X_3)&=e(-u+2p-2\l)\nonumber\\
\pi_\l(X_4)&=e(2u+4\l)\nonumber
}
where $p=\frac{1}{2\pi \bi}\frac{d}{du}$, such that by \eqref{sl2emb}, the Chevalley generators of $\cU_q(\sl_2)$ act on the Hilbert space $L^2(\R, du)$ as positive self-adjoint operator as
\Eq{
\pi_\l(\be)&:= e^{\pi b(u+2p+2\l)}+e^{\pi b(-u+2p-2\l)}\\
\pi_\l(\bf)&:= e^{\pi b(u-2p)}+e^{\pi b(-u-2p)}\nonumber\\
\pi_\l(\bK)&:=e^{\pi b(2u+2\l)}\nonumber.
}

\begin{figure}[htb!]
\centering
\begin{tikzpicture}[every node/.style={inner sep=0, minimum size=0.5cm, thick, fill=white, draw}, x=2cm, y=1cm]
\node (1) at (0,1) {$1$};
\node(2) at (1,0) [circle]{$2$};
\node [label={[yshift=-0.4em] \tiny $-2\l$}](3) at (2,1) {$3$};
\node [label={[yshift=-0.4em] \tiny $4\l$}] (4) at (1,2)  [circle]{$4$};
\drawpath{1,2,3}{blue}
\drawpath{3,4,1}{red}
\end{tikzpicture}
\caption{The polarization of $\cX_q^{\mathrm{std}}$ with central parameters.}
\end{figure}
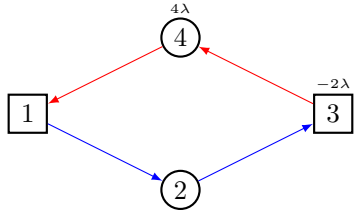

Since the representation $\cP_\l$ is irreducible, the Casimir operator (with $K'=K\inv$)
\Eqn{
\bC&:=\be\bf  - q\inv\bK - q \bK\inv \\
&\mapsto X_{1,3}+X_{1,2,3,4}
}
acts as multiplication by scalar, given by
\Eq{
\pi_\l(\bC)=e^{2\pi b\l}+e^{-2\pi b \l}.
}

Now we observe that the pair $(\bK,\bf)$ forms a quantum torus. Hence, one can unitarily transform $\cP_\l$ in such a way that $(\bK,\bf)$ acts by the canonical representation
\Eq{
\pi_\l(\bf)=e^{-2\pi b p},\tab \pi_\l(\bK)=e^{2\pi b u}.
}

This can be done by polarizing the quiver obtained from cluster mutation at vertex 2 according to Proposition \ref{monopolar}.
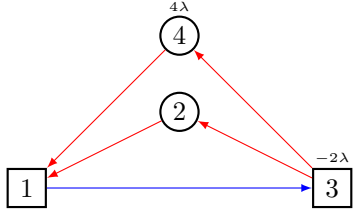
\begin{figure}[htb!]
\centering
\begin{tikzpicture}[every node/.style={inner sep=0, minimum size=0.5cm, thick, fill=white, draw}, x=2cm, y=1cm]
\node (1) at (0,0) {$1$};
\node (2) at (1,1) [circle]{$2$};
\node [label={[yshift=-0.4em] \tiny $-2\l$}] (3) at (2,0) {$3$};
\node [label={[yshift=-0.4em] \tiny $4\l$}] (4) at (1,2)  [circle]{$4$};
\drawpath{1,3}{blue}
\drawpath{3,4,1}{red}
\drawpath{3,2,1}{red}
\end{tikzpicture}
\caption{The polarization of $\cX_q^{\mathrm{sym}}$ with central parameters.}
\end{figure}

Recall that any two polarizations are unitarily equivalent if the central characters are the same. Here the central monomials are given by $\pi_\l(X_4X_2\inv) = e(4\l)$ as well as $\iota(\bK \bK')=X_{1^2,2,3^2,4}=1$. Hence there exists a change of variables such that the polarization is equivalent to that of Figure \ref{fig-A1polar}.
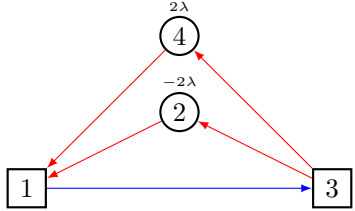
\begin{figure}[htb!]
\centering
\begin{tikzpicture}[every node/.style={inner sep=0, minimum size=0.5cm, thick, fill=white, draw}, x=2cm, y=1cm]
\node  (1) at (0,0) {$1$};
\node[label={[yshift=-0.4em] \tiny $-2\l$}]  (2) at (1,1) [circle]{$2$};
\node (3) at (2,0) {$3$};
\node[label={[yshift=-0.4em] \tiny $2\l$}] (4) at (1,2)  [circle]{$4$};
\drawpath{1,3}{blue}
\drawpath{3,4,1}{red}
\drawpath{3,2,1}{red}
\end{tikzpicture}
\caption{Another polarization of $\cX_q^{\mathrm{sym}}$.}\label{fig-A1polar}
\end{figure}

The equivalence $\pi_\l(\cX_q^{\mathrm{std}})\simeq \pi_\l(\cX_q^{\mathrm{sym}})$ can be realized explicitly by the following unitary transformation, which is given by a multiplication by the \emph{quantum dilogarithm function} $g_b(z)$ (see \cite{Ip1} for details) followed by a change of variables:
\Eq{
\Psi:= (u\mapsto u-\l) \circ e^{\pi i u^2} g_b^*(e^{2\pi b u}).
}

The new polarization together with the change of variables is given by
\Eq{
\pi_\l(X_1)&=e(-2p)\nonumber\\
\pi_\l(X_2)&=e(2u-2\l)\\
\pi_\l(X_3)&=e(-2u+2p)\nonumber\\
\pi_\l(X_4)&=e(2u+2\l)\nonumber.
}
With this new polarization applied to \eqref{sl2embsym}, we obtain a representation unitary equivalent to $\cP_\l$ given by
\Eq{\label{sl2posrep}
\pi_\l(\be)&=e^{2\pi b(u+p)}+e^{2\pi b(-u+p)}+e^{2\pi b(\l+p)}+e^{2\pi b(-\l+p)}\nonumber\\
&=e^{2\pi b(u+p)}+e^{2\pi b(-u+p)}+\pi_\l(\bC)e^{2\pi bp}\\
\pi_\l(\bf)&=e^{-2\pi b p}\nonumber\\
\pi_\l(\bK)&=e^{2\pi b u}\nonumber\\
\pi_\l(\bC)&=e^{2\pi b \l}+e^{-2\pi b \l}\nonumber.
}
We note that this representation matches with the formal relation
\Eq{
\be=(\bC+q\inv\bK+q \bK\inv)\bf\inv.
}

If we now set $\bC=0$, the last two monomials of $\pi_\l(\be)$ vanish. This gives us a new representation of $\cU_q(\sl(2,\R))$ acting by positive operators that does not belong to the standard family.
\begin{Thm} We have an irreducible representation $\cP^0$ of $\cU_q(\sl(2,\R))$ by positive operators on $L^2(\R)$, given by
\Eq{
\pi^0(\be)&=e^{2\pi b(u+p)}+e^{2\pi b(-u+p)}\\
\pi^0(\bf)&=e^{-2\pi b p}\nonumber\\
\pi^0(\bK)&=e^{2\pi b u}\nonumber
}
such that the Casimir operator $\bC$ acts as zero on $L^2(\R)$.
\end{Thm}
\begin{Rem} In fact, the representation $\cP^0$ is unitarily equivalent to the integrable representation of type $(I)_{+,+,c=0}$ classified by Schm\"udgen \cite{Sch}.
\end{Rem}

\subsection{Analytic continuation and modular double}\label{sec:sl2:mod}
From another point of view, what we have done is replacing the action of the Casimir operator by zero:
\Eq{
\pi_\l(\bC)=e^{2\pi b\l}+e^{-2\pi b \l}\leadsto 0,
}
which amounts to ``solving for $\l$'' and we obtain
\Eq{
\l=\pm\frac{\bi}{4b}+\frac{\bi}{b}n,\tab n\in \Z.
}
Informally we have analytically continued the weight parameter $\l$ to take complex values. Moreover, if the representation of Faddeev's modular double is also taken into account, the value of $\l$ can be uniquely determined up to a sign. Hence, we shall only consider the \emph{general solution} of \eqref{symeq} in the sense of Definition \ref{gensol}.

We recall the following useful Lemma.
\begin{Lem}\cite{PT2}\label{qbin} Let $X,Y$ be positive self-adjoint operators on $L^2(\R)$ such that $XY=q^2 YX$. Then $X+Y$ is also positive self-adjoint and
\Eq{(X+Y)^{\frac{1}{b^2}} = X^{\frac{1}{b^2}}+Y^{\frac{1}{b^2}}.}
\end{Lem}
Here again the composition of the unbounded operators is taken in the integrable sense \cite{Ip1, PT2} and the powers of positive operators are defined by functional calculus.
\begin{Thm}\label{sl2mod} The transcendental relations are satisfied for $\l=\pm\frac{\bi}{4b}$. In other words, the positive operators
$$\pi^0(\be)^{\frac{1}{b^2}},\tab \pi^0(\bf)^{\frac{1}{b^2}},\tab \pi^0(\bK)^{\frac{1}{b^2}}$$
on $L^2(\R)$ define a positive representation of the modular double counterpart $\cU_{\til{q}}(\sl(2,\R))$.
\end{Thm}
\begin{proof}
Since $\pi^0(\be)$ and $\pi^0(\bK)$ are monomials, the transcendental relations are trivial:
\Eqn{
\pi^0(\bf)^{\frac{1}{b^2}}&=e^{-2\pi b\inv p}\\
\pi^0(\bK)^{\frac{1}{b^2}}&=e^{2\pi b\inv u}.
}
For $\pi^0(\be)$, since the two monomial terms are $q^4$ commuting, we compute using Lemma \ref{qbin} and the binomial formula to get
\Eqn{
\pi^0(\be)^{\frac{1}{b^2}}&=(e^{2\pi b(-u+p)}+e^{2\pi b(u+p)})^{\frac{1}{b^2}}\\
&=\left((e^{2\pi b(-u+p)}+e^{2\pi b(u+p)})^{\frac{1}{2b^2}}\right)^2\\
&=\left((e^{2\pi b\inv(-u+p)}+e^{2\pi b\inv(u+p)})\right)^2\\
&=e^{2\pi b\inv(-u+p)}+e^{2\pi b\inv(u+p)}+(\til{q}^{\frac12}+\til{q}^{-\frac12})e^{2\pi b\inv p}
}
where $\til{q}=e^{\pi \bi b^{-2}}$. Comparing with the explicit expression of $\cP_\l$, we see that this is indeed a representation of $\cU_{\til{q}}(\sl(2,\R))$ with the new Casimir 
$$\til{\bC}=\til{q}^{\frac12}+\til{q}^{-\frac12}=e^{2\pi b\inv \l}+e^{-2\pi b\inv \l}$$
corresponding to setting the weight parameter to $\l=\pm\frac{\bi}{4b}$.
\end{proof}

\subsection{Tensor product decomposition}\label{sec:sl2:tensor}

The main result of this section is the decomposition of the tensor product of $\cP^0$ with itself. It turns out that the tensor product lies in a certain ``closure'' of the standard positive representations in the form of a direct integral, with respect to a special Plancherel measure that is different from the usual one \cite{Ip1,PT2} by a factor of $\sqrt{2}$.

\begin{Thm}\label{sl2casspec}
We have the following decomposition
\Eq{
\cP^0\ox \cP^0 \simeq \int_{\R_{\geq 0}}^\o+ \cP_\l d\mu^0(\l),}
where $\cP_\l$ is the standard positive representation parametrized by $\l\geq 0$ and the rescaled Plancherel measure is given by 
\Eq{
d\mu^0(\l) = 2\sqrt{2}\sinh(4\pi b \l)\sinh(\pi b\inv \l)d\l.
}
\end{Thm}
\begin{proof}
The tensor product representation can be realized by the amalgamation of two copies of the quiver corresponding to $\cX_q^0$  along one side of the frozen vertex (here denoted by vertex $2$) as shown in Figure \ref{fig-tensor}, such that the actions of the generators by coproduct correspond to concatenation of their telescopic paths.

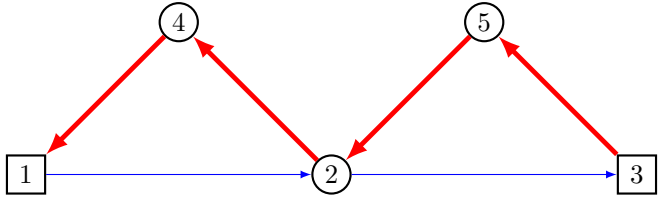
\begin{figure}[htb!]
\centering
\begin{tikzpicture}[every node/.style={inner sep=0, minimum size=0.5cm, thick, fill=white, draw}, x=1cm, y=1cm]
\node (1) at (0,0) {$1$};
\node (2) at (4,0) [circle]{$2$};
\node (3) at (8,0) {$3$};
\node (4) at (2,2) [circle]{$4$};
\node (5) at (6,2) [circle]{$5$};
\drawpath{1,2,3}{blue}
\drawpath{3,5,2,4,1}{vthick, red}
\end{tikzpicture}
\caption{The quiver for $\cX_q^0\ox \cX_q^0$.}\label{fig-tensor}
\end{figure}

The polarization of the quiver for $\cU_q(\sl(2,\R))$ is obtained by multiplying the corresponding polarizations $(u_1,p_1), (u_2,p_2)$ of the two copies of the quiver acting on $L^2(\R^2)$. Note that we do not have any central parameters.

Following \cite{NT}, it suffices to decompose the Casimir operator in the tensor product representation
\Eq{
\D(\bC)&=\D(\bf)\D(\be)-q\D(\bK)-q\inv \D(\bK\inv)\\
&=\bK\ox \bC+\bC\ox \bK\inv + \be\ox \bf+\bf \bK\ox \bK\inv\be+(q+q\inv)\bK\ox \bK\inv.\label{DC}
}

By mutating at vertex 2 (which becomes unfrozen after amalgamation), we obtain the quiver in Figure \ref{fig-tensor2}
\begin{figure}[htb!]
\centering
\begin{tikzpicture}[every node/.style={inner sep=0, minimum size=0.5cm, thick, fill=white, draw}, x=1cm, y=1cm]
\node (1) at (0,0) {$1$};
\node (2) at (2,2) [circle]{$2$};
\node (3) at (4,0) {$3$};
\node (4) at (0,4) [circle]{$4$};
\node (5) at (4,4) [circle]{$5$};
\drawpath{3,2,1}{red}
\drawpath{1,3}{blue}
\drawpath{4,2,5}{vthick, red}
\path (5.north west) edge[->, vthick, red] (4.north east);
\path (5.south west) edge[->, vthick, red] (4.south east);
\end{tikzpicture}
\caption{The quiver for $\mu_2(\cX_q^0\ox \cX_q^0)$.}\label{fig-tensor2}
\end{figure}
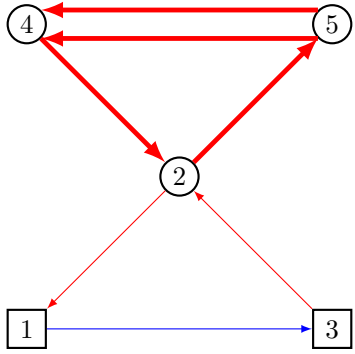
together with the embedding
\Eqn{
\D(\be)&\mapsto X_3+X_{3,2}+X_{3,2,5}+X_{3,5,2,4}+X_{3,5,2^2,4}\\
&=X_3+X_1\inv\D(\bC)+X_{3,5,2^2,4}\\
\D(\bf)&\mapsto X_1\\
\D(\bK)&\mapsto X_{1,2^2,3,4,5}\\
\D(\bK')&\mapsto X_{1,3}\\
\D(\bC)&\mapsto X_{1,2,3}+X_{1,2,3,5}+X_{1,2,3,4,5}.
}

Keeping track of the change in polarization according to Proposition \ref{monopolar}, we obtain the following positive representation of $\cU_q(\sl(2,\R))$ on $L^2(\R^2)$ given by
\Eqn{
\pi_{\cP^0\ox \cP^0}(\be)&=e^{2\pi b(2p_2+p_1)}+e^{2\pi b(u_2+p_1)}+e^{2\pi b(-u_2+p_1)}+e^{2\pi b(-u_1+p_1)}+e^{2\pi b(u_1+p_1)}\\
\pi_{\cP^0\ox \cP^0}(\bf)&=e^{-2\pi bp_1}\\
\pi_{\cP^0\ox \cP^0}(\bK)&=e^{2\pi b u_1}\\
\pi_{\cP^0\ox \cP^0}(\bC)&=e^{4\pi b p_2}+e^{2\pi b u_2}+e^{-2\pi b u_2}.
}

Upon a rescaling of the variables $\case{u_2\mapsto \sqrt{2}u_2\\p_2\mapsto \frac{1}{\sqrt{2}}p_2}$, we see that the Casimir operator is equivalent to Kashaev's geodesic length operator \cite{Ka1}
\Eq{
\bL:=e^{2\pi b\bp}+e^{2\pi b\bx}+e^{-2\pi b\bx},
}
but with the quantum parameter $b\mapsto \sqrt{2}b$ instead. The spectral decomposition of $\bL$ is well-known \cite{Ka2} with simple spectrum 
$$\bL\simeq \int_{\R_{\geq 0}}^\o+ \pi_\l(\bC) d\mu(\l),$$
where $$\pi_\l(\bC)=e^{2\pi b \l}+e^{-2\pi b \l}$$
is the action of the Casimir on the standard positive representation and the Plancherel measure is given by
$$d\mu(\l)=4\sinh(2\pi b\l)\sinh(2\pi b\inv \l)d\l.$$
As a consequence we have the decomposition $\pi_{\cP^0\ox \cP^0}(\bC)$ into simple spectrum 
$$\pi_{\cP^0\ox \cP^0}(\bC)\simeq \int_{\R_{\geq 0}}^\o+ \pi_\l(\bC) d\mu^0(\l)$$ with the corresponding rescaled measure.
\end{proof}

\subsection{Remark on other Casimir values}\label{sec:sl2:rem}
It is natural to ask whether we still obtain a sensible representation by setting the Casimir operator to its other eigenvalues. For example, if we set $\bC=1$, we obtain from \eqref{sl2embsym}--\eqref{sl2e} formally
\Eq{
\be&\mapsto X_3+X_1\inv+X_{3,2,4}\\
\bf&\mapsto X_1\nonumber\\
\bK&\mapsto X_{3,2,4,1}\nonumber\\
\bK'&\mapsto X_{1,3}\nonumber
}
in the image of the quantum torus algebra $\cX_q^{\mathrm{sym}}$. In particular, a polarization of $\cX_q^{\mathrm{sym}}$ provides us with a representation by symmetric operators:
\Eq{
\be&=e^{2\pi b(-u+p)}+e^{2\pi b p}+e^{2\pi b(u+p)}\\
\bf&=e^{-2\pi b p}\nonumber\\
\bK&=e^{2\pi b u}\nonumber\\
\bK'&=e^{-2\pi b u}\nonumber
}
However, we can check that the image of $\bf$ in $\cX_q^{\mathrm{sym}}$ is no a longer universally Laurent polynomial, so that $\bf\notin \cO_q(\cX)$. In particular, one cannot mutate $\bf$ in such a way that it gets transformed into a cluster monomial, unlike the usual cases.

On the other hand, if we set 
\Eq{
\bC=q^\half+q^{-\half}=2\cos\frac{\pi b^2}{2},} then from the calculation in the proof of Theorem \ref{sl2mod} with $\til{q}$ replaced by $q$, the modular duality ensures that the resulting representation is indeed \emph{regular}, in the sense that the Chevalley generators are universally Laurent polynomials lying in $\cO_q(\cX^0)$, provided that we invert the multipliers of $\cX_q^0$ such that $d_2=\half$. 

In Section \ref{sec:class}, we will discuss the general classification of such ``regular" positive representations.

\begin{Rem}\label{b=1} If we further set $b=1$, then both $\bC$ and $\til{\bC}$ act by zero, while the transformation $g_b(z)$ is independent of $\l$ and well-defined as $g_1(z)\sim \mathrm{Li}_2(z)$. This curious connection is worth investigating in future works.
\end{Rem}

We end this section by a simple observation. Let $\cP_\l^{c}$ denote the representation of $\cU_q(\sl(2,\R))$ obtained from \eqref{sl2posrep} where $\bC$ acts by a positive scalar $c\geq 0$. 
\begin{Prop} Let $c,c'\geq 0$. Then $\D(\bC)$ is a symmetric operator with positive spectrum $\geq 2$ acting on $\cP_{\l}^{c}\ox \cP_{\l'}^{c'}$.
\end{Prop}
\begin{proof} It follows from \eqref{DC} that $\D(\bC)$ is of the form
$$\D(\bC)=\bK\ox \bC+\bC\ox \bK\inv + \D(\bC_0)$$
where $\D(\bC_0)$ is the Casimir for $\cP^0\ox \cP^0$ which we know has positive spectrum $\geq 2$ by Theorem \ref{sl2casspec}. Since the action of $\bK\ox \bC+\bC\ox \bK\inv$ is evidently positive symmetric, we have our conclusion.
\end{proof}
\section{Positive representations at zero Casimirs in type $A_n$}\label{sec:An}
Let us denote by $\bi_{A_n}$ the standard reduced word of the longest element of the Weyl group of $\sl_{n+1}$, given by
\Eq{
\bi_{A_n}=1\;21\;321\;4321\cdots n\;n-1\;\cdots\; 321.
}

\subsection{Symmetric quiver $\cX_q^{\mathrm{sym}}$}\label{sec:An:sym}
In type $A_n$, recall that we have an embedding of the Drinfeld's double $\fD_q(\sl_{n+1})$ into the quantum torus algebra $\cX_q^{\mathrm{std}}$ given by the quiver in Figure \ref{fig-A4} corresponding to $\bi_{A_n}$ associated to a triangulation of a once punctured disk, where the paths on the quiver gives the embedding of the Chevalley generators as telescoping sums, as explained in Notation \ref{path}.

\begin{figure}[htb!]
\centering
\begin{tikzpicture}[every node/.style={inner sep=0, minimum size=0.5cm, thick}, x=1.2cm, y=0.5cm]
\foreach \y [count=\c from 1] in{22,17,10,1}
\foreach \x in{1,...,\c}{
	\pgfmathtruncatemacro{\ind}{\x+\y-1}
	\pgfmathtruncatemacro{\indop}{\y+2*\c-\x+1}
	\pgfmathtruncatemacro{\r}{5-\c}
	\pgfmathtruncatemacro{\ex}{\x-1-\c}
	\pgfmathtruncatemacro{\exop}{\c+1-\x}
	\ifthenelse{\x=1}{
		\node(\ind) at (\x-1,2*\c-1-\x)[draw] {\tiny$\ind$};
		\node [label={[xshift=1em, yshift=-0.5em] \tiny $-2\l_\r$}](\indop) at (9-\x,2*\c-1-\x)[draw] {\tiny$\indop$};
	}{
		\node(\ind) at (\x-1,2*\c-1-\x)[draw, circle]{\tiny$\ind$};
		\node(\indop) at (9-\x,2*\c-1-\x)[draw, circle]{\tiny$\indop$};
	}
}
\foreach \x[count=\c from 1] in{5,13,19,23}{
\node (\x) at (4, 4-\c*2)[draw, circle]{\tiny$\x$};
}
\foreach \x[count=\c from 1] in{25,26,27,28}{
\node[label={[xshift=0em, yshift=-0.4em] \tiny $4\l_\c$}](\x) at (4, 12-\c*2)[draw, circle]{\tiny$\x$};
}

\foreach \x[count=\c from 1]in{22,17,10,1}{
\pgfmathtruncatemacro{\ee}{\x+2*\c}{
	\drawpath{\x,...,\ee}{blue}
}
}
\drawpath{9,25,1}{, red}
\drawpath{16,8,26,2,10}{, red}
\drawpath{21,15,7,27,3,11,17}{, red}
\drawpath{24,20,14,6,28,4,12,18,22}{, red}
\drawpath{4,27,6,13,4}{}
\drawpath{3,26,7,14,19,12,3}{}
\drawpath{2,25,8,15,20,23,18,11,2}{}
\drawpath{22,17,10,1}{dashed, }
\drawpath{9,16,21,24}{dashed, }
\end{tikzpicture}
\caption{$A_4$-quiver, with the $\be_i$-paths colored in red and $\bf_i$-paths colored in blue.}
\label{fig-A4}
\end{figure}
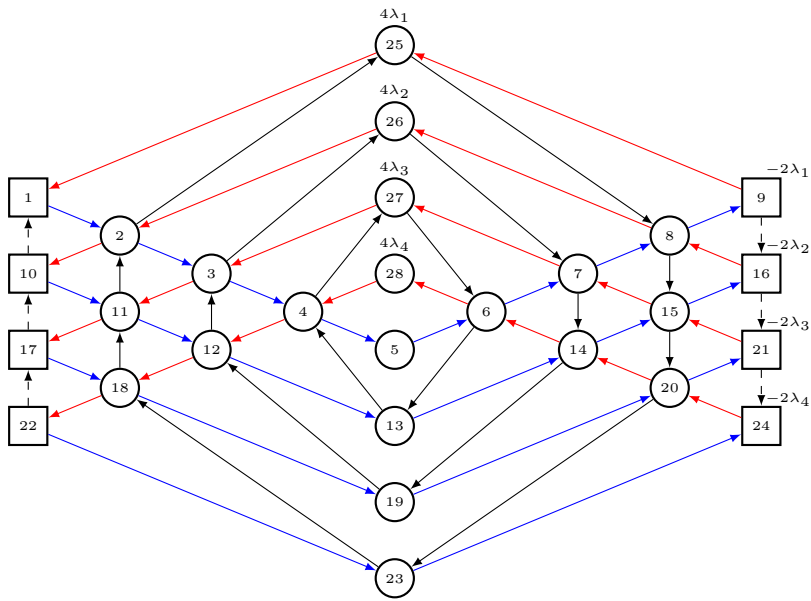
We have also indicated the assignment of the central parameters for the polarization of $\cU_q(\sl(n+1,\R))$, in such a way that 
\begin{itemize}
\item The polarization is group-like, i.e. $\pi_\l(K_iK_i')=1$ for $i\in I$, and
\item the remaining $n$ central monomials $Q_i$ are given by the product of the cluster variables along the monodromy cycles around the puncture.
\end{itemize}
For example, in Figure \ref{fig-A4} showing $n=4$, the central monomials are given by the cluster monomials
\Eq{\label{Qk}
Q_1&:=X_{2,25,8,15,20,23,18,11}\\
Q_2&:=X_{3,26,7,14,19,12}\nonumber\\
Q_3&:=X_{4,27,6,13}\nonumber\\
Q_4&:=X_{5,28}\nonumber
}
(here $Q_4$ corresponds to a degenerate $2$-cycle) such that
\Eq{
\pi_\l(Q_k) = e(4\l_k),\tab k=1,...,n.
}

Furthermore, following the mutation sequence provided in either \cite{FG1, Ip7, SS2}, we can flip the triangulation as in Figure \ref{PP} by $\mat{n+2\\3}$ cluster mutations, obtaining the quiver for $\cX_q^{\mathrm{sf}}$ corresponding to the self-folded triangulation in Figure \ref{fig-sf}.

\begin{figure}[htb!]
\centering 
\begin{tikzpicture}[baseline=(a),x=0.7cm,y=0.7cm]
\draw [vthick] (0,0) circle (2);
\node (a) at (0,0) [draw, circle, minimum size=0.2, inner sep=1.5]{};
\node (b) at (0,2) [draw, circle, minimum size=0.2, inner sep=1.5, fill=black]{};
\node (c) at (0,-2) [draw, circle, minimum size=0.2, inner sep=1.5, fill=black]{};
\path (a) edge[-, thin] (b);
\path (a) edge[-, thin] (c);
\draw[->,dashed,blue, thick] (-2,0)--(0,-1);
\draw[->,dashed,blue, thick] (0,-1)--(2,0);
\end{tikzpicture}
$\longrightarrow$
\begin{tikzpicture}[baseline=(a),x=0.7cm,y=0.7cm]
\draw [vthick] (0,0) circle (2);
\node (a) at (0,0) [draw, circle, minimum size=0.2, inner sep=1.5]{};
\node (b) at (0,2) [draw, circle, minimum size=0.2, inner sep=1.5, fill=black]{};
\node (c) at (0,-2) [draw, circle, minimum size=0.2, inner sep=1.5, fill=black]{};
\draw [thin](a) to (b);
\draw [thin](b) to [out=225, in=180](0,-0.7) to [out=0,in=-45](b);
\path (-2,0)edge[->, dashed,blue, thick, bend right=60](2,0);
\end{tikzpicture}
\caption{Flipping to a self-folded triangle. The dotted arrows indicate schematically the transformation of the $\bf_i$ paths on the quiver.}\label{PP}
\end{figure}
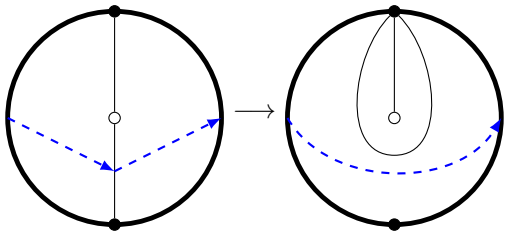

\begin{figure}[htb!]
\centering
\begin{tikzpicture}[every node/.style={inner sep=0, minimum size=0.5cm, thick}, x=1cm, y=1.3cm]

\node (22) at (3,0) [draw]{\tiny $22$};
\node[label={[xshift=0em, yshift=-2.7em] \tiny $-2\l_4$}] (24) at (5,0) [draw]{\tiny $24$};

\node (17) at (2,1) [draw]{\tiny $17$};
\node (23) at (4,1) [draw, circle]{\tiny $23$};
\node [label={[xshift=0em, yshift=-2.7em] \tiny $-2\l_3$}](21) at (6,1) [draw]{\tiny $21$};

\node (10) at (1,2) [draw]{\tiny $10$};
\node (18) at (3,2) [draw, circle]{\tiny $18$};
\node (20) at (5,2) [draw, circle]{\tiny $20$};
\node [label={[xshift=0em, yshift=-2.7em] \tiny $-2\l_2$}](16) at (7,2) [draw]{\tiny $16$};

\node (1) at (0,3) [draw]{\tiny $1$};
\node (11) at (2,3) [draw, circle]{\tiny $11$};
\node (19) at (4,3) [draw, circle]{\tiny $19$};
\node (15) at (6,3) [draw, circle]{\tiny $15$};
\node [label={[xshift=0em, yshift=-2.7em] \tiny $-2\l_1$}](9) at (8,3) [draw]{\tiny $9$};

\node (2) at (1,4) [draw, circle]{\tiny $2$};
\node (12) at (3,4) [draw, circle]{\tiny $12$};
\node (14) at (5,4) [draw, circle]{\tiny $14$};
\node (8) at (7,4) [draw, circle]{\tiny $8$};

\node (3) at (2,5) [draw, circle]{\tiny $3$};
\node (13) at (4,5) [draw, circle]{\tiny $13$};
\node (7) at (6,5) [draw, circle]{\tiny $7$};
\node [label={[xshift=0em, yshift=-0.4em] \tiny $4\l_1$}](25) at (8,5) [draw, circle]{\tiny $25$};

\node (4) at (3,6) [draw, circle]{\tiny $4$};
\node (6) at (5,6) [draw, circle]{\tiny $6$};
\node [label={[xshift=0em, yshift=-0.4em] \tiny $4\l_2$}](26) at (7,6) [draw, circle]{\tiny $26$};

\node (5) at (4,7) [draw, circle]{\tiny $5$};
\node [label={[xshift=0em, yshift=-0.4em] \tiny $4\l_3$}](27) at (6,7) [draw, circle]{\tiny $27$};

\node [label={[xshift=0em, yshift=-0.4em] \tiny $4\l_4$}](28) at (5,8) [draw, circle]{\tiny $28$};
\drawpath{1,11,19,15,9}{blue}
\drawpath{10,18,20,16}{blue}
\drawpath{17,23,21}{blue}
\drawpath{22,24}{blue}
\drawpath{9,8,25,2,1}{red}
\drawpath{16,15,14,7,26,3,12,11,10}{orange}
\drawpath{21,20,19,12,13,6,27,4,13,14,19,18,17}{purple}
\drawpath{24,23,18,11,2,3,4,5,6,7,8,15,20,23,22}{green}
\drawpath{4,28}{bend left, green}
\drawpath{28,6}{green}
\drawpath{25,7,13,3}{}
\drawpath{26,6,4}{}
\drawpath{3,25}{bend left=10}
\drawpath{4,26}{bend left=10}
\drawpath{22,17,10,1}{dashed}
\drawpath{9,16,21,24}{dashed}
\end{tikzpicture}
\caption{The quiver for the self-folded quiver $\cX_q^{\mathrm{sf}}$}\label{fig-sf}
\end{figure}
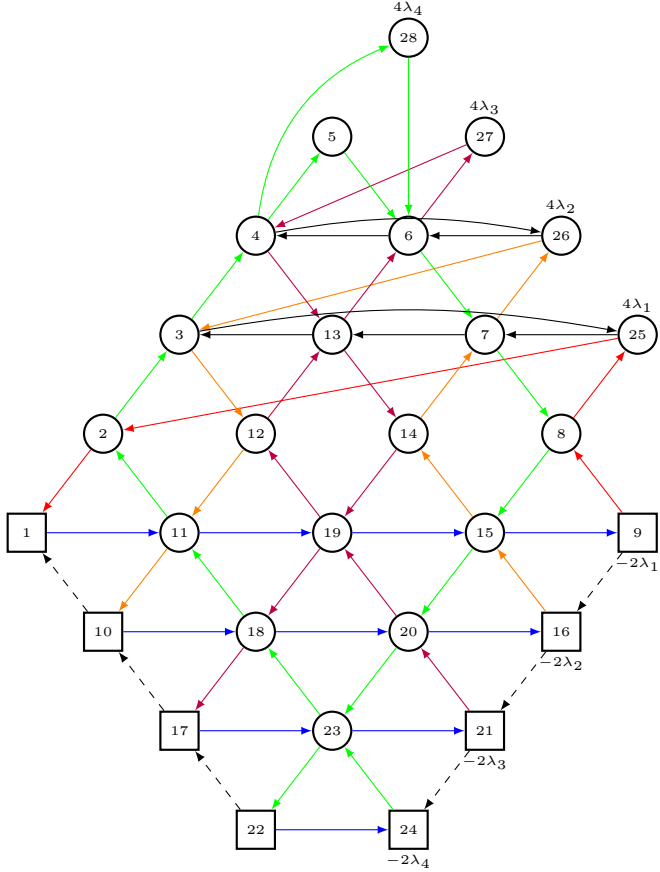

\begin{Rem} The quiver in Figure \ref{fig-sf} is a rearrangement of \cite[Figure 15]{SS2} where the top self-folded part is a mirror image, and the top vertex is connected since \cite{SS2} is dealing with the universal cover and modulo the deck transformation.
\end{Rem}

Finally, a mutation sequence, provided in \cite{SS2} in order to study the monodromy that gives the central characters, transforms the quiver into the \emph{self-folded symmetric} form associated to the quantum torus algebra $\cX_q^{\mathrm{sym}}$. Our new result here is the computation of the central parameters associated to the polarization of $\cX_q^{\mathrm{sym}}$.

\begin{Thm}\label{thmsym} The self-folded symmetric quiver of $\cX_q^{\mathrm{sym}}$ (using the same index), together with the central parameters of the vertices, is given by Figure \ref{fig-sym}.

\begin{figure}[htb!]
\centering
\begin{tikzpicture}[every node/.style={inner sep=0, minimum size=0.5cm, thick}, x=1cm, y=1.3cm]

\node (22) at (3,0) [draw]{\tiny $22$};
\node[label={[xshift=0em, yshift=-2.7em] \tiny $-2\l_4$}] (24) at (5,0) [draw]{\tiny $24$};

\node (17) at (2,1) [draw]{\tiny $17$};
\node (23) at (4,1) [draw, circle]{\tiny $23$};
\node [label={[xshift=0em, yshift=-2.7em] \tiny $-2\l_3$}](21) at (6,1) [draw]{\tiny $21$};

\node (10) at (1,2) [draw]{\tiny $10$};
\node (18) at (3,2) [draw, circle]{\tiny $18$};
\node (20) at (5,2) [draw, circle]{\tiny $20$};
\node [label={[xshift=0em, yshift=-2.7em] \tiny $-2\l_2$}](16) at (7,2) [draw]{\tiny $16$};

\node (1) at (0,3) [draw]{\tiny $1$};
\node (11) at (2,3) [draw, circle]{\tiny $11$};
\node (19) at (4,3) [draw, circle]{\tiny $19$};
\node (15) at (6,3) [draw, circle]{\tiny $15$};
\node [label={[xshift=0em, yshift=-2.7em] \tiny $-2\l_1$}](9) at (8,3) [draw]{\tiny $9$};

\node [label={[xshift=0em, yshift=-0.4em] \tiny $4\l_1$}](2) at (1,4) [draw, circle]{\tiny $2$};
\node (12) at (3,4) [draw, circle]{\tiny $12$};
\node (14) at (5,4) [draw, circle]{\tiny $14$};
\node (8) at (7,4) [draw, circle]{\tiny $8$};

\node [label={[xshift=0em, yshift=-0.4em] \tiny $4\l_2$}](13) at (2,5) [draw, circle]{\tiny $13$};
\node (7) at (4,5) [draw, circle]{\tiny $7$};
\node (25) at (6,5) [draw, circle]{\tiny $25$};

\node [label={[xshift=-0.5em, yshift=-0.4em] \tiny $4\l_3$}](26) at (3,6) [draw, circle]{\tiny $26$};
\node (3) at (5,6) [draw, circle]{\tiny $3$};

\node [label={[xshift=0em, yshift=-0.4em] \tiny $4\l_4$}](28) at (4,7) [draw, circle]{\tiny $28$};
\node (5) at (4,8) [draw, circle]{\tiny $5$};
\node [label={[xshift=0em, yshift=-0.4em] \tiny $-4\l_3$}](27) at (4,9) [draw, circle]{\tiny $27$};
\node [label={[xshift=0em, yshift=-0.4em] \tiny $-4\l_2-4\l_3$}](4) at (4,10) [draw, circle]{\tiny $4$};
\node [label={[xshift=0em, yshift=-0.4em] \tiny $-4\l_1-4\l_2-4\l_3$}](6) at (4,11) [draw, circle]{\tiny $6$};
\drawpath{1,11,19,15,9}{blue}
\drawpath{10,18,20,16}{blue}
\drawpath{17,23,21}{blue}
\drawpath{22,24}{blue}
\drawpath{9,8}{red}
\drawpath{8,2}{bend right=10, red}
\drawpath{2,1}{red}
\drawpath{16,15,14,25}{orange}
\drawpath{25,13}{bend right=10,orange}
\drawpath{13,12,11,10}{orange}
\drawpath{21,20,19,12,7,3, 26,7,14,19,18,17}{purple}
\drawpath{3,26}{bend right=10,green}
\drawpath{24,23,18,11,2,25,7,13,3}{green}
\drawpath{3,25,8,15,20,23,22}{green}
\drawpath{26,4,3}{green}
\drawpath{26,6,3}{green}
\drawpath{26,27,3}{green}
\drawpath{26,5,3}{green}
\drawpath{26,28,3}{green}
\drawpath{22,17,10,1}{dashed}
\drawpath{9,16,21,24}{dashed}
\path (0,6) edge[dotted] (8,6);
\end{tikzpicture}
\caption{The quiver for $\cX_q^{\mathrm{sym}}$ with central parameters.}\label{fig-sym}
\end{figure}
\end{Thm}
\begin{proof} A complete description and an example of the mutation are illustrated in Appendix \ref{App:mutate}. We mutate with respect to the opposite quiver on the top half so that the resulting quiver differs from those of \cite{SS2} by a mirror reflection. The computation of the central parameters of the vertices follows by induction, keeping track of the monomial changes at each wave of mutations.

The color of the paths again help illustrate the embedding of $\fD_q(\sl_{n+1})$.
\end{proof}

The quiver of $\cX_q^{\mathrm{sym}}$ can be considered as an amalgamation of the ``standard part" and the ``symmetric part", separated by the dotted lines in Figure \ref{fig-sym}. 
\begin{Def}
We call the vertices above the dotted lines the \emph{symmetric vertices} and label them from top to bottom by $c_0, c_1,...,c_n$ (e.g. vertex $6$, $4$, $27$, $5$, $28$ in Figure \ref{fig-sym}). 
\end{Def}

In particular, the monomials $Q_k:=X_{c_k}X_{c_{k-1}}\inv$ for $k=1,...,n$ form a generating set of the non-Cartan central monomials of $\cX_q^{\mathrm{sym}}$ with polarization
\Eq{
\pi_\l(Q_k) = e(4\l_k).
}
Hence, they are the image of the previous central monomials $Q_k$ under the sequence of mutations.

The important point to note is that all the Chevalley generators $\be_i$ and $\bf_i$, except $\be_n$, do not depend on the symmetric variables. On the other hand, the action of $\be_n$ is expressed in terms of the elementary symmetric polynomials $\cE_k$ in the cluster variables associated to the symmetric part. 

More precisely, rewriting the results of \cite{SS2}, we have
\Eq{\label{enAB}
\iota(\be_n)=\sum_{k=0}^{n+1} A_k B_k
}
where $B_0=1$, $B_k=\cE_k(X_{c_0},...,X_{c_{n}})$ are the elementary symmetric polynomials, and $A_k \in \cX_q^{\mathrm{sym}}$ are Laurent polynomials in the quantum cluster variables along the colored path of the quiver, which do not involve $X_{c_0},...,X_{c_{n}}$. By factorizing $X_{c_0}^k$ out of $B_k$, we see that
\Eq{
B_k:=X_{c_0}^k \cC_k
}
where 
\Eq{
\cC_k:=\cE_k(1,\frac{X_{c_1}}{X_{c_0}},\cdots \frac{X_{c_n}}{X_{c_0}})= \cE_k(1, Q_1,Q_1Q_2,...,Q_1\cdots Q_n)
} lies in the center of $\cX_q^{\mathrm{sym}}$. Furthermore, the polarization of $\cC_k$ gives
\Eq{\label{ckc}
\pi_\l(\cC_k)=e\left(-\frac{4}{n+1}\sum_{k=1}^n (n+1-k)\l_k\right) \pi_\l(\bC_k)
}
which is proportional to the action of the Casimir operators given by \eqref{casact}.

The following lemma is needed in Section \ref{sec:higher}.
\begin{Lem}\label{outer} Let $\Xi=X_{i_1,i_2,...,i_{2n}}\in \cX_q^{\mathrm{std}}$, where $i_1,...,i_{2n}$ are the indices of the frozen variables in the standard quiver (see Figure \ref{fig-A4}). In particular, the central parameter of $\Xi$ is given by $\dis -2\sum_{k=1}^n\ol_k$.

Under the cluster mutation of Theorem \ref{thmsym}, the monomial transforms as
\Eq{
\Xi\mapsto \mu_q(\Xi):=X'_{j_1,...,j_{n(n+2)}}
}
where $j_k$ are given by the indices of all the vertices of $\cX_q^{\mathrm{sym}}$, except $c_0,...,c_{n-1}$.
\end{Lem}
\begin{proof} One can follow the mutation sequence directly, as described in Appendix \ref{App:mutate}, to show the required transformation.

Alternatively, one can argue as follows. Consider the standard quiver $\bQ_{n+1}$ for type $A_{n+1}$ which is one rank higher as in Figure \ref{fig-A5}. We can regard $\Xi$ together with the top and bottom-most vertices (colored in red) as the central monomial $Q_1$ with central parameter $4\l_1$. Now we perform the symmetric folding for the $A_n$ subalgebra with root index $2,...,n+1$ (the part shaded in green). Note that the top and bottom-most vertices do not play a role under the mutation since they are outside of the $A_n$ subalgebra quiver. The resulting quiver is shown in Figure \ref{fig-A5sym}.

Since $\Xi$ lies in the center of $\cX_q^{\bQ_{n+1}}$, under the quantum mutation, only the monomial transformation as in Proposition \ref{monopolar} is effective, and hence the resulting image $\mu_q(\Xi)$ in the mutated quiver stays a monomial.

\begin{figure}[htb!]
\centering
\begin{tikzpicture}[every node/.style={inner sep=0, minimum size=0.2cm}, x=0.8cm, y=0.3cm]
\fill[green!10](5,11)--(9,7)--(9,1)--(5,-3)--(1,1)--(1,7)--cycle;
\foreach \y [count=\c from 1] in{33,28,21,12,1}
\foreach \x in{1,...,\c}{
	\pgfmathtruncatemacro{\ind}{\x+\y-1}
	\pgfmathtruncatemacro{\indop}{\y+2*\c-\x+1}
	\pgfmathtruncatemacro{\r}{6-\c}
	\pgfmathtruncatemacro{\ex}{\x-1-\c}
	\pgfmathtruncatemacro{\exop}{\c+1-\x}
	\ifthenelse{\x=1}{
		\node(\ind) at (\x-1,2*\c-1-\x)[draw] {};
		\node [label={[xshift=1.2em, yshift=-0.65em] \tiny $-2\l_\r$}](\indop) at (11-\x,2*\c-1-\x)[draw] {};
	}{\ifthenelse{\x=2}{
		\node(\ind) at (\x-1,2*\c-1-\x)[draw, circle, fill=red] {};
		\node (\indop) at (11-\x,2*\c-1-\x)[draw, circle, fill=red] {};
	}{
		\node(\ind) at (\x-1,2*\c-1-\x)[draw, circle]{};
		\node(\indop) at (11-\x,2*\c-1-\x)[draw, circle]{};
	}}
}
\foreach \x[count=\c from 1] in{6,16,24,30}{
\node (\x) at (5, 5-\c*2)[draw, circle]{};
}
\node(34) at (5,-5)[draw,circle,fill=red]{};
\foreach \x[count=\c from 1] in{36,37,38,39,40}{
\ifthenelse{\x=36}{
\node[label={[xshift=0em, yshift=-0.1em] \tiny $4\l_\c$}](\x) at (5, 15-\c*2)[draw, circle, fill=red]{};
}{
\node[label={[xshift=0em, yshift=-0.1em] \tiny $4\l_\c$}](\x) at (5, 15-\c*2)[draw, circle]{};
}
}

\foreach \x[count=\c from 1]in{33,28,21,12,1}{
\pgfmathtruncatemacro{\ee}{\x+2*\c}{
	\drawpath{\x,...,\ee}{}
}
}
\drawpath{11,36,1}{}
\drawpath{20,10,37,2,12}{}
\drawpath{27,19,9,38,3,13,21}{}
\drawpath{32,26,18,8,39,4,14,22,28}{}
\drawpath{35,31,25,17,7,40,5,15,23,29,33}{}
\drawpath{33,28,21,12,1}{dashed}
\drawpath{11,20,27,32,35}{dashed}
\drawpath{36,10,19,26,31,34,29,22,13,2,36}{}
\drawpath{37,9,18,25,30,23,14,3,37}{}
\drawpath{38,8,17,24,15,4,38}{}
\drawpath{39,7,16,5,39}{}
\end{tikzpicture}
\caption{The quiver for the standard embedding for type $A_5$, with the vertices of the central monomial $Q_1$ labeled in red.}\label{fig-A5}
\end{figure}
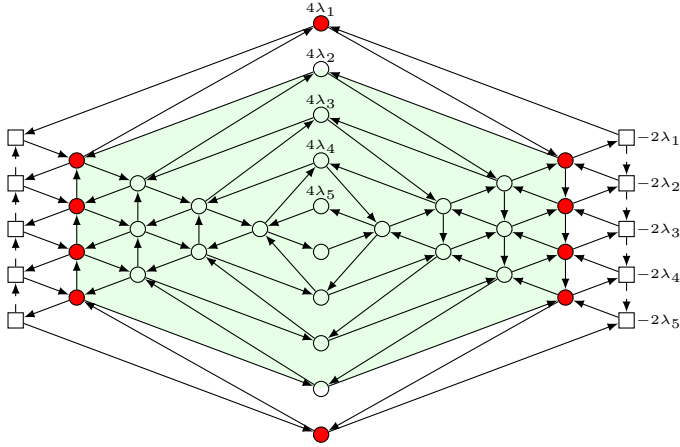

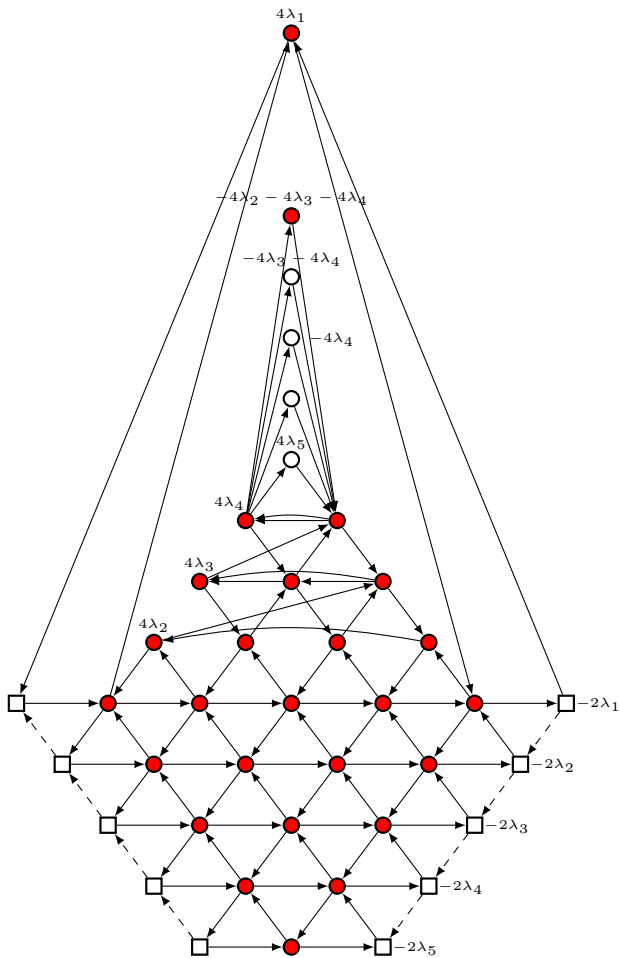
\begin{figure}[htb!]
\centering
\begin{tikzpicture}[every node/.style={inner sep=0, minimum size=0.2cm, thick}, x=0.6cm, y=0.8cm]

\node (31) at (-2,3)[draw]{};
\node (32) at (-1,2)[draw]{};
\node (33) at (0, 1)[draw]{};
\node (34) at (1, 0)[draw]{};
\node (35) at (2, -1) [draw]{};
\node [label={[xshift=1.2em, yshift=-0.7em] \tiny $-2\l_1$}] (36) at (10,3)[draw]{};
\node [label={[xshift=1.2em, yshift=-0.7em] \tiny $-2\l_2$}](37) at (9,2)[draw]{};
\node [label={[xshift=1.2em, yshift=-0.7em] \tiny $-2\l_3$}]  (38) at (8,1)[draw]{};
\node [label={[xshift=1.2em, yshift=-0.7em] \tiny $-2\l_4$}](39) at (7,0)[draw]{};
\node [label={[xshift=1.2em, yshift=-0.7em] \tiny $-2\l_5$}] (40) at (6,-1)[draw]{};
\node (41) at (4,-1)[draw,fill=red,circle]{};
\node [label={[xshift=0em, yshift=0em] \tiny $4\l_1$}](42)  at (4,14)[draw,fill=red,circle]{};

\node (22) at (3,0) [draw, fill=red, circle]{};
\node (24) at (5,0) [draw, fill=red, circle]{};

\node (17) at (2,1) [draw, fill=red, circle]{};
\node (23) at (4,1) [draw, fill=red, circle]{};
\node (21) at (6,1) [draw, fill=red, circle]{};

\node (10) at (1,2) [draw, fill=red, circle]{};
\node (18) at (3,2) [draw, fill=red, circle]{};
\node (20) at (5,2) [draw, fill=red, circle]{};
\node  (16) at (7,2) [draw, fill=red, circle]{};

\node (1) at (0,3) [draw, fill=red, circle]{};
\node (11) at (2,3) [draw, fill=red, circle]{};
\node (19) at (4,3) [draw, fill=red, circle]{};
\node (15) at (6,3) [draw, fill=red, circle]{};
\node (9) at (8,3) [draw, fill=red, circle]{};

\node [label={[xshift=0em, yshift=0em] \tiny $4\l_2$}](2) at (1,4) [draw, fill=red, circle]{};
\node (12) at (3,4) [draw, fill=red, circle]{};
\node (14) at (5,4) [draw, fill=red, circle]{};
\node (8) at (7,4) [draw, fill=red, circle]{};

\node [label={[xshift=0em, yshift=0em] \tiny $4\l_3$}](13) at (2,5) [draw, fill=red, circle]{};
\node (7) at (4,5) [draw, fill=red, circle]{};
\node (25) at (6,5) [draw, fill=red, circle]{};

\node [label={[xshift=-0.6em, yshift=0em] \tiny $4\l_4$}](26) at (3,6) [draw, fill=red, circle]{};
\node (3) at (5,6) [draw, fill=red, circle]{};

\node [label={[xshift=0em, yshift=0em] \tiny $4\l_5$}](28) at (4,7) [draw, circle]{};
\node (5) at (4,8) [draw, circle]{};
\node [label={[xshift=1.5em, yshift=-0.7em] \tiny $-4\l_4$}](27) at (4,9) [draw, circle]{};
\node [label={[xshift=0em, yshift=0em] \tiny $-4\l_3-4\l_4$}](4) at (4,10) [draw, circle]{};
\node [label={[xshift=0em, yshift=0em] \tiny $-4\l_2-4\l_3-4\l_4$}](6) at (4,11) [draw, circle, fill=red]{};
\drawpath{1,11,19,15,9}{}
\drawpath{10,18,20,16}{}
\drawpath{17,23,21}{}
\drawpath{22,24}{}
\drawpath{9,8}{}
\drawpath{8,2}{bend right=10}
\drawpath{2,1}{}
\drawpath{16,15,14,25}{}
\drawpath{25,13}{bend right=10,}
\drawpath{13,12,11,10}{}
\drawpath{21,20,19,12,7,3, 26,7,14,19,18,17}{}
\drawpath{3,26}{bend right=10,}
\drawpath{24,23,18,11,2,25,7,13,3}{}
\drawpath{3,25,8,15,20,23,22}{}
\drawpath{26,4,3}{}
\drawpath{26,6,3}{}
\drawpath{26,27,3}{}
\drawpath{26,5,3}{}
\drawpath{26,28,3}{}
\drawpath{22,17,10,1}{}
\drawpath{9,16,21,24}{}
\drawpath{31,1,32,10,33,17,34,22,35,41,40,24,39,21,38,16,37,9,36}{}
\drawpath{24,41,22}{}
\drawpath{36,42,9}{}
\drawpath{1,42,31}{}
\drawpath{35,34,33,32,31}{dashed}
\drawpath{36,37,38,39,40}{dashed}
\end{tikzpicture}
\caption{The partially self-folded symmetric quiver in type $A_4\subset A_5$, with the resulting vertices for $C_1$ colored in red.}
\label{fig-A5sym}
\end{figure}
Then, we check directly that the product $\Xi'=X'_{j_1,...,j_{n(n+2)}}$ of all the cluster variable as described (colored in red), lies in the center of the resulting quantum torus algebra $\cX_q^{\bQ_{n+1}^{\mathrm{sym}}}$, with central parameter $4\ol_1$. Since $\Xi$ does not involve any frozen variables of $\cX_q^{\bQ_{n+1}^{\mathrm{sym}}}$, there is no contribution by the Cartan monomials $\iota(\bK_i\bK_i')$. Hence, $\Xi'$ must coincide with the transformed monomial $\mu_q(\Xi)$ of the original $\Xi$.
\end{proof}

\subsection{Positive representations at zero Casimir}\label{sec:An:pos0}
Following the motivation from the previous section, we can now state the following Main Theorem.
\begin{Thm}\label{mainthmAn}
There is an embedding from $\fD_q(\sl_{n+1})/\<\bC_k=0\>$ to a skew-symmetrizable quantum cluster algebra $\cO_q(\cX^0)$, such that the image of the Chevalley generators are universally Laurent polynomials.

Passing to a polarization, we have an irreducible representation $\cP^0$ of $\cU_q(\sl(n+1,\R))$ acting on $L^2(\R^N)$ as positive operators, where $N=\frac{n(n+1)}{2}$, such that all the Casimir operators $\bC_k$ act by zero.
\end{Thm}
\begin{proof} We identify all the $n+1$ symmetric vertices by taking their product as a new cluster variable $\dis X_\star:=\prod_{i=0}^n X_{c_i}$. We also let $d_\star=n+1$ be the multiplier of this new vertex. We obtain the quantum torus algebra $\cX_q^0$ with the quiver given by Figure \ref{fig-A4X0}.

\begin{figure}[htb!]
\centering
\begin{tikzpicture}[every node/.style={inner sep=0, minimum size=0.5cm, thick}, x=1cm, y=1.3cm]

\node (22) at (3,0) [draw]{\tiny $22$};
\node(24) at (5,0) [draw]{\tiny $24$};

\node (17) at (2,1) [draw]{\tiny $17$};
\node (23) at (4,1) [draw, circle]{\tiny $23$};
\node (21) at (6,1) [draw]{\tiny $21$};

\node (10) at (1,2) [draw]{\tiny $10$};
\node (18) at (3,2) [draw, circle]{\tiny $18$};
\node (20) at (5,2) [draw, circle]{\tiny $20$};
\node (16) at (7,2) [draw]{\tiny $16$};

\node (1) at (0,3) [draw]{\tiny $1$};
\node (11) at (2,3) [draw, circle]{\tiny $11$};
\node (19) at (4,3) [draw, circle]{\tiny $19$};
\node (15) at (6,3) [draw, circle]{\tiny $15$};
\node (9) at (8,3) [draw]{\tiny $9$};

\node (2) at (1,4) [draw, circle]{\tiny $2$};
\node (12) at (3,4) [draw, circle]{\tiny $12$};
\node (14) at (5,4) [draw, circle]{\tiny $14$};
\node (8) at (7,4) [draw, circle]{\tiny $8$};

\node (13) at (2,5) [draw, circle]{\tiny $13$};
\node (7) at (4,5) [draw, circle]{\tiny $7$};
\node (25) at (6,5) [draw, circle]{\tiny $25$};

\node (26) at (3,6) [draw, circle]{\tiny $26$};
\node (3) at (5,6) [draw, circle]{\tiny $3$};

\node (28) at (4,7) [draw, circle]{$\star$};
\drawpath{1,11,19,15,9}{blue}
\drawpath{10,18,20,16}{blue}
\drawpath{17,23,21}{blue}
\drawpath{22,24}{blue}
\drawpath{9,8}{red}
\drawpath{8,2}{bend right=10, red}
\drawpath{2,1}{red}
\drawpath{16,15,14,25}{orange}
\drawpath{25,13}{bend right=10,orange}
\drawpath{13,12,11,10}{orange}
\drawpath{21,20,19,12,7,3, 26,7,14,19,18,17}{purple}
\drawpath{3,26}{bend right=10,green}
\drawpath{24,23,18,11,2,25,7,13,3}{green}
\drawpath{3,25,8,15,20,23,22}{green}
\drawpath{26,28,3}{green, vthick}
\drawpath{22,17,10,1}{dashed}
\drawpath{9,16,21,24}{dashed}
\end{tikzpicture}
\caption{The quiver for the skew-symmetrizable $\cX_q^0$ in type $A_4$.}\label{fig-A4X0}
\end{figure}
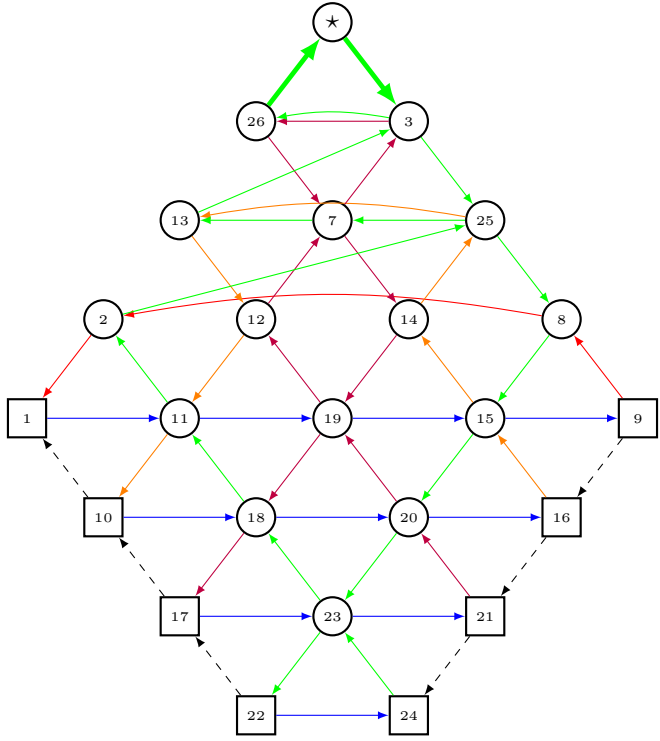

Note that in this quantum cluster seed, we have new relations (e.g. using the labeling from Figure \ref{fig-A4X0})
\Eq{
X_\star X_3=q^{-2(n+1)} X_3 X_\star,\tab X_\star X_{26}=q^{2(n+1)}X_{26}X_\star,
}

This identification comes from setting all the elementary symmetric polynomials $B_k$, $k=1,...,n$ in the embedding \eqref{enAB} of $\be_n$ to zero, while $B_{n+1}$ becomes $X_\star$, so that the image of $\be_n$ reduces to a single ``jump'' in the telescopic sum in the combined variable $X_\star$ from the last monomial term of $A_0$ to the first monomial term of $A_{n+1}B_{n+1}$. We also have the corresponding images for $\bK_n$. Since the symmetric part is proportional to the Casimir element, which lies in the center, the resulting expressions for $\be_n$ and $\bK_n$ obtained by setting the symmetric part to the scalar zero provide a homomorphic image of $\fD_q(\sl_{n+1})$ in $\cX_q^0$.

By the usual properties of the folding of quantum cluster algebras, a mutation at $X_\star$ is equivalent to first mutating simultaneously (unordered) the original unfolded expression at $X_{c_0},...,X_{c_n}$, which remains a symmetric expression involving $X_{c_i}$, and then setting the symmetric part 
\Eq{
\cE_i(X_{c_0},...,X_{c_n})&\mapsto 0, \tab i=1,...,n\\
\cE_{n+1}(X_{c_0},...,X_{c_n})=\prod_{k=0}^n X_{c_i} &\mapsto X_\star.
}
On the other hand, mutating at other variables different from $X_\star$ works as in the original unfolded expression. Since the Chevalley generators are known to be universally Laurent polynomials in the unfolded quantum cluster algebra $\cO_q(\cX)$, the folding procedure still produces universally Laurent polynomials in $\cO_q(\cX^0)$.

In terms of positive representation, the resulting polarization of $X_\star$ is given by the $n$-th power of the Weyl expression of $X_{c_0}$ but with the central parameter associated to $X_\star$ vanishing. Since we essentially identified all the central monomials $Q_k$ from the folding, the resulting positive representation does not depend on any parameter, which matches with the rank of the quiver of $\cX_q^0$. 

Irreducibility in the case of $\cU_q(\sl(2,\R))$ is clear from the action of the pair $(\bK, \bf)$. To prove the irreducibility of the representation in general, we require the following Lemma.

\begin{Lem}\label{fCas} By the embedding $\fD_q(\sl_{n+1})\inj \cX_q$, the generator $\be_n$ can be expressed as a rational expression in $\be_1,...,\be_{n-1}$, $\bf_1,...,\bf_n$, $\bK_1,...,\bK_n$ as well as the Casimirs $\bC_1,...,\bC_n$, considered as elements in the field of fraction $\bT_q$.
\end{Lem}

Recall that in the representation $\cP^0$ constructed from folding, only the action of the generator $\be_n$ is modified and the polarization of all remaining $3n-1$ Chevalley generators remain the same as $\cP_\l$. By Lemma \ref{fCas}, any operator that strongly commutes with the $3n-1$ Chevalley generators also commutes with $\be_n$, and is hence the multiplication by a scalar by the irreducibility of the standard positive representation $\cP_\l$.
\end{proof}
\begin{proof}[Proof of Lemma \ref{fCas}] In $\cU_q(\sl_2)$, recall that we have the expression
$$\be=(\bC-q\inv\bK-q\bK\inv)\bf\inv.$$
By the embedding to $\cX_q$, this element makes sense as a rational element in $\bT_q$, which factorizes and simplifies to a polynomial representing the image of $\be\in \cX_q$.

In general for type $A_n$, the fundamental representations are given by the exterior power of the standard representation $V_k:=\L^k V_1$. It is known that the action of the quantum group generators $E_\a^2=0$ for any positive root $\a$ in these representation. Hence the formula for the universal $R$-matrix restricted to $V_k$ is simply given by \cite{Ip5}
$$R|_{1\ox V_k}=\cK\prod_{\a\in \D^+} (1+\be_\a \ox \bf_\a)$$
where $\cK$ is the Cartan part and the product is over a normal ordering of the positive roots determined by the standard longest word $\bi_{A_n}$.

As a consequence, by expanding the trace and commuting the generators $\be_n$ with $\bK_i$ as well as any $(\be_i, \be_j)$ for $|i-j|>1$, we conclude that the explicit expressions of the Casimir elements are spanned by $1, \be_n, \be_{n-1}\be_n,...,\be_1\cdots \be_n$ with coefficients in $\cU_q(\g)$ without $\be_n$ multiplying from the right. (For example, see \eqref{Cas3a}--\eqref{Cas3b} after expanding the $\be_{ij}$ terms.)

Treating $x_k:=\be_k\cdots \be_n$ as indeterminates, we obtain a system of $n$ linear equations with coefficient in $\cU_q(\g)$. Hence, $\be_n$ can be solved as a rational expression in terms of the remaining $3n-1$ generators as well as the $\bC_k$.
\end{proof}

Formally, setting $\l$ to a general solution of \eqref{symeq} also gives the representation $\cP^0$ by \eqref{ckc}. However, since $\l$ is complex, the default polarization of some cluster variables of $\cX_q^0$ may not be positive if it has nontrivial central parameter. We fix this as follows.
\begin{Cor}\label{eqcor}
 $\cP^0$ is obtained from $\cP_\l$ by 
 \begin{itemize}
 \item choosing a polarization of $\cX_q^{\mathrm{sym}}$ where the central parameters of each symmetric variable $X_{c_i}$ are shifted by a constant $c$, such that $X_\star:=X_{c_0,...,c_n}$ as well as all other variables have trivial central parameters, and
 \item setting $\l$ to a general solution $\l_0\in\C^n$ of \eqref{symeq}.
 \end{itemize}
We may formally write $\cP^0 = \cP_{\l_0}$.
\end{Cor}
\begin{proof} By direct computation, the constant $c$ is given by
\Eqn{
c&=\frac{1}{n}\left(4\l_n-4\l_{n-1}-(4\l_{n-1}+4\l_{n-2})-\cdots - 4\sum_{k=1}^{n-1}\l_k\right)\\
&=4\l_n - \frac{4}{n}\sum_{k=1}^{n-1}k\l_k.
}
One checks that the polarizations of all central monomials remain invariant. The symmetric folding that takes $X_\star=X_{c_0,...,c_n}$ removes all the central parameters in $\cX_q^0$ and the polarization of $\cX_q^0$ is still positive by construction.
\end{proof}

\begin{Rem}\label{notstd} We remark that even the Chevalley generators are universally Laurent polynomials in $\cO_q(\cX^0)$, in type $A_n$ for $n\geq 3$, the special generator $\be_n$ is no longer a standard monomial in general.
\end{Rem}

Finally, we have to show that the construction is independent of the choice of the ``special" generator $\be_n$ used, since the standard positive representations in type $A_n$ is symmetric with respect to the Dynkin involution $\be_i\corr \bf_{i^*}$ as well as $\be_i\corr \be_{i^*}$ induced from the change of words $\bi_0\corr \bi_0^*$ by Coxeter moves, both of which are related by cluster mutations and changes of polarizations. This is addressed in the following Proposition.

\begin{Prop}\label{mainprop} Assume $\cP_\l\simeq \cP_\l'$ are two unitarily equivalent standard positive representations of $\cU_q(\g)$ related by quantum cluster mutations. Then the degenerate representations obtained as in Corollary \ref{eqcor} by setting $\l$ to a general solution of \eqref{symeq} are still unitarily equivalent,
\Eq{\cP^0\simeq {\cP^0}'. \label{l0eq}}
\end{Prop}
\begin{proof} For a given polarization of $\cX_q$, it is known \cite{FG3} that a quantum cluster mutation is realized by unitary transformations on $\cH\simeq L^2(\R^N)$ of the form
$$\Psi=\cM\circ g_b(e^{\pi \mathring{L}})$$
where $g_b$ is the quantum dilogarithm function, $L$ as in Notation \ref{ELnote}, and $\cM\in Sp(2N)$ gives a change of variables with respect to $\{u_i, p_j\}$ realizing the monomial transformation described in Proposition \ref{monopolar}, or in other words, just an invertible change of polarization (cf. Lemma \ref{polaru}).

We recall some analytic properties of the quantum dilogarithm function $g_b(e^{2\pi b x})$ which can be found in \cite{Ip1, PT2}. It can be analytically continued to a meromorphic function $G(x)$, such that it admits poles at $x=\bi\left(\frac{b+b\inv}{2}+nb+mb\inv\right)$ and zeros at $x=-\bi\left(\frac{b+b\inv}{2}+nb+mb\inv\right)$ for $n,m\in \Z_{\geq0}$. Furthermore, for a constant $\a\in\R$, $G(x+\bi\a)$ has at most exponential growth on the real line $x\in\R$ bounded by $e^{2\pi b \a x}$, and is hence a densely defined operator on $L^2(\R)$. The adjoint can be similarly defined for $\a\in\R$, with the reciprocal analytic properties, by $$g_b(e^{2\pi b x+\bi\a})^*:=g_b(e^{2\pi b x-\bi\a})\inv.$$

By a choice of polarization, the general operator $g_b(e^{\pi\mathring{L}})$ can be reduced to an action by a single variable in which the previous analytic properties apply. Furthermore, it depends analytically on $\l_i$. By setting $\l_1=\cdots=\l_n=\frac{\bi}{2(n+1)b}$ to be the general solution of \eqref{symeq}, $x\in \R$ will not pass through the poles and zeros of $g_b$ since $\l_i$ does not involve half multiples of $b$. Hence, we conclude that $g_b(e^{\pi\mathring{L}})$ is invertible on a dense domain and the adjoint is also well-defined.

Therefore, given a composition of the transformations with respect to the mutation sequence (together with a change of variables) realizing $\cP_\l\simeq \cP_\l'$, by analytically continuing $\l$ towards the general solution of \eqref{symeq}, we obtain a nonzero densely defined invertible intertwiner:
$$\Phi: \cP^0\simeq {\cP^0}'.$$

Now by Theorem \ref{mainthmAn}, both representations $\cP^0$ and ${\cP^0}'$ are positive and irreducible defined on a dense subspace of $L^2(\R^N)$. Hence, for any $X\in \cU_q(\g)$ and $t\in \R$ such that $\pi^0(X)^{\bi t}$ is bounded and unitary, we have by positivity
\Eqn{
\Phi^*\Phi\pi^0(X)^{\bi t}&=\Phi^*{\pi^0}'(X)^{\bi t}\Phi\\
&=({\pi^0}'(X)^{-\bi t}\Phi)^*\Phi\\
&=(\Phi\pi^0(X)^{-\bi t})^*\Phi \\
&= \pi^0(X)^{\bi t}\Phi^*\Phi}
hence $\Phi^*\Phi$ commutes strongly with any $\pi^0(X)$. Since $\Phi^*\Phi$ is nonzero, upon rescaling if necessary it must be the identity operator by irreducibility, so $\Phi$ is a unitary equivalence.
\end{proof}
\begin{Ex} In the case $\cU_q(\sl(2,\R))$, consider the quiver in Figure \ref{fig-A1}. We compare the two representations with quivers obtained by a mutation at $2$ followed by folding, and by a mutation at $4$ followed by folding. The two folded quivers are related by a cluster mutation with double weight (see Figure \ref{fig-X2fold}--\ref{fig-A1mut}). By choosing the standard polarization, the conclusion of Proposition \ref{mainprop} gives us an identity between the folded unitary transformation and the unfolded one by specifying $\l=\frac{\bi}{4b}$, namely
\Eq{
g_{\sqrt{2}b}(e^{4\pi b u}) =c\cdot g_b(e^{2\pi bu}e^{\frac{\pi \bi}{2}})g_b(e^{2\pi bu}e^{-\frac{\pi \bi}{2}})
}
for some constant $c\in \C$ with unit norm. By setting $u=0$ and using the functional equation for $g_b$ (cf. \cite{Ip1}), one calculates the constant to be just $c=1$. This identity, valid for $0<b<1$, should be considered as the analytic continuation of the more well-known ones obtained for quantum dilogarithm at the root of unity \cite{IY}.
\end{Ex}
\begin{Rem}\label{notmut} By Remark \ref{notstd} however, we see that in general the unitary equivalence in \eqref{l0eq} may not be expressible as quantum cluster mutations in $\cX_q^0$ anymore, since the unitary transformation $g_b(\pi(X_\star)^\frac{1}{n+1})$ corresponding to the unfolded variable is not available. It will be interesting to find all such equivalences that can actually be realized as cluster mutations, which hold e.g. in type $A_1$ and $A_2$.
\end{Rem}
\subsection{Modular double counterpart}\label{sec:An:mod}
\begin{Thm}\label{thmmod} The representation $\cP^0$ of $\cU_q(\sl(n+1,\R))$ induces a representation $\til{\cP}^0$ of its modular double counterpart $\cU_{\til{q}}(\sl(n+1,\R))$ with the central parameter $\l$ given by the general solution of \eqref{symeq}.
\end{Thm}
\begin{proof}
Let $b_\star:=\sqrt{n+1}b$ and define formally
\Eqn{
\til{X}_i &:= X_i^{\frac{1}{b^2}},\tab  i\neq \star\\
\til{X}_\star&:= X_\star^{\frac{1}{(n+1)b^2}}=X_\star^{\frac{1}{b_\star^2}}
}
with weight $d_\star=\frac{1}{n+1}$ and $d_i=1$ for $i\neq \star$. Let $\til{\cX}_q^0$ be the corresponding quantum torus algebra they generate. (Alternatively one can define $\til{\cX}_q^0$ through a proper rescaling of the lattice and seed basis.)

Note that the polarization of $\til{X}_\star^{n+1}$ is the same as that of $X_\star$ with $b\corr b\inv$.

If $X_iX_\star=q^{-2(n+1)}X_\star X_i$, then
$$\mathrm{Ad}_{g_{b_\star}(X_\star)} X_i = X_i+X_{i,\star}.$$
By the modular invariance of the quantum dilogarithm function, we have
$$g_{b_\star} (X_\star)=g_{b_\star\inv}(\til{X}_\star).$$
Since $\til{X}_i\til{X}_\star=\til{q}^{-2} \til{X}_\star\til{X}_i$, by the binomial theorem, the same mutation gives
$$\mathrm{Ad}_{g_{b_\star\inv}(\til{X}_\star)}\til{X}_i = \sum_{k=0}^{n+1} \bin{n+1\\k}_{\til{q}^{\frac{1}{n+1}}}\til{X}_{i,\star^k}.$$
Therefore, by comparing the coefficients in the unfolded cluster algebra, we see that the modular double counterpart is obtained by setting the Casimir actions of ${\cU_{\til{q}}(\sl(n+1,\R))}$ to $$\pi_\l(\til{\bC}_k)=\bin{n+1\\k}_{\til{q}^{\frac{1}{n+1}}}$$
which is the same as specifying the central parameters $\l\in\C^n$ to one of the general solutions of \eqref{symeq}.
\end{proof}
Replacing $\til{q}$ with $q$ in the discussion above, we have
\begin{Cor}\label{cormainAn} There is a homomorphism from $\cU_q(\sl_{n+1})$ to a skew-symmetrizable quantum cluster algebra $\cO_q(\til{\cX}^0)$ such that for any irreducible polarization $\til{\pi}$ the generalized Casimir elements act as 
\Eq{
\til{\pi}(\bC_k)=\bin{n+1\\k}_{q^{\frac{1}{n+1}}},\tab k=1,...,n} and the Chevalley generators are realized as universally Laurent polynomials.
\end{Cor}

The transcendental relations however become more subtle. In the original positive representations, the proof relies heavily on the fact that the Chevalley generators can be transformed to cluster monomials, such that the transcendental relations are satisfied for any choice of polarization:
\Eq{
(e^{2\pi bL(\bu,\bp, \l)})^{\frac{1}{b^2}}=e^{2\pi b\inv L(\bu, \bp, \l)}.
}
Hence, the transcendental relations hold for $\be_1,...,\be_{n-1}$ as well as $\bf_1,...,\bf_n$ and $\bK_1,...,\bK_n$. However, by Remark \ref{notstd}, the generator $\be_n$ is in general no longer a standard monomial. A new algebraic or analytic way is thus required to prove this relation.
\begin{Con} 
The transcendental relations are satisfied for all Chevalley generators of $\cU_q(\sl(n+1,\R))$ with $\l\in\C^n$ taken to be any general solution in \eqref{symeq}.
\end{Con}
For completeness, we demonstrate an algebraic proof in type $A_2$ below.
\begin{Prop} 
The transcendental relations are satisfied for $\cU_q(\sl(2,\R))$ and $\cU_q(\sl(3,\R))$ with $\l$ taken to be any general solution in \eqref{symeq}.
\end{Prop}
\begin{proof} The case for $\cU_q(\sl(2,\R))$ is proved in Section \ref{sl2mod}. 

For $\cU_q(\sl(3,\R))$, the symmetric quiver is shown in Figure \ref{fig-A2X0}. The thick green arrows indicate that $X_2X_\star=q^{-6}X_\star X_2$ and $X_4X_\star=q^{6}X_\star X_4$.

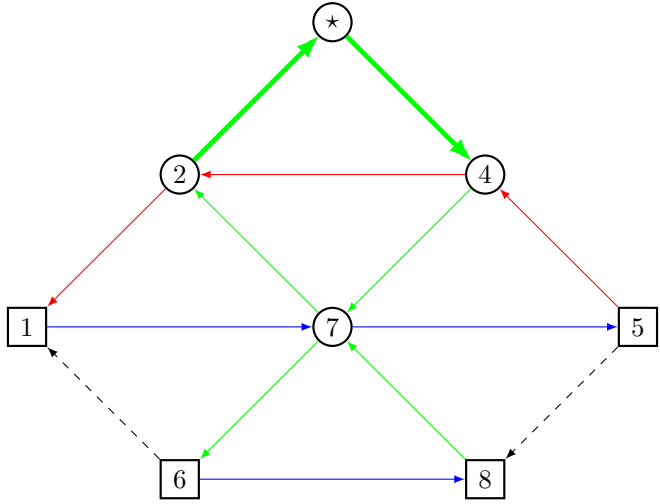
\begin{figure}
\centering
\begin{tikzpicture}[every node/.style={inner sep=0, minimum size=0.5cm, thick}, x=2cm, y=2cm]
\node (6) at (1,0)[draw]{$6$};
\node (8) at (3,0)[draw]{$8$};
\node (1) at (0,1)[draw]{$1$};
\node (7) at (2,1)[draw,circle]{$7$};
\node (5) at (4,1)[draw]{$5$};
\node (2) at (1,2)[draw,circle]{$2$};
\node (4) at (3,2)[draw,circle]{$4$};
\node (3) at (2,3)[draw,circle]{$\star$};
\drawpath{1,7,5}{blue}
\drawpath{6,8}{blue}
\drawpath{5,4,2,1}{red}
\drawpath{8,7,2,3,4,7,6}{green}
\drawpath{2,3,4}{green, vthick}
\drawpath{6,1}{dashed}
\drawpath{5,8}{dashed}
\end{tikzpicture}
\caption{The quiver for $\cX_q^0$ in type $A_2$.}\label{fig-A2X0}
\end{figure}

The image of $\be_2$ under the homomorphism $\fD_q(\sl_3)\to \cX_q^0$ is given by the green path
\Eqn{
\be_2\mapsto e_2=X_8+X_{8,7}+X_{8,7,2}+X_{8,7,2,\star}+X_{8,7,2,\star,4}+X_{8,7,2,\star,4,7}.
}
Now we observe that
\Eqn{
(X_8+X_{8,7})(X_{8,7,2}+X_{8,7,2,\star})&=q^{-2}(X_{8,7,2}+X_{8,7,2,\star})(X_8+X_{8,7})\\
(X_8+X_{8,7})(X_{8,7,2,\star,4}+X_{8,7,2,\star,4,7})&=q^{-2}(X_{8,7,2,\star,4}+X_{8,7,2,\star,4,7})(X_8+X_{8,7})\\
(X_{8,7,2}+X_{8,7,2,\star})(X_{8,7,2,\star,4}+X_{8,7,2,\star,4,7})&=q^{-2}(X_{8,7,2,\star,4}+X_{8,7,2,\star,4,7})(X_{8,7,2}+X_{8,7,2,\star}).
}
Hence, by Lemma \ref{qbin}, we have
$$e_2^{\frac{1}{b^2}}=(X_8+X_{8,7})^{\frac{1}{b^2}}+(X_{8,7,2}+X_{8,7,2,\star})^{\frac{1}{b^2}}+X_{8,7,2,\star,4}+(X_{8,7,2,\star,4,7})^{\frac{1}{b^2}}.$$
Note that the two terms in the first and the last brackets $q^{-2}$-commute, while those in the middle bracket $q^{-6}$-commute. Also recall from definition that $X_\star^{\frac{1}{b^2}}=\til{X}_\star^3$. Hence, again by Lemma \ref{qbin}, we compute
{\small\Eqn{
e_2^{\frac{1}{b^2}}&=(X_8^{\frac{1}{b^2}}+X_{8,7}^{\frac{1}{b^2}})+(X_{8,7,2}+X_{8,7,2,\star})^{\frac{1}{3b^2}3}+(X_{8,7,2,\star,4}^{\frac{1}{b^2}}+X_{8,7,2,\star,4,7}^{\frac{1}{b^2}})\\
&=\til{X}_8+\til{X}_{8,7}+(\til{X}_{8,7,2}^\frac13+\til{X}_{8,7,2,\star^3}^\frac13)^3+\til{X}_{8,7,2,\star^3,4}+\til{X}_{8,7,2,\star^3,4,7}\\
&=\til{X}_8+\til{X}_{8,7}+(\til{X}_{8,7,2}+[3]_{\til{q}^\frac13}\til{X}_{8,7,2,\star}+[3]_{\til{q}^\frac13}\til{X}_{8,7,2,\star^2}+\til{X}_{8,7,2,\star^3})+\til{X}_{8,7,2,\star^3,4}+\til{X}_{8,7,2,\star^3,4,7}
}}
which compares with the unfolded expression by setting the central parameter $\l$ to be the general solution of \eqref{symeq}.

Alternatively, one can first mutate $\be_2$ at the vertices $7,3,2,4$ to reduce it to a standard monomial $X_8'$ in the resulting cluster, apply $\frac{1}{b^2}$-th power, and reverse the mutation to obtain the same expression above.
\end{proof}
\section{Degenerate representations for general Lie types}\label{sec:higher}

For general Lie types, it seems impossible to nullify all the central characters simultaneously as in type $A_n$. This is particularly true for the non-simply-laced cases as the central characters corresponding to the long and short root should be treated separately. Therefore, the next best scenario is to isolate some parabolic part of $\cU_q(\g)$ that is isomorphic to a direct product of type $A_k$ Lie subalgebras and perform the folding construction for each of those subalgebras.

Let $J\subset I$ be a nontrivial subset of the Dynkin index such that the parabolic subgroup $W_J\subset W$ of the Weyl group is isomorphic to a direct product of type $A_k$ Weyl group. Let $w_J$ be the longest element of $W_J$ and write $w_0 = w_Jw'$, with its corresponding longest expression decomposed as $\bi_0=\bi_J\bi'$. We have an embedding of $\cU_q(\g)$ into the quantum torus algebra constructed from $\bi_0=\bi_J\bi'$. 

In \cite{Ip8}, we proved using the generalized Heisenberg double, that we still have a homomorphism of $\cU_q(\g)$ as long as the parabolic part of the quiver corresponding to $\bi_J$ still produces an image of $\cU_q(\g)$. In other words, we can perform the construction in the previous section to the $\bi_J$ part of the quiver. The rank of $W_J$ dictates the reduction of the dimension of the center. Hence, under a choice of polarization, we obtain a new family of representations called the \emph{degenerate positive representations} denoted by $\cP_\l^{0,J}$, parametrized by $n-|J|$ scalars. Note that as Hilbert spaces, they have the same functional dimension
\Eq{
\cP_\l^{0,J}\simeq L^2(\R^N) \simeq \cP_\l.
}
Since we are only modifying the expression of some parabolic $A_n$ part of the representations, the same argument as in Theorem \ref{mainthmAn} shows that the resulting representation is still irreducible by restricting to the variables corresponding to the parabolic part.

Hence, we state the Main Theorem of the paper.

\begin{Thm}\label{mainthm} We have a homomorphism of $\cU_q(\g)$ into a certain skew-symmetrizable quantum cluster algebra $\cO_q(\cX^0)$, obtained by identifying the symmetric part of the type $A$ parabolic subgroup $W_J$, such that the image of the Chevalley generators are universally Laurent polynomials in $\cO_q(\cX^0)$.

Upon choosing a polarization for any cluster chart of $\cO_q(\cX^0)$ produces a family of irreducible positive representations of $\cU_q(\g_\R)$ parametrized by $n-|J|$ central parameters $\l_i\in\R$.
\end{Thm}

\subsection{Maximal degenerate representations}\label{sec:higher:max}
In this section, we focus on the calculation of the action of the Casimir operators $\bC_k$ for general Lie types, by folding a parabolic $A_{n-1}$ part of maximal rank according to the previous section. Since this reduces the rank of the quiver by $n-1$, the resulting polarization of the skew-symmetrizable quantum torus algebra $\cX_q^0$ will be parametrized by a single scalar $\l\in\R_{\geq0}$ which can be taken to be positive. We call this family of positive representations the \emph{maximal degenerate representations}.

The calculation of the actions of the Casimir operators relies on the explicit expression of the generating monomials of the center of the quantum torus algebra $\cX_q^{\mathrm{sym}}$, how the central parameters are changed under the folding, and the fact that two polarizations are unitarily equivalent if and only if their actions on the central monomials coincide. We will focus on the representations of the quantum group $\cU_q(\g_\R)$, so that in particular our polarizations are always group-like, i.e. $\pi_\l(K_iK_i')=1$ for $i\in I$.

In the following, we illustrate the calculation using type $B_5$ as an example. The general strategy for the calculation of $\bC_k$ in other Lie types follows with appropriate modifications, and are explained in the proof of Theorem \ref{thmother}.

\subsubsection{Step 1: Central monomials of $\cX_q^{\mathrm{sym}}$}\label{sec:higher:sym}
We consider the parabolic part $A_{n-1}$ corresponding to all the long roots of $B_n$. With our labeling (see Appendix \ref{App:Dyn}), they are indexed by $J=\{1,2,...,n-1\}$. It is known that the reduced expression of $w_0=w_Jw'$ can be written as
\Eqn{
\bi_0&=(1\;2\;3\;\cdots\; n)^n\\
&=(1\;2\;3\;\cdots\; n)(1\;2\;3\cdots\;n)\cdots (1\;2\;3\;\cdots n)\\
&=(1\;2\;1\;3\;2\;1\cdots n-1\cdots 1)(n\;n-1\;n\;n-2\;n-1\;n\cdots 1\cdots n)\\
&=:\bi_{A_{n-1}}\bi'.
}
As in \cite{Ip8}, the quiver corresponding to the quantum torus algebra $\cX_q^{\mathrm{std}}$ can be rearranged as in Figure \ref{fig-B5} such that the shaded part consists of the parabolic $A_{n-1}$ subquiver. The colors indicate the embedding of the Chevalley generators, which can be obtained from the parabolic representations described explicitly in \cite{Ip8}, but otherwise are not really relevant to the discussion below.

\begin{figure}[htb!]
\centering
\begin{tikzpicture}[every node/.style={inner sep=0, minimum size=0.2cm, thick}, x=0.8cm, y=0.5cm]
\fill[green!10] (4,-10)--(4,-8)--(4,-6)--(4,-4)--(8,-0)--(12,-4)--(12,-6)--(12,-8)--(12,-10)--(8,-14)--cycle;
\node (1) at (0,-16)[draw,]{};
\node (2) at (1,-16)[draw,circle,fill=red]{};
\node (3) at (2,-16)[draw,circle,fill=red]{};
\node (4) at (3,-16)[draw,circle,fill=red]{};
\node (5) at (4,-16)[draw,circle,fill=red]{};
\node (6) at (8,-16)[draw,circle,fill=red]{};
\node (7) at (12,-16)[draw,circle,fill=red]{};
\node (8) at (13,-16)[draw,circle,fill=red]{};
\node (9) at (14,-16)[draw,circle,fill=red]{};
\node (10) at (15,-16)[draw,circle,fill=red]{};
\node [label={[xshift=1em, yshift=-1.1em] \tiny $-2\l_5$}](11) at (16,-16)[draw,]{};
\node (12) at (0,-10)[draw,]{};
\node (13) at (1,-10)[draw,circle,fill=red]{};
\node (14) at (2,-10)[draw,circle,fill=red]{};
\node (15) at (3,-10)[draw,circle,fill=red]{};
\node (16) at (4,-10)[draw,circle,fill=red]{};
\node (17) at (8,-14)[draw,circle]{};
\node (18) at (12,-10)[draw,circle,fill=red]{};
\node (19) at (13,-10)[draw,circle,fill=red]{};
\node (20) at (14,-10)[draw,circle,fill=red]{};
\node (21) at (15,-10)[draw,circle,fill=red]{};
\node [label={[xshift=1em, yshift=-1.1em] \tiny $-2\l_4$}](22) at (16,-10)[draw,]{};
\node (23) at (0,-8)[draw,]{};
\node (24) at (1,-8)[draw,circle,fill=red]{};
\node (25) at (2,-8)[draw,circle,fill=red]{};
\node (26) at (4,-8)[draw,circle,fill=red]{};
\node (27) at (5,-9)[draw,circle]{};
\node (28) at (8,-12)[draw,circle]{};
\node (29) at (11,-9)[draw,circle]{};
\node (30) at (12,-8)[draw,circle,fill=red]{};
\node (31) at (14,-8)[draw,circle,fill=red]{};
\node (32) at (15,-8)[draw,circle,fill=red]{};
\node [label={[xshift=1em, yshift=-1.1em] \tiny $-2\l_3$}](33) at (16,-8)[draw,]{};
\node (34) at (0,-6)[draw,]{};
\node (35) at (1,-6)[draw,circle,fill=red]{};
\node (36) at (4,-6)[draw,circle,fill=red]{};
\node (37) at (5,-7)[draw,circle]{};
\node (38) at (6,-8)[draw,circle]{};
\node (39) at (8,-10)[draw,circle]{};
\node (40) at (10,-8)[draw,circle]{};
\node (41) at (11,-7)[draw,circle]{};
\node (42) at (12,-6)[draw,circle,fill=red]{};
\node (43) at (15,-6)[draw,circle,fill=red]{};
\node [label={[xshift=1em, yshift=-1.1em] \tiny $-2\l_2$}](44) at (16,-6)[draw,]{};
\node (45) at (0,-4)[draw,]{};
\node (46) at (4,-4)[draw,circle,fill=red]{};
\node (47) at (5,-5)[draw,circle]{};
\node (48) at (6,-6)[draw,circle]{};
\node (49) at (7,-7)[draw,circle]{};
\node (50) at (8,-8)[draw,circle]{};
\node (51) at (9,-7)[draw,circle]{};
\node (52) at (10,-6)[draw,circle]{};
\node (53) at (11,-5)[draw,circle]{};
\node (54) at (12,-4)[draw,circle,fill=red]{};
\node [label={[xshift=1em, yshift=-1.1em] \tiny $-2\l_1$}](55) at (16,-4)[draw,]{};
\node [label={[xshift=0em, yshift=0em] \tiny $4\l_5$}](56) at (8,-18)[draw,circle,fill=red]{};
\node [label={[xshift=0em, yshift=0em] \tiny $4\l_4$}](57) at (8,-0)[draw,circle]{};
\node [label={[xshift=0em, yshift=0em] \tiny $4\l_3$}](58) at (8,-2)[draw,circle]{};
\node [label={[xshift=0em, yshift=0em] \tiny $4\l_2$}](59) at (8,-4)[draw,circle]{};
\node [label={[xshift=0em, yshift=0em] \tiny $4\l_1$}](60) at (8,-6)[draw,circle]{};
\drawpath{1,...,11}{blue}
\drawpath{12,...,22}{blue,thick}
\drawpath{23,...,33}{blue,thick}
\drawpath{34,...,44}{blue,thick}
\drawpath{45,...,55}{blue,thick}
\drawpath{11,56,1}{brown}
\drawpath{2,56,10}{}
\drawpath{22,10,21,32,43,54,57,46,35,24,13,2,12}{red,thick}
\drawpath{33,21,9,20,31,42,53,58,47,36,25,14,3,13,23}{orange,thick}
\drawpath{44,32,20,8,19,30,41,52,59,48,37,26,15,4,14,24,34}{purple,thick}
\drawpath{55,43,31,19,7,18,29,40,51,60,49,38,27,16,5,15,25,35,45}{green,thick}
\drawpath{39,49,59,51,39}{thick}
\drawpath{28,38,48,58,52,40,28}{thick}
\drawpath{17,27,37,47,57,53,41,29,17}{thick}
\drawpath{54,42,30,18,6,16,26,36,46}{thick}
\drawpath{45,34,23,12,1}{dashed,thick}
\drawpath{11,22,33,44,55}{dashed,thick}
\end{tikzpicture}
\caption{The quiver for the standard quiver of $\cX_q^{\mathrm{std}}$ in type $B_5$, with the vertices for the central monomial $Q_5$ colored in red.}\label{fig-B5}
\end{figure}
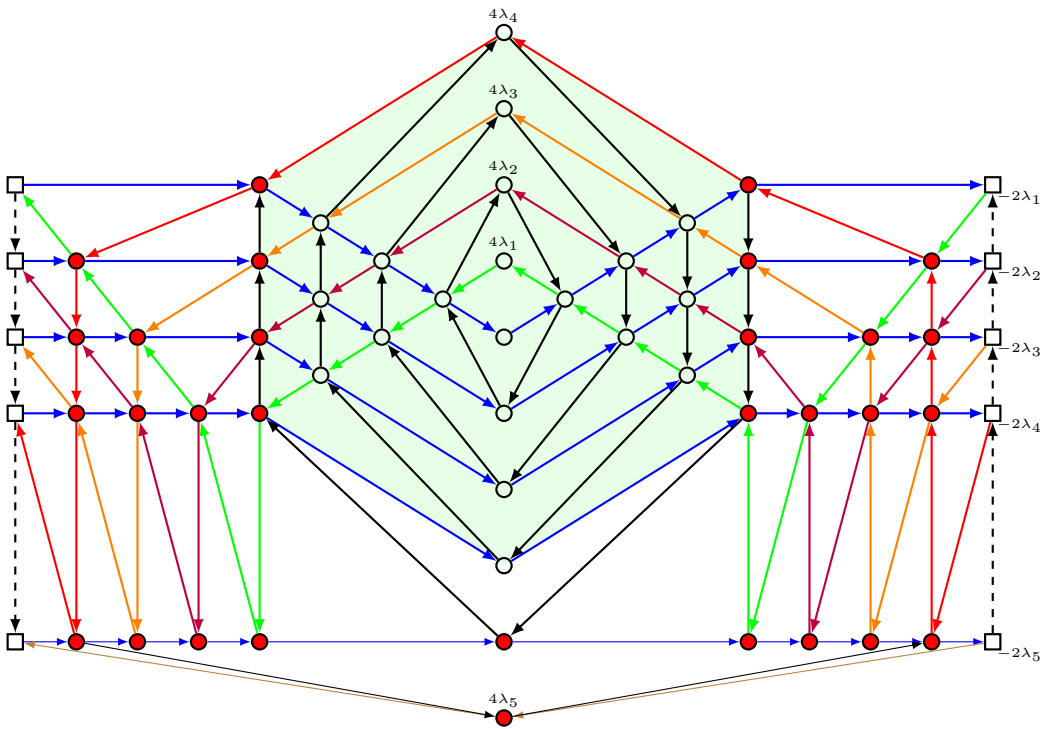

It is important to note that, by the results from the parabolic positive representations \cite{Ip8}, the indexing of the parabolic part $A_{n-1}$ requires a twisting by its Dynkin involution, i.e. the change in indices given by $i\mapsto i^*=n-i$. Therefore, the assignments of the central parameters are reversed compared with Figure \ref{fig-A4}.
\begin{Prop}
The center of $\cX_q^{\mathrm{std}}$ is generated by the following monomials:
\begin{itemize}
\item $\iota(\bK_i \bK_i')$, $i\in I$, which is the product of the cluster variables along each distinct blue path for the $\bf_i$ generators, and each distinct colored paths for the $\be_i$ generators.
\item $Q_k$, $k=1,...,n-1$, with central parameter $4\ol_k$, which coincides with the center of the parabolic $A_{n-1}$ part as in \eqref{Qk}.
\item $Q_n$, given by the product of the cluster variables associated to the red nodes, i.e. all the mutable vertices that are not within the interior of the parabolic part. It has central parameter $4\ol_n$. 
\end{itemize}
\end{Prop}
Note that the index $n\in I$ is \emph{short}, so $4\ol_n = 4b_s \l_n$ as in Notation \ref{notbi} and \ref{ol}.

Now we perform the folding construction for the parabolic $A_{n-1}$ part and obtain the quiver as in Figure \ref{fig-B5sym} with the corresponding assignments of the central parameters (again noting the twisting of index).
\begin{figure}[htb!]
\centering
\begin{tikzpicture}[every node/.style={inner sep=0, minimum size=0.2cm, thick}, x=0.7cm, y=1cm]
\fill[green!10] (11,-12)--(12,-11)--(13,-10)--(14,-9)--(11,-6)--(10,-1)--(9,-6)--(6,-9)--(7,-10)--(8,-11)--(9,-12)--cycle;
\node [label={[xshift=0em, yshift=0em] \tiny $4\l_5$}](56) at (10,-14)[draw,circle, fill=red]{};
\node (1) at (0,-13)[draw,]{};
\node (2) at (2,-13)[draw,circle, fill=red]{};
\node (3) at (4,-13)[draw,circle, fill=red]{};
\node (4) at (6,-13)[draw,circle, fill=red]{};
\node (5) at (8,-13)[draw,circle, fill=red]{};
\node (6) at (10,-13)[draw,circle, fill=red]{};
\node (7) at (12,-13)[draw,circle, fill=red]{};
\node (8) at (14,-13)[draw,circle, fill=red]{};
\node (9) at (16,-13)[draw,circle, fill=red]{};
\node (10) at (18,-13)[draw,circle, fill=red]{};
\node [label={[xshift=1em, yshift=-1.1em] \tiny $-2\l_5$}](11) at (20,-13)[draw,]{};
\node (12) at (0,-12)[draw,]{};
\node (13) at (3,-12)[draw,circle, fill=red]{};
\node (14) at (5,-12)[draw,circle, fill=red]{};
\node (15) at (7,-12)[draw,circle, fill=red]{};
\node (16) at (9,-12)[draw,circle, fill=red]{};
\node (18) at (11,-12)[draw,circle, fill=red]{};
\node (19) at (13,-12)[draw,circle, fill=red]{};
\node (20) at (15,-12)[draw,circle, fill=red]{};
\node (21) at (17,-12)[draw,circle, fill=red]{};
\node [label={[xshift=1em, yshift=-1.1em] \tiny $-2\l_4$}](22) at (20,-12)[draw,]{};
\node (23) at (0,-11)[draw,]{};
\node (24) at (4,-11)[draw,circle, fill=red]{};
\node (25) at (6,-11)[draw,circle, fill=red]{};
\node (26) at (8,-11)[draw,circle, fill=red]{};
\node (17) at (10,-11)[draw,circle, fill=red]{};
\node (30) at (12,-11)[draw,circle, fill=red]{};
\node (31) at (14,-11)[draw,circle, fill=red]{};
\node (32) at (16,-11)[draw,circle, fill=red]{};
\node [label={[xshift=1em, yshift=-1.1em] \tiny $-2\l_3$}](33) at (20,-11)[draw,]{};
\node (34) at (0,-10)[draw,]{};
\node (35) at (5,-10)[draw,circle, fill=red]{};
\node (36) at (7,-10)[draw,circle, fill=red]{};
\node (27) at (9,-10)[draw,circle, fill=red]{};
\node (29) at (11,-10)[draw,circle, fill=red]{};
\node (42) at (13,-10)[draw,circle, fill=red]{};
\node (43) at (15,-10)[draw,circle, fill=red]{};
\node [label={[xshift=1em, yshift=-1.1em] \tiny $-2\l_2$}](44) at (20,-10)[draw,]{};
\node (45) at (0,-9)[draw,]{};
\node (46) at (6,-9)[draw,circle, fill=red]{};
\node (37) at (8,-9)[draw,circle, fill=red]{};
\node (28) at (10,-9)[draw,circle, fill=red]{};
\node (41) at (12,-9)[draw,circle, fill=red]{};
\node (54) at (14,-9)[draw,circle, fill=red]{};
\node [label={[xshift=1em, yshift=-1.1em] \tiny $-2\l_1$}](55) at (20,-9)[draw,]{};
\node [label={[xshift=0em, yshift=0em] \tiny $4\l_4$}](47) at (7,-8)[draw,circle, fill=red]{};
\node (38) at (9,-8)[draw,circle, fill=red]{};
\node (40) at (11,-8)[draw,circle, fill=red]{};
\node (53) at (13,-8)[draw,circle, fill=red]{};

\node [label={[xshift=0em, yshift=0em] \tiny $4\l_3$}](39) at (8,-7)[draw,circle, fill=red]{};
\node (52) at (10,-7)[draw,circle, fill=red]{};
\node (57) at (12,-7)[draw,circle, fill=red]{};
\node [label={[xshift=-0.6em, yshift=0em] \tiny $4\l_2$}](58) at (9,-6)[draw,circle, fill=red]{};
\node (48) at (11,-6)[draw,circle, fill=red]{};

\node [label={[xshift=2.5em, yshift=-0.6em] \tiny $4\l_1$}](60) at (10,-5)[draw,circle]{};
\node (50) at (10,-4)[draw,circle]{};
\node [label={[xshift=2em, yshift=-0.6em] \tiny $-4\l_2$}](59) at (10,-3)[draw,circle]{};
\node [label={[xshift=2.5em, yshift=-0.6em] \tiny $-4\l_2-4\l_3$}](49) at (10,-2)[draw,circle]{};
\node [label={[xshift=3.2em, yshift=-0.6em] \tiny $-4\l_2-4\l_3-4\l_4$}](51) at (10,-1)[draw,circle, fill=red]{};
\drawpath{1,...,11}{blue}
\drawpath{12,13,14,15,16,18,19,20,21,22}{blue,thick}
\drawpath{23,24,25,26,17,30,31,32,33}{blue,thick}
\drawpath{34,35,36,27,29,42,43,44}{blue,thick}
\drawpath{45,46,37,28,41,54,55}{blue,thick}
\drawpath{11,56,1}{brown}
\drawpath{2,56,10}{}
\drawpath{54,42,30,18,6,16,26,36,46}{thick}
\drawpath{22,10,21,32,43,54,53}{red,thick}
\drawpath{53,47}{red,thick,bend right = 10}
\drawpath{47,46,35,24,13,2,12}{red,thick}
\drawpath{33,21,9,20,31,42,41,40,57}{orange,thick}
\drawpath{57,39}{orange,thick,bend right=10}
\drawpath{39,38,37,36,25,14,3,13,23}{orange, thick}
\drawpath{44,32,20,8,19,30,29,28,38,52,48}{purple, thick}
\drawpath{48,58}{purple, thick, bend right=10}
\drawpath{58,52,40,28,27,26,15,4,14,24,34}{purple, thick}
\drawpath{55,43,31,19,7,18,17,27,37,47,57,52,39,48,58,60,48,57,53,41,29,17,16,5,15,25,35,45}{green, thick}
\drawpath{58,50,48}{green,thick}
\drawpath{58,59,48}{green,thick}
\drawpath{58,49,48}{green,thick}
\drawpath{58,51,48}{green,thick}
\drawpath{45,34,23,12,1}{dashed,thick}
\drawpath{11,22,33,44,55}{dashed,thick}
\end{tikzpicture}
\caption{The quiver for $\cX_q^{\mathrm{sym}}$ in type $B_5$, with the vertices for the central monomial $Q_5$ colored in red.}\label{fig-B5sym}
\end{figure}
We label the symmetric nodes by $X_{c_0},...,X_{c_{n-1}}$ from top to bottom as before. Under this folding, the central monomials $Q_k$ for $k\neq n$ are given by the ratio 
\Eq{
Q_k=X_{c_{n-k}}X_{c_{n-k-1}}\inv,\tab k=1,...,n-1
}
as before, while by Lemma \ref{outer}, $Q_n$ is the product of the cluster variables (colored in red) associated to all the mutable variables not from the symmetric part, together with the single variable $X_{c_0}$.

\subsubsection{Step 2: Central parameters of $\cX_q^0$}\label{sec:higher:central}
We fold the parabolic $A_{n-1}$ part along the symmetric nodes as before, with the new cluster variable given by
\Eq{\label{X*}
X_\star:=\prod_{k=0}^{n-1}X_{c_k}
,} and obtain a skew-symmetrizable quantum torus algebra $\cX_q^0$ with the same rank but $n-1$ less variables. In particular, the non-Cartan central monomials are parametrized by a single real number $\l\in\R$. The quiver is shown in Figure \ref{fig-B5X0}.

By definition, the multipliers of the bottom-most vertex and vertices along the bottom row are $d=\half$, the folded vertex $d_\star=n$, and the remaining vertices $d=1$. In particular, the central parameter $\l$ is a \emph{short} weight.

\begin{figure}[htb!]
\centering
\begin{tikzpicture}[every node/.style={inner sep=0, minimum size=0.2cm, thick}, x=0.7cm, y=1cm]
\fill[green!10] (11,-12)--(12,-11)--(13,-10)--(14,-9)--(11,-6)--(10,-5)--(9,-6)--(6,-9)--(7,-10)--(8,-11)--(9,-12)--cycle;
\node [label={[xshift=0em, yshift=0em] \tiny $4\l$}](56) at (10,-14)[draw,circle, fill=red]{};
\node (1) at (0,-13)[draw,]{};
\node (2) at (2,-13)[draw,circle, fill=red]{};
\node (3) at (4,-13)[draw,circle, fill=red]{};
\node (4) at (6,-13)[draw,circle, fill=red]{};
\node (5) at (8,-13)[draw,circle, fill=red]{};
\node (6) at (10,-13)[draw,circle, fill=red]{};
\node (7) at (12,-13)[draw,circle, fill=red]{};
\node (8) at (14,-13)[draw,circle, fill=red]{};
\node (9) at (16,-13)[draw,circle, fill=red]{};
\node (10) at (18,-13)[draw,circle, fill=red]{};
\node [label={[xshift=0em, yshift=-1.4em] \tiny $-2\l$}](11) at (20,-13)[draw,]{};
\node (12) at (0,-12)[draw,]{};
\node (13) at (3,-12)[draw,circle, fill=red]{};
\node (14) at (5,-12)[draw,circle, fill=red]{};
\node (15) at (7,-12)[draw,circle, fill=red]{};
\node (16) at (9,-12)[draw,circle, fill=red]{};
\node (18) at (11,-12)[draw,circle, fill=red]{};
\node (19) at (13,-12)[draw,circle, fill=red]{};
\node (20) at (15,-12)[draw,circle, fill=red]{};
\node (21) at (17,-12)[draw,circle, fill=red]{};
\node (22) at (20,-12)[draw,]{};
\node (23) at (0,-11)[draw,]{};
\node (24) at (4,-11)[draw,circle, fill=red]{};
\node (25) at (6,-11)[draw,circle, fill=red]{};
\node (26) at (8,-11)[draw,circle, fill=red]{};
\node (17) at (10,-11)[draw,circle, fill=red]{};
\node (30) at (12,-11)[draw,circle, fill=red]{};
\node (31) at (14,-11)[draw,circle, fill=red]{};
\node (32) at (16,-11)[draw,circle, fill=red]{};
\node (33) at (20,-11)[draw,]{};
\node (34) at (0,-10)[draw,]{};
\node (35) at (5,-10)[draw,circle, fill=red]{};
\node (36) at (7,-10)[draw,circle, fill=red]{};
\node (27) at (9,-10)[draw,circle, fill=red]{};
\node (29) at (11,-10)[draw,circle, fill=red]{};
\node (42) at (13,-10)[draw,circle, fill=red]{};
\node (43) at (15,-10)[draw,circle, fill=red]{};
\node (44) at (20,-10)[draw,]{};
\node (45) at (0,-9)[draw,]{};
\node (46) at (6,-9)[draw,circle, fill=red]{};
\node (37) at (8,-9)[draw,circle, fill=red]{};
\node (28) at (10,-9)[draw,circle, fill=red]{};
\node (41) at (12,-9)[draw,circle, fill=red]{};
\node (54) at (14,-9)[draw,circle, fill=red]{};
\node (55) at (20,-9)[draw,]{};
\node (47) at (7,-8)[draw,circle, fill=red]{};
\node (38) at (9,-8)[draw,circle, fill=red]{};
\node (40) at (11,-8)[draw,circle, fill=red]{};
\node (53) at (13,-8)[draw,circle, fill=red]{};

\node (39) at (8,-7)[draw,circle, fill=red]{};
\node (52) at (10,-7)[draw,circle, fill=red]{};
\node (57) at (12,-7)[draw,circle, fill=red]{};
\node (58) at (9,-6)[draw,circle, fill=red]{};
\node (48) at (11,-6)[draw,circle, fill=red]{};

\node (60) at (10,-5)[draw,circle, fill=red]{$\star$};
\drawpath{1,...,11}{blue}
\drawpath{12,13,14,15,16,18,19,20,21,22}{blue,thick}
\drawpath{23,24,25,26,17,30,31,32,33}{blue,thick}
\drawpath{34,35,36,27,29,42,43,44}{blue,thick}
\drawpath{45,46,37,28,41,54,55}{blue,thick}
\drawpath{11,56,1}{brown}
\drawpath{2,56,10}{}
\drawpath{54,42,30,18,6,16,26,36,46}{thick}
\drawpath{22,10,21,32,43,54,53}{red,thick}
\drawpath{53,47}{red,thick,bend right = 10}
\drawpath{47,46,35,24,13,2,12}{red,thick}
\drawpath{33,21,9,20,31,42,41,40,57}{orange,thick}
\drawpath{57,39}{orange,thick,bend right=10}
\drawpath{39,38,37,36,25,14,3,13,23}{orange, thick}
\drawpath{44,32,20,8,19,30,29,28,38,52,48}{purple, thick}
\drawpath{48,58}{purple, thick, bend right=10}
\drawpath{58,52,40,28,27,26,15,4,14,24,34}{purple, thick}
\drawpath{55,43,31,19,7,18,17,27,37,47,57,52,39,48,58}{green, thick}
\drawpath{48,57,53,41,29,17,16,5,15,25,35,45}{green, thick}
\drawpath{58,60,48}{green,vthick}
\drawpath{45,34,23,12,1}{dashed,thick}
\drawpath{11,22,33,44,55}{dashed,thick}
\end{tikzpicture}
\caption{The quiver for $\cX_q^0$ in type $B_5$, with the vertices for the central monomial $Q_\star$ colored in red.}\label{fig-B5X0}
\end{figure}
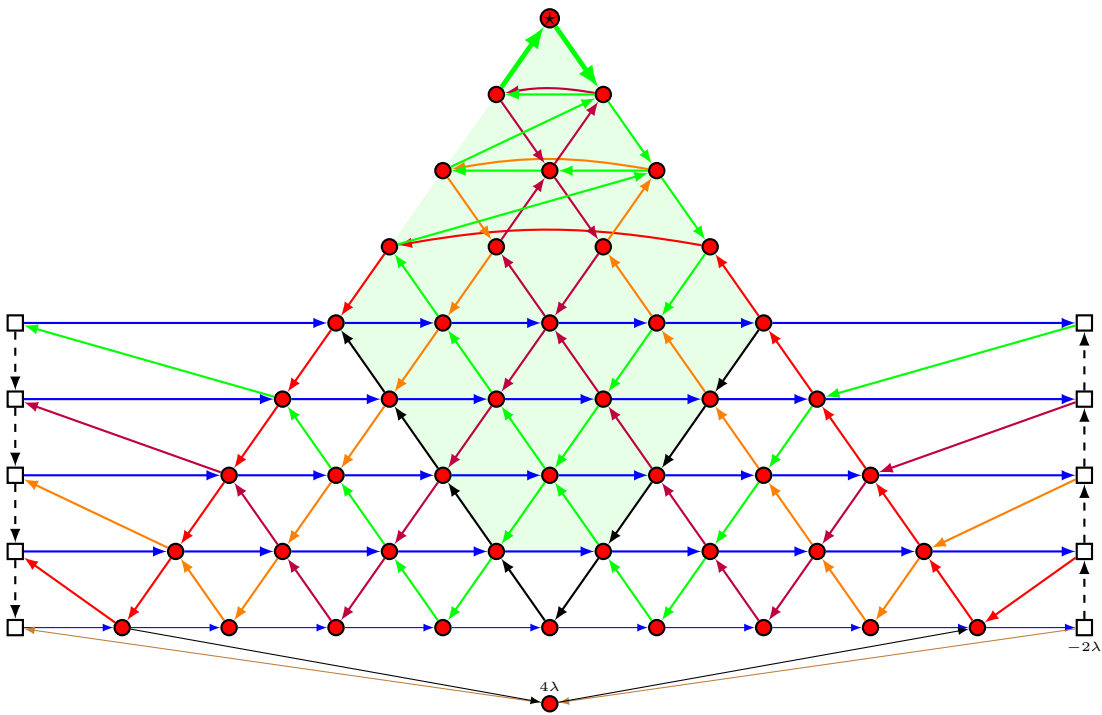

We wish to compute the action of the Casimir operators associated to this polarization. In order to do this, we try to find an equivalent description of the central parameters by comparing them with a specification of complex values to the central parameters of the unfolded quiver as in Section \ref{sec:An:pos0}.

Recall that under the polarization of the unfolded quiver above, the effect of folding is equivalent to setting $\l_1,...,\l_{n-1}$ to the general solution of $\eqref{symeq}$, and is independent of the choice of the solutions. We choose the standard solution $\l_1=\cdots =\l_{n-1}=\frac{\bi}{2nb}$. At the same time, the resulting polarization of $X_\star$ is obtained by taking the product of all those of the symmetric vertices, with their central parameters adding up. The sum of these central parameters (using the rescaled variables) is given by
\Eqn{
&4\ol_1+0+(-4\ol_2)+(-4\ol_2-4\ol_3)+\cdots + (-4\ol_2+\cdots -4\ol_{n-1})\\
&=4\ol_1-4\sum_{k=2}^{n-1}(n-k)\ol_k.
}
By setting the parameters to complex numbers, the resulting polarization is certainly not positive. However, note that a linear change of variables in the parameters does not change the action of the center. Thus, it suffices to find a specialization of the parameters, complex or not, in such a way that the central parameters of the central monomials coincide with the positive ones chosen as in Figure \ref{fig-B5X0}.

We know that the central monomial $Q_n$ of the original unfolded algebra $\cX_q^{\mathrm{sym}}$ is a monomial which, among other variables, contains a single cluster variable $X_{c_0}$. Since the Weyl part of the polarization of all the $X_{c_i}$'s are the same, by \eqref{X*} the polarization of $X_{c_0}^n$ and $X_\star$ coincide. From this, we easily see that the monomial given by
\Eq{
Q_\star:=Q_n X_{c_0}\inv X_\star^{\frac{1}{n}}
}
formally lies in the center of $\cX_q^0$, if we allow fractional powers of the cluster variables. It still makes sense to talk about its polarization, and $Q_\star$ will be the unique central monomial carrying a nontrivial central parameter.

By the folding, we compute the central parameter of $Q_\star$ to be
\Eqn{
&4\ol_n- (-4\ol_2-4\ol_3-\cdots -4\ol_{n-1}) + \frac{1}{n}\left(4\ol_1- 4b\sum_{k=2}^{n-1} (n-k)\l_k\right)\\
&=\frac{4}{n}\sum_{k=1}^nk\ol_k.
}

On the other hand, the central parameter of $Q_\star$ is given by $4\ol:=4b_s\l$ from the polarization in Figure \ref{fig-B5X0}. By setting \Eq{4\ol = \frac{4}{n}\sum_{k=1}^n k\ol_k} and substituting the general solution, we conclude that
\begin{Thm}
The polarization of the central monomials of $\cX_q^0$ is given by setting the parameters of $\cX_q^{\mathrm{sym}}$ to be
\Eq{
\l_n&=\l-\bi\h_{B_n}\label{subl1}\\
\l_k&=\frac{\bi}{2nb},&& k=1,...,n-1\label{subl2}
}
where the \emph{complex shift} is given by
\Eq{
\h_{B_n}:= \frac{n-1}{4nb_s}.\label{subl3}
}
Moreover, $Q_\star$ acts as the multiplication by a positive scalar.
\end{Thm}

\subsubsection{Step 3: Action of Casimir operators}\label{sec:higher:cas}
Using the equivalence of the central characters of the folded quantum cluster algebra with the specification of complex values to the central parameters of the original quiver, we can now compute the action of the Casimir operators by substituting \eqref{subl1}--\eqref{subl3} to the explicit expressions of $\bC_k$ in \eqref{CasActEq}.

We shall compute this directly by using the generating function of the positive Casimirs derived in \cite{Ip5}, together with the explicit form of the finite dimensional fundamental representations of type $B_n$ Lie algebra, which consist of the spin representation $V_n=S$, the standard representation $V_1=V$ and its exterior products $V_k:=\L^k V$ for $k=1,...,n-1$.

\begin{Thm} Under any irreducible polarization of $\cX_q^0$ with central parameters as in Figure \ref{fig-B5X0}, the Casimir operators of the maximal degenerate representations $\cP_\l^{0,J}$ of type $B_n$ act as multiplication by the following scalars:
\Eq{
\pi_\l(\bC_1)&=1\nonumber\\
\pi_\l(\bC_2)&=0\nonumber\\
\vdots\\
\pi_\l(\bC_{n-1})&=0\nonumber\\
\pi_\l(\bC_n)&=e^{2\pi nb_s\l}+e^{-2\pi nb_s\l}\nonumber.
}
In particular, we have $\cP_\l^{0,J}\simeq \cP_{-\l}^{0,J}$, so we can restrict to positive parameter $\l\in\R_{\geq 0}$.
\end{Thm}
\begin{proof}
By the explicit weight decomposition of the spin module \cite{Car} and \eqref{CasActEq}, we have
\Eqn{
\pi_\l(\bC_n)&=\prod_{i=1}^{n}\left(\prod_{j=n+1-i}^n e^{2\pi \ol_j} +\prod_{j=n+1-i}^n e^{-2\pi \ol_j} \right)\\
&=e^{-2\pi \sum_{k=1}^{n} k\ol_k}\prod_{i=1}^n\left(1+\prod_{j=n+1-i}^n e^{4\pi \ol_j} \right).
}
Upon setting $\l_1=\cdots=\l_{n-1}=\frac{\bi}{2nb}$ and $\l_n=\l-\frac{n-1}{4nb_s}\bi$, we obtain
\Eq{\label{CnProd}
\pi_\l(\bC_n)=e^{-2\pi nb_s\l}\prod_{k=1}^n \left(1- e^{4\pi b_s\l}\xi_n^{\frac{2k-1}{2}}\right)
}
where $\xi_n=e^{\frac{2\pi \bi}{n}}$ is a primitive $n$-th root of unity. Using the factorization
$$t^n+a^n=\prod_{k=1}^n \left(t-a \xi_n^{\frac{2k-1}{2}}\right)$$ and substituting $t=1$, the product equals $$\pi_\l(\bC_n)=e^{-2\pi nb_s\l}(1+e^{4\pi nb_s\l})=e^{2\pi nb_s\l}+e^{-2\pi nb_s\l}$$ as required.

We can also observe directly from the computation above that choosing $\l_i$ to be other general solutions of \eqref{symeq} amounts to permuting the factors in \eqref{CnProd}, which does not affect the final result.

The other Casimirs are obtained from the generating function of the weights of exterior product representations
\Eq{
\label{pt}
p(t)=(1+t)\prod_{i=1}^n\left( \left(t+\prod_{j=n+1-i}^n e^{4\pi \ol_j}\right) \left(t+\prod_{j=n+1=i}^n e^{-4\pi \ol_j}\right) \right),
}
which can be rewritten as
$$p(t)=(1+t^{2n+1})+(t+t^{2n})\bC_1+(t^2+t^{2n-1})\bC_2+\cdots + (t^n+t^{n+1})\bC$$
where $$\bC=\bC_n^2-\bC_{n-1}-\cdots-\bC_2-\bC_1-1$$
that follows from the decomposition of the fundamental representations:
\Eq{
S\ox S\simeq 1\o+V\o+\L^2 V\o+\cdots \o+\L^n V.
}

Exactly as in \eqref{CnProd}, upon setting $\l_n=\l-\frac{n-1}{4nb_s}\bi$ and $\l_1=\cdots=\l_{n-1}=\frac{\bi}{2nb}$, we compute that
$$
\prod_{i=1}^n \left(t+\prod_{j=n+1-i}^n e^{\pm 4\pi \ol_j} \right)=t^n+e^{\pm 4\pi nb_s\l}$$
and hence by \eqref{pt},
\Eqn{
p(t)&=(1+t)(t^n+e^{4\pi nb_s\l})(t^n+e^{-4\pi nb_s\l})\\
&=1+t+(e^{4\pi nb_s\l}+e^{-4\pi nb_s\l})(t^n+t^{n+1})+t^{2n}+t^{2n+1}.
}
By comparing coefficients, we obtain 
\Eqn{
\pi_\l(\bC_1)=1\\
\pi_\l(\bC_2)=\cdots = \pi_\l(\bC_{n-1})=0.
}
We verify that indeed $\pi_\l(\bC)=e^{4\pi nb_s\l}+e^{-4\pi nb_s\l}=\pi_\l(\bC_n)^2-2$. 
\end{proof}
This completes the analysis of the maximal degenerate representations of type $B_n$.

\subsubsection{Results for general Lie types}\label{sec:higher:other}
Using the algorithm in the previous subsections, we compute the actions of $\bC_k$ of the maximal degenerate representations of other Lie types (except type $F_4$ and $G_2$ which are treated separately) with respect to the parabolic part $A_{n-1}$ with root index $1,...,n-1$ labeled as in Appendix \ref{App:Dyn}. We compute the complex shift $\h_\g$ by analyzing the central monomials $Q_n$ under the symmetric folding as in subsection \ref{sec:higher:central}, and substitute the corresponding complexified central parameters to \eqref{CasActEq} in order to obtain the action of $\bC_k$ as in subsection \ref{sec:higher:cas}.

For completeness, we also include type $A_{n-1}\subset A_n$ which was discussed previously in Lemma \ref{outer} and is easy to deal with using the explicit expressions of the actions of $\bC_k$ as elementary symmetric polynomials \eqref{casact}.

\begin{Thm}\label{thmother} The complex shifts, together with the actions of the generalized Casimirs $\bC_k$ of the maximal degenerate representations in type $A_n$ to $E_n$ are given as follows.

$$\begin{array}{|c|c|l|}
\hline
\emph{Type}&\emph{Complex Shifts}&\multicolumn{1}{|c|}{\emph{Action of Casimirs}}\\
\hline\hline
A_n, n\geq 2& \h_{A_n}=\frac{n-1}{4nb}& \begin{array}{rl}\pi_\l(\bC_1)&=e^{\frac{4\pi nb}{n+1}\l}\\
\pi_\l(\bC_2)&=\cdots =\pi_\l(\bC_{n-1})=0\\
\pi_\l(\bC_n)&=e^{-\frac{4\pi nb}{n+1}\l}\end{array}\\\hline
B_n& \h_{B_n}=\frac{n-1}{4nb_s}&\begin{array}{rl}\pi_\l(\bC_1)&=1\\
\pi_\l(\bC_2)&=\cdots =\pi_\l(\bC_{n-1})=0\\
\pi_\l(\bC_n)&=e^{2n\pi b_s\l}+e^{-2n\pi b_s\l}\end{array}\\\hline
C_2&\multirow{4}{*}{$\h_{C_n}=\frac{n-1}{2nb}$}& \begin{array}{rl}\pi_\l(\bC_1)&=0\\\pi_\l(\bC_2)&=e^{4\pi b\l}+e^{-4\pi b\l}-1\end{array}\\\cline{1-1}\cline{3-3}
C_n, n\geq3&& \begin{array}{rl}\pi_\l(\bC_1)&=\pi_\l(\bC_3)=\pi_\l(\bC_4)=\cdots=\pi_\l(\bC_{n-1})=0\\\pi_\l(\bC_2)&=-1\\\pi_\l(\bC_n)&=e^{2\pi nb\l}+e^{-2\pi nb\l}\end{array}\\\hline
D_n, n=2k& \multirow{4}{*}{$\h_{D_n}=\frac{n-2}{2nb}$}& \begin{array}{rl}\pi_\l(\bC_1)&=\cdots=\pi_\l(\bC_{n-1})=0\\\pi_\l(\bC_n)&=e^{\pi n b\l}+e^{-\pi n b \l}\end{array}\\\cline{1-1}\cline{3-3}
D_n, n=2k+1& & \begin{array}{rl}\pi_\l(\bC_1)&=\cdots=\pi_\l(\bC_{n-2})=0\\\pi_\l(\bC_{n-1})&=e^{-\pi n b\l}\\\pi_\l(\bC_n)&=e^{\pi nb\l}\end{array}\\\hline
E_6& \multirow{7}{*}{$\h_{E_n}=\frac{3(n-3)}{4nb}$}&\begin{array}{rl}\pi_\l(\bC_1)&=\pi_\l(\bC_2)=\pi_\l(\bC_4)=\pi_\l(\bC_5)=0\\\pi_\l(\bC_3)&=1\\\pi_\l(\bC_6)&=e^{8\pi b\l}+e^{-8\pi b\l}\end{array}\\\cline{1-1}\cline{3-3}
E_7&&\begin{array}{rl}\pi_\l(\bC_1)&=\pi_\l(\bC_2)=\pi_\l(\bC_3)=\pi_\l(\bC_4)=\pi_\l(\bC_6)=0\\\pi_\l(\bC_5)&=-1\\\pi_\l(\bC_7)&=e^{14\pi b\l}+e^{-14\pi b\l}\end{array}\\\cline{1-1}\cline{3-3}
E_8&&\begin{array}{rl}\pi_\l(\bC_2)&=\pi_\l(\bC_3)=\pi_\l(\bC_5)=\pi_\l(\bC_6)=\pi_\l(\bC_7)=0\\\pi_\l(\bC_1)&=-1\\\pi_\l(\bC_4)&=-2\\\pi_\l(\bC_8)&=e^{32\pi b\l}+e^{-32\pi b\l}\end{array}\\\hline
\end{array}$$
\end{Thm}
\begin{proof} We comment on some special aspects of the computation of different types.

\textbf{Type $C_n$.} It is known that the quantum torus algebras corresponding to positive representations of types $B_n$ and $C_n$ are related by Langlands duality, where the long and short weights are interchanged. More precisely, the quantum torus algebra $\cX_q^{C_n}$ is obtained from $\cX_q^{B_n}$ simply by setting the multipliers $\til{d_i}:=\frac{1}{2d_i}$, i.e. interchanging $1\corr\half$. 

The central monomials of $\cX_q^{C_n}$ is then given by the same expression as $\cX_q^{B_n}$, but with all variables rescaled by $X_i\mapsto X_i^{2d_i}$, i.e. all the original ``long" variables are squared. Therefore, when we compute the central monomial $Q_n$ in the folding, the corresponding complex shifts in $\l_1,...,\l_{n-1}$ which correspond to ``short" variables in type $C_n$ are doubled. Hence, the complex shifts are related by 
\Eq{
b\h_{C_n}=2b_s\h_{B_n}.}

The fundamental representations of $C_n$ are, however, quite different from those of $B_n$. They are given by the standard representation $V_1=V$ and the quotients of its exterior products $V_k:=\L^k V/\L^{k-2}V$ for $k=2,...,n$. The corresponding generating function is given by
\Eqn{
p(t)&=\prod_{i=1}^n\left(\left(t+\prod_{j=n+1-i}^n e^{2\pi \ol_j}\right) \left(t+\prod_{j=n+1-i}^n e^{-2\pi\ol_j}\right)\right)\\
&=(1+t^2+\cdots+t^{2n})+\G_1(t+t^{2n-1})+\G_2(t^2+t^{2n-2})+\cdots + \G_n(t^{n-1}+t^{n+1})+\G_1 t^n
}
where
\Eq{\label{Gk}
\G_k:=\sum_{j=0}^{\lfloor \frac{k-1}{2}\rfloor}\bC_{\over{k}+2j}
}
and $\over{k}\cong k\pmod{2} \in\{0,1\}$. Upon setting $\l_n=\l-\frac{n-1}{2nb}\bi$ (which is a long weight) and $\l_1=\cdots=\l_{n-1}=\frac{\bi}{2nb_s}$, we get
$$p(t)=1+(e^{2\pi nb\l}+e^{-2\pi nb\l})t^n+t^{2n}.$$
Comparing the coefficients with \eqref{Gk}, we can solve for the actions of $\bC_k$.\\

\textbf{Type $D_n$.} First, we choose a reduced word of $w_0$ of the form $\bi_0=\bi_{A_{n-1}}\bi'$. This can be computed to be given by
$$\bi_0=(1\;2\;1\;3\;2\;1\cdots n-1\;\cdots 1)(n\; n-2\;n-3\cdots 1)(n-1\; n-2\;n-3\cdots 2)(n\; n-2\cdots 3)\cdots( n-\over{n})$$
where $\over{n}\cong n\pmod{2} \in\{0,1\}$ depending on the parity of $n$. After constructing the corresponding quantum torus algebra, we compute that the central monomial $Q_n$ contains cluster variables from the outer two layers of the parabolic $A_{n-1}$ part of the quiver. Upon mutation to the symmetric quiver, $Q_n$ is transformed into a monomial which only consists of two symmetric vertices:
\Eq{
Q_n = X_{\cdots} X_{c_0}X_{c_1}.
}
Following the previous subsection, after folding, we see that the unique non-Cartan central monomial becomes
\Eq{
Q_\star:=Q_n X_{c_0}\inv X_{c_1}\inv X_\star^{\frac{2}{n}}.
}
The indexing of the central parameters of the parabolic $A_{n-1}$ part may differ depending on the parity of $n$, but does not matter once we set $\l_1=\cdots=\l_{n-1}$. Hence, we compute the central parameter of $Q_\star$ to be
\Eqn{
4\ol&:=4\ol_n-4\sum_{k=2}^{n-1}(-\ol_k)-4\sum_{k=2}^{n-2}(-\ol_k)+\frac{2}{n}\left(4\ol_1-4\sum_{k=2}^{n-1}(n-k)\ol_k\right)\\
&=4\ol_n-4\ol_{n-1}+\frac{8}{n}\sum_{k=1}^{n-1} k\ol_k.
}
By setting $\l_1=\cdots=\l_{n-1}=\frac{\bi}{2nb}$, we obtain the complex shift $\h_{D_n}=\frac{n-2}{2nb}$.

The fundamental representations of $D_n$ comprise the standard $2n$-dimensional module $V_1=V$ together with its exterior products $V_k=\L^k V$, $k=1,...,n-2$, as well as two spin modules $S^+, S^-$. The generating polynomial $p(t)$ for the Casimirs is similar to that of type $B_{n-1}$ but without the $(1+t)$ factor, and with $\l_{n-1}$ replaced by $\l_n-\l_{n-1}$. Similar computation shows that $\pi_\l(\bC_1)=\cdots =\pi_\l(\bC_{n-2})=0$ from the coefficients of $t^k$ in $p(t)$ for $k=1,...,n-2$.

For the spin modules, we have an explicit description of its weight spaces \cite{Car} as follows. The $2^n$ dimensional module $S$ has weight vectors
\Eq{\label{muVe}
\mu_{\cV_\e}:=\frac12 \sum_{i=1}^n \e_i\mu_i
}
where $\e=(\e_1,...,\e_n)$ ranges over all possible choices of signs $\e_i=\pm 1$, and the weights $\mu_i\in H^*$ are given by 
\Eq{\mu=M\w}
where
\Eq{
M=\mat{1\\-1&1\\&-1&1\\&&\ddots&\ddots\\&&&-1&1\\&&&&-1&1&1\\&&&&&-1&1}
}
and $\w=(\w_i)$ are the fundamental weights.

Now we can split the sum so that it gives a decomposition \Eq{
S\simeq S^+\o+ S^-
}
where $S^\pm$ corresponds to the weight spaces with even ($+$) or odd ($-$) number of $-1$'s in the choices of $\e_i$. Then 
\Eq{
\case{V_{n-1}\simeq S^-, V_n\simeq S^+&\mbox{$n$ is even,}\\V_{n-1}\simeq S^+,V_n\simeq S^-&\mbox{$n$ is odd.}}
}
The formula \eqref{CasActEq} together with \eqref{muVe} now allow the computation of the actions of $\bC_{n-1}$ and $\bC_n$, which are explicitly given by
\Eq{
\pi_\l(\bC_\e)=\sum_\e e^{-2\pi b \e\cdot M A\inv \L}
}
summing over all possible choices of signs of $\e=(\e_i)$ with even or odd parity, where $\L=(\l_i)$ are the parameters. Upon setting $\l_1=\cdots=\l_{n-1}=\frac{\bi}{2nb}$ and $\l_n=\l-\frac{n-2}{2nb}\bi$, a direct calculation using root of unity shows that
\begin{itemize}
\item The summation over $\e$ which consists of both $+1$ and $-1$ vanishes.
\item The top term with $\e=(1,...,1)$ gives $e^{-\pi nb\l}$, while the bottom term $\e=(-1,...,-1)$ gives $e^{\pi nb\l}$.
\end{itemize}
Hence, taking into account the description of $S^\pm$ based on the parity of the number of $-1$'s in $\e$, we complete the analysis of $\bC_{n-1}$ and $\bC_n$.\\

\textbf{Type $E_n$.} For type $E_n$, we use the following reduced word $\bi_0=\bi_{A_{n-1}}\bi'$ corresponding to parabolic $A_{n-1}$ with root index $1,...,n-1$:
\Eqn{
E_6:\bi_0=&(1 \;2 1 \;3 2 1 \;4 3 2 1 \;5 4 3 2 1)(6 3 2 1 4 3 2 6 3 4 5 4 3 6 2 3 4 1 2 3 6),\\
E_7:\bi_0=&(1\; 2 1 \;3 2 1 \;4 3 2 1 \;5 4 3 2 1 \;6 5 4 3 2 1)(7 3 2 1 4 3 2 7 3 4 5 4 3 7 2 3 4 1 2 3 7 6 5 4 3 
2 1 7 3 2 4 3 7 5 4 3 2 6 5 4 3 7),\\
E_8:\bi_0=&(1 \;2 1 \;3 2 1 \;4 3 2 1 \;5 4 3 2 1 \;6 5 4 3 2 1 \;7 6 5 4 3 2 1)(8 3 2 1 4 3 2 8 3 4 5 4 3 8 2 3 4 1 2 3 8 6 5 4 3 2 1 8 3 2 4 3 \\
&8 5 4 3 2 6 5 4 3 8 7 6 5 4 3 8 2 3 4 5 6 7 1 2 3 4 5 6 8 3 4 5 2 3 4 8 3 2 1 2 3 8 4 3 2 5 4 3 8 6 5 4 3 2 7 6 5 4 3 8).
}

The central monomial $Q_n$ can be computed explicitly and is shown to depend on the outer 3 layers of the parabolic $A_{n-1}$ part. Upon mutation to the symmetric quiver, $Q_n$ is transformed into the form
\Eq{
Q_n=X_{\cdots} X_{c_0}X_{c_1}X_{c_2}.
}
Again, after folding, the central monomial $Q_\star$ is given by
\Eq{
Q_\star:=Q_n X_{c_0}\inv X_{c_1}\inv X_{c_2}\inv X_\star^{\frac{3}{n}}.
}
Upon substituting $\l_1=\cdots =\l_{n-1}=\frac{\bi}{2nb}$, we obtain the resulting complex shifts $\h_{E_n}$. For $n=6,7,8$, they are given by $\frac{3}{8b}, \frac{3}{7b}$ and $\frac{15}{32b}$ respectively.

To compute explicitly the action of the Casimir operators using \eqref{CasWeylEq}, we use the \texttt{WeylCharacterRing::fundamental\_weights} function in \texttt{Sage} to output the set of weights for each fundamental representation of $\g$. Then, we manipulate the set of weights to get the exponents $L_i$ in the expression for each Casimir of the form
$$\pi_\l(\bC_k)=\sum e^{4\pi b L_i},$$
substitute the complex shifts, compute symbolically the real and imaginary parts of $\pi_\l(\bC_k)$ separately and simplify the results to obtain the action listed in Theorem \ref{thmother}.
\end{proof}

In type $F_4$, there is no parabolic subalgebra of type $A_3$. Instead, we consider the degenerate representations obtained by folding the parabolic subalgebra of type $A_1 \x A_2$ with respect to two different root lengths, namely the root index $(1,2,4)$ and $(1,3,4)$ respectively. The corresponding reduced words can be chosen as
\Eqn{\bi_0&=(1\;21\;4) (3 2 3 4 3 2 3 1 2 3 2 1 4 3 2 3 1 2 4 3)}
and
\Eqn{
\bi_0&=(4\;34\;1)(2 3 2 1 2 3 2 4 3 2 3 4 1 2 3 2 4 3 1 2).
}
respectively. They are chosen to be Langlands dual of each other with $1\corr 4$ and $2\corr 3$.

In type $G_2$, there are also two ways to obtain the maximal degenerate representation by folding the parabolic subalgebra of type $A_1$ with respect to root $1$ and $2$. 

The results are summarized as follows, where details are omitted as they can be obtained by direct computations from the definition. 
\begin{Thm}\label{thmother2} The complex shifts, together with the actions of the generalized Casimirs $\bC_k$ of the maximal degenerate representations in type $F_4$ and $G_2$ are given as follows.
$$\begin{array}{|c|c|c|l|}
\hline
\emph{Type}&J&\emph{Complex Shifts}&\multicolumn{1}{|c|}{\emph{Action of Casimirs}}\\
\hline\hline
\multirow{4}{*}{$F_4$}&\{1,2,4\}& \h_{F_4}^s=\frac{7}{24b_s}&\begin{array}{rl}\pi_\l(\bC_1)&=-1\\
\pi_\l(\bC_2)&=e^{24\pi b_s\l}+e^{-24\pi b_s\l}\\
\pi_\l(\bC_3)&=e^{24\pi b_s\l}+e^{-24\pi b_s\l}-1\\
\pi_\l(\bC_4)&=0\end{array} \\\cline{2-4}
& \{1,3,4\}&\h_{F_4}^l=\frac{11}{24b}&\begin{array}{rl}\pi_\l(\bC_1)&=-1\\
\pi_\l(\bC_2)&=e^{24\pi b\l}+e^{-24\pi b\l}\\
\pi_\l(\bC_3)&=2\\
\pi_\l(\bC_4)&=-1\end{array}\\\hline
\multirow{4}{*}{$G_2$}&\{1\}& \h_{G_2}^s=\frac{1}{8b_s}&\begin{array}{rl}\pi_\l(\bC_1)&=e^{8\pi b_s\l}+e^{-8\pi b_s\l}\\
\pi_\l(\bC_2)&=e^{8\pi b_s\l}+e^{-8\pi b_s\l}+1\end{array} \\\cline{2-4}
& \{2\}&\h_{G_2}^l=\frac{3}{8b}&\begin{array}{rl}
\pi_\l(\bC_1)&=e^{8\pi b\l}+e^{-8\pi b\l}\\
\pi_\l(\bC_2)&=-1\end{array}\\\hline
\end{array}$$
\end{Thm}

\subsection{Modular double counterpart} \label{sec:higher:mod}
As explained at the end of the discussion of type $A_n$ in Section \ref{sec:An:mod}, one can also obtain the modular double counterpart of the symmetric folding by substituting 
\Eq{
\l_1=\cdots=\l_{n-1}=\frac{\bi b}{2n}
}
instead, i.e. replacing $b$ by $b\inv$ in the folding of the $A_{n-1}$ parabolic part. The resulting quiver $\til{\cX}_q^0$ is the same as that of $\cX_q^0$ but with the weights of $d_\star$ being $\frac{1}{n}$ instead. More generally, we have a modular double counterpart to Theorem \ref{mainthm}.

\begin{Thm}\label{mainthmmod} For any parabolic subgroup $W_J\subset W$ of type $A_{k_1}\x\cdots \x A_{k_m}$, we have another new family of homomorphisms of $\fD_q(\g)$ into a skew-symmetrizable quantum cluster algebra $\cO_q(\til{\cX}^0)$, where the multiplier of the symmetric node $d_\star$ is inverse to that of $\cX_q^0$.
\end{Thm}

Consequently, we obtain another family of representations of $\cU_q(\g_\R)$ by positive operators.
We call this the \emph{modular double counterpart of the degenerate positive representations}, denoted by $\cP_{\til{\l}}^{0,J}$.

In the case of maximal degenerate representations, one can compute using techniques from the previous subsection the actions of $\bC_k$ for the modular double counterpart, which are again parametrized by a single scalar $\l$. In general, they are Laurent polynomials in $e^{\pi \ol}$, with coefficients in terms of $\bc_k:=\bin{n\\k}_{q^{\frac{1}{n}}}$ such that, by the discussion of Theorem \ref{thmmod}, setting $\bc_k=0$ reduces the expression back to that of $\cP_\l^{0,J}$ in Theorem \ref{thmother}.

For example, in type $B_n$, some of the actions of the Casimir elements in the modular double counterpart of the maximal degenerate representations are given by
\Eqn{
\pi_{\til{\l}}(\bC_1)&=1+\bc_1(e^{4\pi b_s\l}+e^{-4\pi b_s\l})\\
\pi_{\til{\l}}(\bC_2)&=\bc_1\left(\bc_1+(e^{4\pi b_s\l}+e^{-4\pi b_s\l})+\bc_2(e^{8\pi b_s\l}+e^{-8\pi b_s\l})\right)\\
&\vdots\\
\pi_{\til{\l}}(\bC_n)&=e^{2\pi n b_s\l}+e^{-2\pi nb_s\l}+\sum_{k=1}^{\lfloor \frac{n}{2}\rfloor} \bc_k(e^{2\pi(n-2k)b_s\l}+e^{-2\pi(n-2k)b_s\l})
}
The actions for other Lie types can be computed similarly, but their expressions are omitted because they are in general quite complicated and not very illuminating.

\section{Classification of regular positive representations}\label{sec:class}
To extend the class of positive representations to incorporate the new family related to their images in quantum cluster algebra, we make the following definition.
\begin{Def}\label{regposrep} A representation $\cP$ of $\cU_q(\g_\R)$ is called a \emph{regular positive representation} if the Chevalley generators of $\cU_q(\g_\R)$ can be realized as a polarization of elements from the quantum algebra $\cO_q(\cX)$ of regular functions of some quantum cluster variety. 
\end{Def}
In other words,
\begin{itemize}
\item There exists a quantum torus algebra $\cX_q^\bQ$ and a homomorphism $\cU_q(\g)\to\cX_q^\bQ$.
\item The Chevalley generators are realized as universally Laurent polynomials in every quantum cluster chart mutation equivalent to $\cX_q^\bQ$
\item The representation $\cP$ is obtained from the above homomorphism by choosing a polarization of $\cX_q^\bQ$.
\end{itemize}
If $\cP$ is irreducible, it is necessary that the Casimir elements act as multiplications by real scalars.

For the standard positive representation $\cP_\l$, recall that we can always choose positive weight parameters $\l\in\R_{\geq0}^{\mathrm{rank}(\g)}$ since $\cP_\l\simeq \cP_{w\cdot \l}$ for any Weyl action $w\in W$. Now from the results of Theorem \ref{thmother} on the maximal degenerate representations, we note that the Casimir actions are the same for both $\cP_\l^{0,J}$ and $\cP_{-\l}^{0,{J^*}}$ where $J^*$ is obtained by replacing each index with their Dynkin involutions. This suggests that we can always restrict ourselves to a positive weight parameter $\l\in\R_{\geq 0}$.

In fact this is true for other ranks of $J$ as well.
\begin{Prop} Let $\mathrm{rank}(\g)=n$. For any $J\subset I$ and weight parameters $\l\in\R^{n-|J|}$, there exists $J^+\subset I$ with $|J|=|J^+|$ and $\l^+\in\R_{\geq0}^{n-|J|}$ such that
$$\cP_\l^{0,J}\simeq \cP_{\l^+}^{0,J^+}.$$
In particular we can always restrict to the case of positive weight parameters. 
\end{Prop}
\begin{proof} We consider the Weyl group action on the weight parameter vector $\L=(\l_1,\l_2...,\l_n)$. Assume $\l_j=0$ for $j\in J$. Then $\L$ lies in the corank $|J|$ wall of a Weyl chamber in the weight space. Hence there exists $w\in W$ such that $\L^+:=w\cdot \L$ lies in the corank $|J|$ wall of the positive Weyl chamber. In particular, there are $|J|$ zero entries in $\L^+$. Let $J^+:=\{j: \L_j^+=0\}\subset I$.

If we now consider the Weyl action on a generic weight parameter vector $\L$, this means the $J^+$ entries are parametrized by $\l_j$ for $j\in J$. We can further act by a Weyl element (generated by $s_j\in W$ with $j\in J^+$) in such a way that $(\L^+)_{j\in J^+}$ is a permutation of $(\l_j)_{j\in J}$, while the other entries are parametrized by the remaining $n-|J|$ positive entries of $\L^+$ which we denote by $\l^+$, possibly with a shift in $\l_j$ for $j\in J$. Upon setting $\l_j$ ($j\in J$) to be the complex values required for the symmetric folding, the positivity of the central monomials $Q_i$ ($i\notin J^+$) ensures that the resulting parameters have the appropriate complex shifts coinciding with that of the degenerate representation $P_{\l^+}^{0,J^+}$. By Proposition \ref{mainprop}, the resulting representations are unitarily equivalent. In particular, the action of the Casimir operators are the same for both $\cP_\l^{0,J}\simeq \cP_{\l^+}^{0,J^+}$.
\end{proof}

As a consequence, we currently have the following list of known irreducible regular positive representations:
\begin{itemize}
\item[(1)] The standard positive representations $\cP_\l$, parametrized by $n$ positive scalars.
\item[(2)] The parabolic positive representations $\cP_\l^J$, with respect to the parabolic subgroup $W_J\subset W$ and parametrized by $n-|J|$ positive scalars\footnote{The analysis of the Casimir operators for parabolic positive representations will be done in a subsequent publication.}.
\item[(3)] The degenerate representations $\cP_\l^{0,J}$ with respect to the parabolic subgroup $W_J\subset W$ of type $A_{k_1}\x\cdots \x A_{k_m}$ and parametrized by $n-(k_1+\cdots + k_m)$ positive scalars.
\item[(4)] The modular double counterpart of the degenerate representations $\cP_{\til{\l}}^{0,J}$, also parametrized by $n-(k_1+\cdots + k_m)$ positive scalars.
\item[(5)] A mixture of type (2)--(4) for disconnected subsets of Dynkin indices.
\end{itemize}
\begin{Ex} As an example to explain (5), say, in type $A_6$, we can write $\bi_0=(1)(3)(565)\bi'$, and do a parabolic reduction for $J=\{1\}$, a degenerate reduction for $J=\{3\}$ and a modular degenerate reduction for $J=\{5,6\}$. The resulting representation is then a regular positive representation parametrized by $$6-1-1-2=2$$ positive weight parameters.
\end{Ex}
We state the main conjecture as follows.
\begin{Con}\label{mainconj}The class of all irreducible regular positive representations of $\cU_q(\g_\R)$ is classified by the five families listed in (1)--(5) above.
\end{Con}
Since the tensor product of the positive representations is just given by amalgamation of two copies of the basic quiver, such that the action of the Chevalley generators are  the concatenation of the telescopic paths, we also expect that in general,
\begin{Con} The tensor product of regular positive representations is again regular, and they can be decomposed into a direct integral of regular positive representations.
\end{Con}
In view of Theorem \ref{sl2casspec}, it is interesting to classify those tensor products in which the decomposition involves only the standard positive representations. In other words, we wish to understand the closure of the braided tensor category of the standard positive representations, which will be important in the theory of integrable system and quantum geometry \cite{FG1, GS}.

\subsection{Example in type $A_2$}\label{sec:higher:A2}

Let us illustrate the joint spectrum $\left(\pi_\l(\bC_1),\pi_\l(\bC_2)\right)$ of the Casimirs of the representations of $\cU_q(\sl_3)$ listed in Conjecture \ref{mainconj}. We assume that $q$ is not a root of unity and the central parameters $\l$ are positive. Note that in type $A_2$, despite the discussion in Remark \ref{notmut}, the equivalence $\cP_\l^{0,J}\simeq \cP_{-\l}^{0,J^*}$ can actually be realized as cluster mutations and a change of polarizations. 

The actions of the Casimirs in different families are calculated in \cite{Ip5}, \cite{Ip8}, Theorem \ref{thmother} and Section \ref{sec:higher:mod} respectively. They are summarized in the table below.

$$\begin{array}{|c|c|l|}
\hline
\mbox{Representations}&J&\multicolumn{1}{|c|}{\mbox{Action of Casimirs}}\\
\hline\hline
\cP_\l&&\begin{array}{rl}
\pi_\l(\bC_1)&=e^{\frac43\pi b\l_1+\frac83\pi b\l_2}+e^{\frac43\pi b\l_1-\frac43\pi b\l_2}+e^{-\frac83\pi b\l_1-\frac43\pi b\l_2}\\
\pi_\l(\bC_2)&=e^{\frac83\pi b\l_1+\frac43\pi b\l_2}+e^{-\frac43\pi b\l_1+\frac43\pi b\l_2}+e^{-\frac43\pi b\l_1-\frac83\pi b\l_2}\end{array}\\\hline
\multirow{4}{*}{$\cP_\l^J$}&\{1\}&\begin{array}{rl}
\pi_\l(\bC_1)&=e^{-\frac83\pi b \l}-(q+q\inv)e^{\frac43\pi b\l}\\
\pi_\l(\bC_2)&=e^{\frac83\pi b\l}-(q+q\inv)e^{-\frac43\pi b\l}\end{array}\\\cline{2-3}
&\{2\}&\begin{array}{rl}
\pi_\l(\bC_1)&=e^{\frac83\pi b \l}-(q+q\inv)e^{-\frac43\pi b\l}\\
\pi_\l(\bC_2)&=e^{-\frac83\pi b\l}-(q+q\inv)e^{\frac43\pi b\l}\end{array}\\\hline
\multirow{4}{*}{$\cP_\l^{0,J}$}&\{1,2\}&\begin{array}{rl}
\pi_\l(\bC_1)&=0\\
\pi_\l(\bC_2)&=0\end{array}\\\cline{2-3}
&\{1\}&\begin{array}{rl}
\pi_\l(\bC_1)&=e^{\frac83\pi b\l}\\
\pi_\l(\bC_2)&=e^{-\frac83\pi b\l}\end{array}\\\cline{2-3}
&\{2\}&\begin{array}{rl}
\pi_\l(\bC_1)&=e^{-\frac83\pi b\l}\\
\pi_\l(\bC_2)&=e^{\frac83\pi b\l}\end{array}\\\hline
\multirow{4}{*}{$\cP_{\til{\l}}^{0,J}$}&\{1,2\}&\begin{array}{rl}
\pi_\l(\bC_1)&=q^{\frac23}+1+q^{-\frac23}\\
\pi_\l(\bC_2)&=q^{\frac23}+1+q^{-\frac23}\end{array}\\\cline{2-3}
&\{1\}&\begin{array}{rl}
\pi_\l(\bC_1)&=e^{\frac83\pi b\l}+(q^\half+q^{-\half})e^{-\frac43\pi b\l}\\
\pi_\l(\bC_2)&=e^{-\frac83\pi b\l}+(q^\half+q^{-\half})e^{\frac43\pi b\l}\end{array}\\\cline{2-3}
&\{2\}&\begin{array}{rl}
\pi_\l(\bC_1)&=e^{-\frac83\pi b\l}+(q^\half+q^{-\half})e^{\frac43\pi b\l}\\
\pi_\l(\bC_2)&=e^{\frac83\pi b\l}+(q^\half+q^{-\half})e^{-\frac43\pi b\l}\end{array}\\\hline
\end{array}$$

A plot (for all possible positive weight parameters $\l$ and a generic $q$ which is taken to be $b\sim0.5$) of the joint spectrum $\left(\pi_\l(\bC_1),\pi_\l(\bC_2)\right)$ of the list of regular positive representations of $\cU_q(\sl(3,\R))$ is given in Figure \ref{plot}. 

\begin{figure}[htb!]
\centering
\includegraphics[width=8cm]{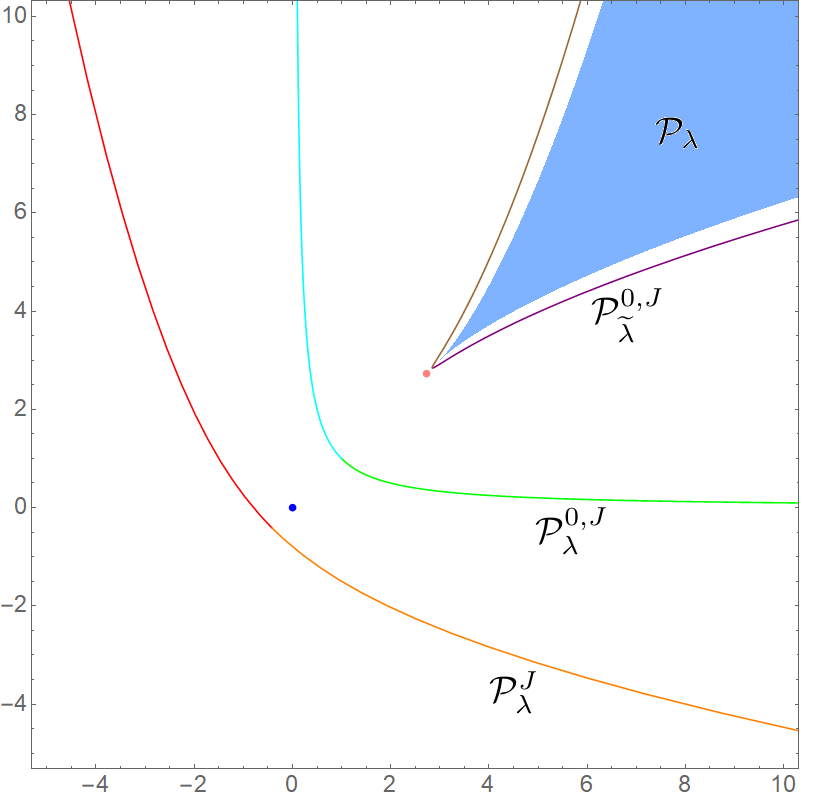}
\caption{A plot $(x,y)=(\pi_\l(\bC_1),\pi_\l(\bC_2))$ of the different families of the joint spectrum of the Casimir elements with different colors.}\label{plot}
\end{figure}

Recall that \cite{Ip5} the boundary of the spectrum of the standard positive Casimirs is given by the discriminant variety
\Eq{
(xy+9)^2=4(x^3+y^3+27),
}
which is independent of $q$. We observe that as expected, the spectral curves do not cross each other, reflecting the fact that these families of representations are not equivalent to one another.

As a generalization of Remark \ref{b=1}, the situation becomes interesting when $q$ is a root of unity. For $b=1$, we observe that the actions of Casimirs for $\cP_{\l}^{0,J}$ and $\cP_{\til{\l}}^{0,J}$ are the same, since $q^\half+q^{-\half}=0$. On the other hand, for $b=\frac{1}{\sqrt{2}}$, we note that the joint spectrum of the degenerate $\cP_\l^{0,J}$ and parabolic $\cP_{\l}^{J}$ coincide as well, since $q+q\inv=0$. These curious coincidences may be related to the representation theory of (compact) quantum groups at root of unity, and are worth studying in the future.
\begin{appendices}
\section{Labeling of Dynkin diagrams}\label{App:Dyn}
In this paper, the labeling of the simple roots are chosen such that (except for type $F_4$) the indices $1,2,...,n-1$ form a parabolic $A_{n-1}$ subalgebra of $\g$. We denote the short roots by black nodes.
\begin{itemize}
\item Type $A_n$: 
\begin{center}
  \begin{tikzpicture}[scale=.4]
    \draw[xshift=0 cm,thick] (0 cm, 0) circle (.3 cm);
    \foreach \x in {1,...,5}
    \draw[xshift=\x cm,thick] (\x cm,0) circle (.3cm);
    \draw[dotted,thick] (8.3 cm,0) -- +(1.4 cm,0);
    \foreach \y in {0.15,...,3.15}
    \draw[xshift=\y cm,thick] (\y cm,0) -- +(1.4 cm,0);
    \foreach \z in {1,...,5}
    \node at (2*\z-2,-1) {$\z$};
\node at (10,-1){$n$};
  \end{tikzpicture}
\end{center}
\item Type $B_n$: 
\begin{center}
  \begin{tikzpicture}[scale=.4]
    \foreach \x in {0,...,4}
    \draw[xshift=\x cm,thick] (\x cm,0) circle (.3cm);
    \draw[dotted,thick] (6.3 cm,0) -- +(1.4 cm,0);
    \draw[xshift=0 cm,thick,fill=black] (10 cm, 0) circle (.3 cm);
    \foreach \y in {0.15,...,2.15}
    \draw[xshift=\y cm,thick] (\y cm,0) -- +(1.4 cm,0);
    \draw[thick] (8.3 cm, .1 cm) -- +(1.4 cm,0);
    \draw[thick] (8.3 cm, -.1 cm) -- +(1.4 cm,0);
    \foreach \z in {1,...,4}
    \node at (2*\z-2,-1) {$\z$};
    \node at (8,-1){\small$n-1$};
\node at (10,-1){$n$};
  \end{tikzpicture}
\end{center}
\item Type $C_n$: 
\begin{center}
  \begin{tikzpicture}[scale=.4]
    \foreach \x in {0,...,4}
    \draw[xshift=\x cm,thick,fill=black] (\x cm,0) circle (.3cm);
    \draw[dotted,thick] (6.3 cm,0) -- +(1.4 cm,0);
    \draw[xshift=0 cm,thick] (10 cm, 0) circle (.3 cm);
    \foreach \y in {0.15,...,2.15}
    \draw[xshift=\y cm,thick] (\y cm,0) -- +(1.4 cm,0);
    \draw[thick] (8.3 cm, .1 cm) -- +(1.4 cm,0);
    \draw[thick] (8.3 cm, -.1 cm) -- +(1.4 cm,0);
    \foreach \z in {1,...,4}
    \node at (2*\z-2,-1) {$\z$};
    \node at (8,-1){\small$n-1$};
\node at (10,-1){$n$};
  \end{tikzpicture}
\end{center}
\item Type $D_n$: 
\begin{center}
  \begin{tikzpicture}[scale=.4]
    \draw[xshift=0 cm,thick] (10 cm, 1) circle (.3 cm);
    \draw[xshift=0 cm,thick] (10 cm, -1) circle (.3 cm);
    \foreach \x in {0,...,4}
    \draw[xshift=\x cm,thick] (\x cm,0) circle (.3cm);
    \draw[dotted,thick] (6.3 cm,0) -- +(1.4 cm,0);
   \draw[xshift=0.25 cm] (8 cm,0) -- +(1.4 cm,-1);
   \draw[xshift=0.25 cm] (8 cm,0) -- +(1.4 cm,1);   
 \foreach \y in {0.15,...,2.15}
    \draw[xshift=\y cm,thick] (\y cm,0) -- +(1.4 cm,0);
    \foreach \z in {1,...,4}
    \node at (2*\z-2,-1) {$\z$};
\node at (8,-1){\small $n-2$};
\node at (11,-1){$n$};
\node at (11.5,1){\small $n-1$};
  \end{tikzpicture}
\end{center}
\item Type $E_n$:
\begin{center}
  \begin{tikzpicture}[scale=.4]
    \draw[xshift=0 cm,thick] (0 cm, 0) circle (.3 cm);
    \foreach \x in {1,...,4}
    \draw[xshift=\x cm,thick] (\x cm,0) circle (.3cm);
    \foreach \y in {0.15,...,2.15}
    \draw[xshift=\y cm,thick] (\y cm,0) -- +(1.4 cm,0);
    \foreach \z in {1,...,4}
    \node at (2*\z-2,1) {$\z$};
    \draw[xshift=3.15 cm,thick, dotted] (3.15 cm,0) -- +(1.4 cm,0);
    \node at (8,1) {$n-1$};
\draw[xshift=0 cm,thick] (4 cm, -2) circle (.3 cm);
  \draw[xshift=0 cm] (4 cm,-0.25) -- +(0 cm,-1.5);
\node at (4,-3){$n$};
  \end{tikzpicture}
\end{center}
\item Type $F_4$:
 \begin{center}
  \begin{tikzpicture}[scale=.4]
    \draw[thick] (-2 cm ,0) circle (.3 cm);
	\node at (-2,-1) {$1$};
    \draw[thick] (0 ,0) circle (.3 cm);
	\node at (0,-1) {$2$};
    \draw[thick,fill=black] (2 cm,0) circle (.3 cm);
	\node at (2,-1) {$3$};
    \draw[thick,fill=black] (4 cm,0) circle (.3 cm);
	\node at (4,-1) {$4$};
    \draw[thick] (15: 3mm) -- +(1.5 cm, 0);
    \draw[xshift=-2 cm,thick] (0: 3 mm) -- +(1.4 cm, 0);
    \draw[thick] (-15: 3 mm) -- +(1.5 cm, 0);
    \draw[xshift=2 cm,thick] (0: 3 mm) -- +(1.4 cm, 0);
  \end{tikzpicture}
\end{center}
\item Type $G_2$: 
\begin{center}
  \begin{tikzpicture}[scale=.4]
    \draw[thick] (0 ,0) circle (.3 cm);
	\node at (0,-1) {$1$};
    \draw[thick,fill=black] (2 cm,0) circle (.3 cm);
	\node at (2,-1) {$2$};
    \draw[thick] (30: 3mm) -- +(1.5 cm, 0);
    \draw[thick] (0: 3 mm) -- +(1.5 cm, 0);
    \draw[thick] (-30: 3 mm) -- +(1.5 cm, 0);
  \end{tikzpicture}
\end{center}
\end{itemize}
\section{Mutation sequence for symmetric folding}\label{App:mutate}
In this section, we describe the mutation sequence for symmetric folding outlined in \cite{SS2}, and include the assignments of the corresponding changes in the central parameters, together with the monomial components of the product $\Xi$ of frozen variables.

We start with the self-folded quiver as in Figure \ref{fig-A4} (we provided one rank higher for clarity) and colored in red the resulting monomial components of $\Xi$.
\begin{figure}[H] 
\centering
\begin{tikzpicture}[scale=0.7, every node/.style={inner sep=0, minimum size=0.2cm, thick}, x=1cm, y=1.3cm]

\node(33) at (4,0) [draw, fill=red]{};
\node(35) at (6,0) [draw, fill=red, label={[xshift=1.3em, yshift=-0.7em] \tiny $-2\l_5$}]{};

\node (28) at (3,1) [draw, fill=red]{};
\node (34) at (5,1) [draw, circle, fill=red]{};
\node (32) at (7,1) [draw, fill=red, label={[xshift=1.3em, yshift=-0.7em] \tiny $-2\l_4$}]{};

\node (21) at (2,2) [draw, fill=red]{};
\node (29) at (4,2) [draw, circle, fill=red]{};
\node (31) at (6,2) [draw, circle, fill=red]{};
\node (27) at (8,2) [draw, fill=red, label={[xshift=1.3em, yshift=-0.7em] \tiny $-2\l_3$}]{};

\node (12) at (1,3) [draw, fill=red]{};
\node (22) at (3,3) [draw, circle, fill=red]{};
\node (30) at (5,3) [draw, circle, fill=red]{};
\node (26) at (7,3) [draw, circle, fill=red]{};
\node (20) at (9,3) [draw, fill=red,label={[xshift=1.3em, yshift=-0.7em] \tiny $-2\l_2$}]{};

\node (1) at (0,4) [draw, fill=red]{};
\node (13) at (2,4) [draw, circle, fill=red]{};
\node (23) at (4,4) [draw, circle, fill=red]{};
\node (25) at (6,4) [draw, circle, fill=red]{};
\node (19) at (8,4) [draw, circle, fill=red]{};
\node (11) at (10,4) [draw, fill=red,label={[xshift=1.3em, yshift=-0.7em] \tiny $-2\l_1$}]{};

\node (2) at (1,5) [draw, circle, fill=red]{};
\node (14) at (3,5) [draw, circle, fill=red]{};
\node (24) at (5,5) [draw, circle, fill=red]{};
\node (18) at (7,5) [draw, circle, fill=red]{};
\node (10) at (9,5) [draw, circle, fill=red]{};

\node (3) at (2,6) [draw, circle]{};
\node (15) at (4,6) [draw, circle]{};
\node (17) at (6,6) [draw, circle]{};
\node (9) at (8,6) [draw, circle]{};
\node (36) at (10,6) [draw, circle,label={[xshift=0em, yshift=0em] \tiny $4\l_1$}]{};

\node (4) at (3,7) [draw, circle]{};
\node (16) at (5,7) [draw, circle]{};
\node (8) at (7,7) [draw, circle]{};
\node (37) at (9,7) [draw, circle,label={[xshift=0em, yshift=0em] \tiny $4\l_2$}]{};

\node (5) at (4,8) [draw, circle]{};
\node (7) at (6,8) [draw, circle]{};
\node (38) at (8,8) [draw, circle,label={[xshift=0em, yshift=0em] \tiny $4\l_3$}]{};

\node (6) at (5,9) [draw, circle]{};
\node (39) at (7,9) [draw, circle,label={[xshift=0em, yshift=0em] \tiny $4\l_4$}]{};

\node (40) at (6,10) [draw, circle,label={[xshift=0em, yshift=0em] \tiny $4\l_5$}]{};

\drawpath{33,35,34,32,31,27,26,20,19,11,10,19,26,31,34,29,31,30,26,25,19,18,25,30,29,22,30,23,25,24,23,22,13,23,14,13,2,1,13,12,22,21,29,28,34,33}{}
\drawpath{2,3,14,15,24,17,18,9,10,36,9,17,15,3,4,15,16,17,8,9,37,8,16,4,5,16,7,8,38,7,5,6,7,39,5}{}
\drawpath{40,7}{}
\drawpath{5,40}{bend left}
\drawpath{5,38}{bend left=10}
\drawpath{4,37}{bend left=10}
\drawpath{3,36}{bend left=10}
\drawpath{38,4}{}
\drawpath{37,3}{}
\drawpath{36,2}{}
\drawpath{33,28,21,12,1}{dashed}
\drawpath{11,20,27,32,35}{dashed}
\end{tikzpicture}
\caption{The quiver for $\cX_q^{\mathrm{sym}}$ in type $A_5$.}
\end{figure}

We focus on the upper part of the quiver associated to the self-folded triangle (with dotted arrows added accordingly) shown in Figure \ref{fig-upper}, and rearrange the vertices such that it is a mirror image of \cite[Figure 30]{SS2}. 

\begin{figure}[H] 
\centering
\begin{tikzpicture}[scale=0.7, every node/.style={inner sep=0, minimum size=0.2cm, thick}, x=1cm, y=1.3cm]
\node (2) at (2,5) [draw,  fill=red]{};
\node (14) at (4,5) [draw,  fill=red]{};
\node (24) at (6,5) [draw,  fill=red]{};
\node (18) at (8,5) [draw,  fill=red]{};
\node (10) at (10,5) [draw,  fill=red]{};

\node (3) at (2,6) [draw, circle]{};
\node (15) at (4,6) [draw, circle]{};
\node (17) at (6,6) [draw, circle]{};
\node (9) at (8,6) [draw, circle]{};
\node (36) at (10,6) [draw, circle, fill=green,label={[xshift=0em, yshift=0em] \tiny $4\l_1$}]{\tiny$y_1$};

\node (4) at (2,7) [draw, circle]{};
\node (16) at (4,7) [draw, circle]{};
\node (8) at (6,7) [draw, circle]{};
\node (37) at (8,7) [draw, circle, fill=green,label={[xshift=0em, yshift=0em] \tiny $4\l_2$}]{\tiny$y_2$};

\node (5) at (2,8) [draw, circle]{};
\node (7) at (4,8) [draw, circle]{};
\node (38) at (6,8) [draw, circle, fill=green,label={[xshift=0em, yshift=0em] \tiny $4\l_3$}]{\tiny$y_3$};

\node (6) at (2,9) [draw, circle]{};
\node (39) at (4,9) [draw, circle, fill=green,label={[xshift=0em, yshift=0em] \tiny $4\l_4$}]{\tiny$y_4$};

\node (40) at (2,10) [draw, circle,label={[xshift=0em, yshift=0em] \tiny $4\l_5$}]{};
\drawpath{2,3,14,15,24,17,18,9,10,36,9,17,15,3,4,15,16,17,8,9,37,8,16,4,5,16,7,8,38,7,5,6,7,39,5}{}
\drawpath{40,7}{}
\drawpath{5,40}{bend left}
\drawpath{5,38}{bend left=10}
\drawpath{4,37}{bend left=10}
\drawpath{3,36}{bend left=10}
\drawpath{38,4}{}
\drawpath{37,3}{}
\drawpath{36,2}{}
\drawpath{10,18,24,14,2}{dashed}
\end{tikzpicture}
\caption{The top part of $\cX_q^{\mathrm{sf}}$ rearranged.}\label{fig-upper}
\end{figure}

The preliminary step involves mutations at the outer $n-1$ vertices indexed $y_1,...,y_{n-1}$ in green in Figure \ref{fig-upper}. Note that the monomial transformation contributes the variable $X_{y_1}$ to $\Xi$. The result is shown in Figure \ref{fig-pre}, where we rearrange the top 3 symmetric nodes after the mutations.

\begin{figure}[H] 
\centering
\begin{tikzpicture}[scale=0.7, every node/.style={inner sep=0, minimum size=0.2cm, thick}, x=1cm, y=1.3cm]
\node (2) at (2,5) [draw,  fill=red,label={[xshift=0em, yshift=0em] \tiny $4\l_1$}]{};
\node (14) at (4,5) [draw,  fill=red]{};
\node (24) at (6,5) [draw,  fill=red]{};
\node (18) at (8,5) [draw,  fill=red]{};
\node (10) at (10,5) [draw,  fill=red]{};

\node (3) at (2,6) [draw, circle,label={[xshift=0em, yshift=0em] \tiny $4\l_2$}]{};
\node (15) at (4,6) [draw, circle]{};
\node (17) at (6,6) [draw, circle]{};
\node (9) at (8,6) [draw, circle,label={[xshift=-0.8em, yshift=-0.2em] \tiny $4\l_1$}]{};
\node (36) at (10,6) [draw, circle,fill=red, label={[xshift=0em, yshift=0em] \tiny $-4\l_1$}]{};

\node (4) at (2,7) [draw, circle,label={[xshift=0em, yshift=0em] \tiny $4\l_3$}]{};
\node (16) at (4,7) [draw, circle]{};
\node (8) at (6,7) [draw, circle,label={[xshift=-0.8em, yshift=-0.2em] \tiny $4\l_2$}]{};
\node (37) at (8,7) [draw, circle,label={[xshift=0em, yshift=0em] \tiny $-4\l_2$}]{};

\node (5) at (2,8) [draw, circle,fill=green, label={[xshift=-0.6em, yshift=0em] \tiny $4\l_4$}]{\tiny $y_2$};
\node (7) at (4,8) [draw, circle,fill=green, label={[xshift=0.6em, yshift=0.2em] \tiny $4\l_3$}]{\tiny $y_1$};
\node (38) at (6,8) [draw, circle,label={[xshift=0em, yshift=0em] \tiny $-4\l_3$}]{};

\node (40) at (3,9) [draw, circle,label={[xshift=0em, yshift=0em] \tiny $4\l_5$}]{};
\node (6) at (3,10) [draw, circle]{};
\node (39) at (3,11) [draw, circle,label={[xshift=0em, yshift=0em] \tiny $-4\l_4$}]{};

\drawpath{3,14,15,24,17,18,9,36,10}{}
\drawpath{10,2}{bend right=10}
\drawpath{9,17,15,3,37,9}{}
\drawpath{36,3}{bend right=10}
\drawpath{8,16,4,38,8,37}{}
\drawpath{37,4}{bend right=10}
\drawpath{15,16,7,5,16,17,8}{}
\drawpath{7,38}{}
\drawpath{38,5}{bend right=10}
\drawpath{5,39,7}{}
\drawpath{5,40,7}{}
\drawpath{5,6,7}{}
\drawpath{10,18,24,14,2}{dashed}
\end{tikzpicture}
\caption{The quiver after performing the preliminary step of mutations at $y_1,...,y_4$ of the previous figure.}\label{fig-pre}
\end{figure}
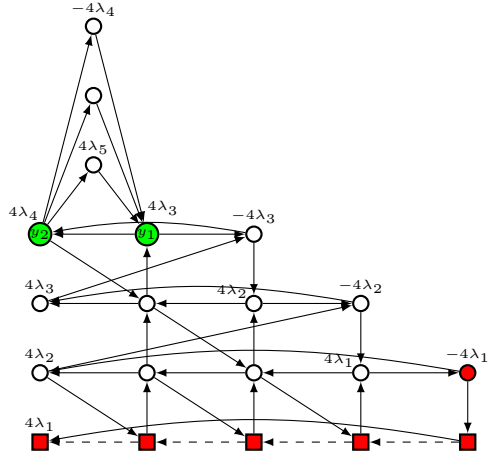

Let us call the row below the symmetric part (i.e. the row containing the green vertices in Figure \ref{fig-pre}) the first row of the folded part. The remaining mutation sequence is then given in $n-2$ waves of mutations. The $k$-th wave of mutations is along the zigzag paths from the second-to-last vertex in the $n-1-k$-th row (labeled by $y_1,...,y_{N_k}$ in each figure, where $N_j=\frac12k(k+3)$.). After each wave of mutations, the last vertex of the zigzag path is moved to the top of the symmetric part, and the left-most vertex of each remaining row to the right-most end of the row above. Finally, each row is shifted accordingly.

\begin{figure}[H] 
\centering
\begin{tikzpicture}[scale=0.7, every node/.style={inner sep=0, minimum size=0.2cm, thick}, x=1cm, y=1.3cm]
\node (2) at (2,5) [draw,  fill=red,label={[xshift=0em, yshift=0em] \tiny $4\l_1$}]{};
\node (14) at (4,5) [draw,  fill=red]{};
\node (24) at (6,5) [draw,  fill=red]{};
\node (18) at (8,5) [draw,  fill=red]{};
\node (10) at (10,5) [draw,  fill=red]{};

\node (3) at (2,6) [draw, circle,label={[xshift=0em, yshift=0em] \tiny $4\l_2$}]{};
\node (15) at (4,6) [draw, circle]{};
\node (17) at (6,6) [draw, circle]{};
\node (9) at (8,6) [draw, circle,label={[xshift=-0.8em, yshift=-0.2em] \tiny $4\l_1$}]{};
\node (36) at (10,6) [draw, circle,fill=red, label={[xshift=0em, yshift=0em] \tiny $-4\l_1$}]{};

\node (4) at (2,7) [draw, circle, fill=green,label={[xshift=0em, yshift=0em] \tiny $4\l_3$}]{\tiny$y_3$};
\node (16) at (4,7) [draw, circle, fill=green]{\tiny$y_2$};
\node (8) at (6,7) [draw, circle,fill=green,label={[xshift=0em, yshift=0.4em] \tiny $4\l_2$}]{\tiny$y_1$};
\node (37) at (8,7) [draw, circle,label={[xshift=0em, yshift=0em] \tiny $-4\l_2$}]{};

\node (7) at (2,8) [draw, circle, fill=green, label={[xshift=-0.6em, yshift=0em] \tiny $4\l_4$}]{\tiny$y_5$};
\node (38) at (4,8) [draw, circle, fill=green]{\tiny$y_4$};

\node (40) at (3,9) [draw, circle,label={[xshift=0em, yshift=0em] \tiny $4\l_5$}]{};
\node (6) at (3,10) [draw, circle]{};
\node (39) at (3,11) [draw, circle,label={[xshift=1.6em, yshift=-0.6em] \tiny $-4\l_4$}]{};
\node (5) at (3,12) [draw, circle, label={[xshift=0em, yshift=0em] \tiny $-4\l_3-4\l_4$}]{};

\drawpath{3,14,15,24,17,18,9,36,10}{}
\drawpath{9,17,15,3,37,9}{}
\drawpath{36,3}{bend right=10}
\drawpath{8,16,4,38,8,37}{}
\drawpath{37,4}{bend right=10}
\drawpath{4,15,16,38,7,16,17,8}{}
\drawpath{38,7}{bend right=10}
\drawpath{7,5,38}{}
\drawpath{7,39,38}{}
\drawpath{7,40,38}{}
\drawpath{7,6,38}{}
\drawpath{10,18,24,14,2}{dashed}
\drawpath{10,2}{bend right=10}
\end{tikzpicture}
\caption{After the first wave of mutations at $y_1,y_2$ of the previous figure.}
\end{figure}
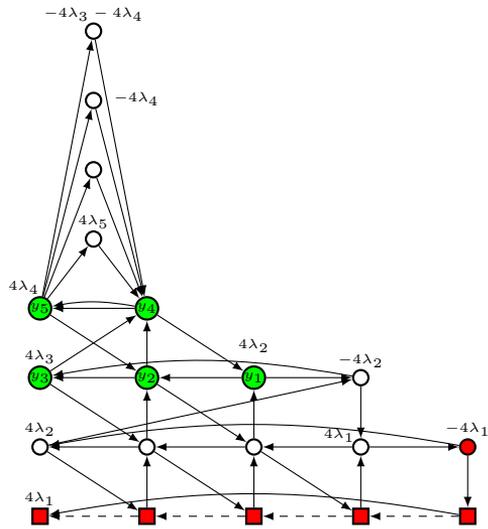

\begin{figure}[H]
\centering
\begin{tikzpicture}[scale=0.7, every node/.style={inner sep=0, minimum size=0.2cm, thick}, x=1cm, y=1.3cm]
\node (2) at (2,5) [draw,  fill=red,label={[xshift=0em, yshift=0em] \tiny $4\l_1$}]{};
\node (14) at (4,5) [draw,  fill=red]{};
\node (24) at (6,5) [draw,  fill=red]{};
\node (18) at (8,5) [draw,  fill=red]{};
\node (10) at (10,5) [draw,  fill=red]{};

\node (3) at (2,6) [draw, circle, fill=green,label={[xshift=0em, yshift=0em] \tiny $4\l_2$}]{\tiny$y_4$};
\node (15) at (4,6) [draw, circle, fill=green]{\tiny$y_3$};
\node (17) at (6,6) [draw, circle, fill=green]{\tiny$y_2$};
\node (9) at (8,6) [draw, circle, fill=green,label={[xshift=0em, yshift=0.4em] \tiny $4\l_1$}]{\tiny$y_1$};
\node (36) at (10,6) [draw, circle,fill=red, label={[xshift=0em, yshift=0em] \tiny $-4\l_1$}]{};

\node (16) at (2,7) [draw, circle, fill=green, label={[xshift=0em, yshift=0em] \tiny $4\l_3$}]{\tiny$y_7$};
\node (8) at (4,7) [draw, circle, fill=green]{\tiny$y_6$};
\node (37) at (6,7) [draw, circle, fill=green]{\tiny$y_5$};

\node (4) at (4,8) [draw, circle, fill=green]{\tiny$y_8$};
\node (38) at (2,8) [draw, circle, fill=green, label={[xshift=-0.6em, yshift=0em] \tiny $4\l_4$}]{\tiny$y_9$};

\node (40) at (3,9) [draw, circle,label={[xshift=0em, yshift=0em] \tiny $4\l_5$}]{};
\node (6) at (3,10) [draw, circle]{};
\node (39) at (3,11) [draw, circle,label={[xshift=2.1em, yshift=-0.6em] \tiny $-4\l_4$}]{};
\node (5) at (3,12) [draw, circle, label={[xshift=2.6em, yshift=-0.6em] \tiny $-4\l_3-4\l_4$}]{};
\node (7) at (3,13) [draw, circle, label={[xshift=0em, yshift=0em] \tiny $-4\l_2-4\l_3-4\l_4$}]{};

\drawpath{3,14,15,24,17,18,9,36,10}{}
\drawpath{9,17,15,3,37,9}{}
\drawpath{36,3}{bend right=10}
\drawpath{8,16,4,38,8}{}
\drawpath{37,16}{bend right=10}
\drawpath{16,15,8,17,37,8}{}
\drawpath{4,38}{bend right=10}
\drawpath{38,5,4}{}
\drawpath{38,39,4}{}
\drawpath{38,40,4}{}
\drawpath{38,6,4}{}
\drawpath{38,7,4}{}
\drawpath{10,18,24,14,2}{dashed}
\drawpath{10,2}{bend right=10}
\end{tikzpicture}
\caption{After the second wave of mutations at $y_1,...,y_5$ of the previous figure.}
\end{figure}
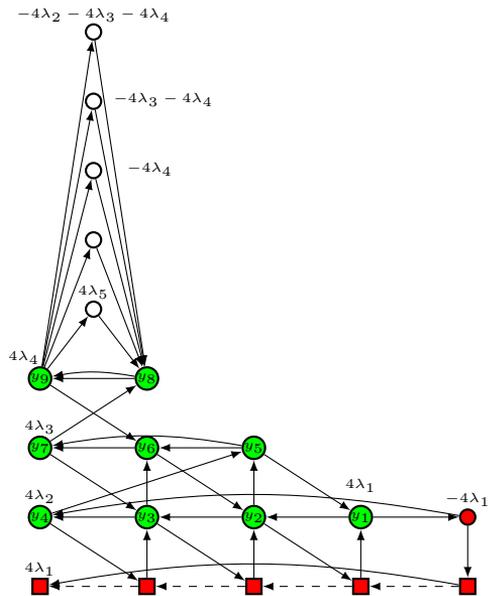

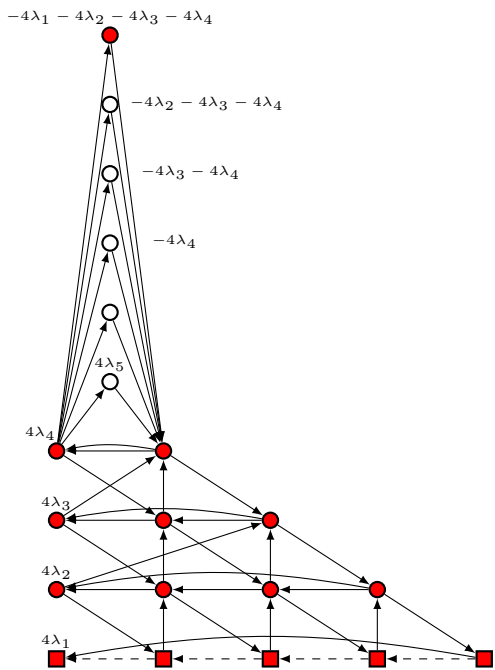
\begin{figure}[H] 
\centering
\begin{tikzpicture}[scale=0.7, every node/.style={inner sep=0, minimum size=0.2cm, thick}, x=1cm, y=1.3cm]
\node (2) at (2,5) [draw,  fill=red,label={[xshift=0em, yshift=0em] \tiny $4\l_1$}]{};
\node (14) at (4,5) [draw,  fill=red]{};
\node (24) at (6,5) [draw,  fill=red]{};
\node (18) at (8,5) [draw,  fill=red]{};
\node (10) at (10,5) [draw,  fill=red]{};

\node (15) at (2,6) [draw, circle,  fill=red,label={[xshift=0em, yshift=0em] \tiny $4\l_2$}]{};
\node (17) at (4,6) [draw, circle,  fill=red]{};
\node (9) at (6,6) [draw, circle,  fill=red]{};
\node (36) at (8,6) [draw, circle,fill=red]{};

\node (8) at (2,7) [draw, circle,  fill=red,label={[xshift=0em, yshift=0em] \tiny $4\l_3$}]{};
\node (37) at (4,7) [draw, circle,  fill=red]{};
\node (3) at (6,7) [draw, circle,  fill=red]{};

\node (4) at (2,8) [draw, circle,  fill=red,label={[xshift=-0.6em, yshift=0em] \tiny $4\l_4$}]{};
\node (16) at (4,8) [draw, circle,  fill=red]{};

\node (40) at (3,9) [draw, circle,label={[xshift=0em, yshift=0em] \tiny $4\l_5$}]{};
\node (6) at (3,10) [draw, circle]{};
\node (39) at (3,11) [draw, circle,label={[xshift=2.4em, yshift=-0.6em] \tiny $-4\l_4$}]{};
\node (5) at (3,12) [draw, circle, label={[xshift=3.0em, yshift=-0.6em] \tiny $-4\l_3-4\l_4$}]{};
\node (7) at (3,13) [draw, circle, label={[xshift=3.6em, yshift=-0.6em] \tiny $-4\l_2-4\l_3-4\l_4$}]{};
\node (38) at (3,14) [draw, circle,  fill=red, label={[xshift=0em, yshift=0em] \tiny $-4\l_1-4\l_2-4\l_3-4\l_4$}]{};

\drawpath{15,14,17,24,9,18,36,9,17,15,3,36,10}{}
\drawpath{8,17,37,9,3,37,8,16,4}{}
\drawpath{4,37,16,3}{}
\drawpath{36,15}{bend right=10}
\drawpath{3,8}{bend right=10}
\drawpath{16,4}{bend right=10}
\drawpath{4,5,16}{}
\drawpath{4,39,16}{}
\drawpath{4,40,16}{}
\drawpath{4,6,16}{}
\drawpath{4,7,16}{}
\drawpath{4,38,16}{}
\drawpath{10,18,24,14,2}{dashed}
\drawpath{10,2}{bend right=10}
\end{tikzpicture}
\caption{The top part of the quiver $\cX_q^{\mathrm{sym}}$ after the last wave of mutations at $y_1,...,y_9$ of the previous figure.}
\end{figure}

Note that only the last wave of mutations affects the monomial components of $\Xi$, where each successive mutation contributes a single monomial. The result includes all vertices from the non-symmetric part, together with the top vertex, which is the last vertex of the zigzag path in the last wave of mutations.
\end{appendices}


\begin{thebibliography}{99}

 \bibitem{BT}
  A.G. Bytsko, J. Teschner,
  \textit{$R$-operator, co-product and Haar-measure for the modular double of $\cU_q(\sl(2,\R))$},
  Comm. Math. Phys., \textbf{240} (1--2), (2003):171--196.
  \bibitem{Car}
R. Carter, 
\textit{Lie algebras of finite and affine type},
\textbf{96}. Cambridge University Press, (2005).
\bibitem{Dr}
V. G. Drinfeld,
\textit{Hopf algebras and the quantum Yang--Baxter equation},
Doklady Akademii Nauk SSSR, \textbf{283} (5), (1985):1060--1064.\\
------, \textit{Quantum Groups},
Proc. Int. Con. Math., Berkeley, (1986):798--820.
\bibitem{Fa1}
  L. D. Faddeev,
  \textit{Discrete Heisenberg--Weyl group and modular group},
  Lett. Math. Phys., \textbf{34} (3), (1995):249--254.
  
\bibitem{Fa2}
  L. D. Faddeev,
  \textit{Modular double of quantum group},
arXiv:math/9912078, (1999).

  \bibitem{FG1}
V. Fock, A. Goncharov,
\textit{Moduli spaces of local systems and higher Teichm\"uller theory},
Publications Math\'ematiques de l'IH\'ES, \textbf{103}, (2006):1--211.

\bibitem{FG3}
V. Fock, A. Goncharov,
\textit{The quantum dilogarithm and representations of quantum cluster varieties},
Invent. Math. \textbf{175}(2), (2009):223--286.

\bibitem{FI}
I. Frenkel, I. Ip,
\textit{Positive representations of split real quantum groups and future perspectives},
Int. Math. Res. Not., \textbf{2014} (8), (2014):2126--2164.

\bibitem{GS}
A. Goncharov, L. Shen,
\textit{Quantum geometry of moduli spaces of local systems and representation theory},
arXiv:1904.10491 (2019).

 \bibitem{Ip1}
  I. Ip,
  \textit{Representation of the quantum plane, its quantum double and harmonic analysis on $GL_q^+(2,\R)$},  
Sel. Math. New Ser., \textbf{19} (4), (2013):987--1082.

\bibitem{Ip2}
I. Ip,
\textit{Positive representations of split real simply-laced quantum groups}, 
Publ. R.I.M.S., \textbf{56} (3), (2020):603--646.
\bibitem{Ip3}
I. Ip,
\textit{Positive representations of split real non-simply-laced quantum groups}, 
J. Alg, \textbf{425}, (2015):245--276.

\bibitem{Ip4}
I. Ip,
\textit{Positive representations of split real quantum groups:the universal $R$ operator},
Int. Math. Res. Not., \textbf{2015} (1), (2015):240--287.

\bibitem{Ip5}
I. Ip,
\textit{Positive Casimir and central characters of split real quantum groups},
Commun. Math. Phys. \textbf{344} (3), (2016):857--888.

\bibitem{Ip7}
I. Ip,
\textit{Cluster realization of $\cU_q(\g)$ and factorization of universal $\cR$ matrix},
Sel. Math. New Ser., \textbf{24} (5), (2018):4461--4553.

\bibitem{Ip8}
I. Ip,
\textit{Parabolic positive representations of $\cU_q(\g_\R)$},
arXiv:2008.08589, (2020).

\bibitem{IY}
I. Ip, M. Yamazaki,
\textit{Quantum Dilogarithm Identities at Root of Unity},
Int. Math. Res. Not., \textbf{2016}(3), (2016):669--695.

\bibitem{Ji1}
M. Jimbo,
\textit{A $q$-difference analogue of $\cU(\g)$ and the Yang--Baxter equation},
Lett. Math. Phys., \textbf{10}, (1985):63--69.

\bibitem{Ka1}
R. M. Kashaev,
\textit{Quantization of Teichm\"uller spaces and the quantum dilogarithm},
Lett. Math. Phys., \textbf{43}(2), (1998):105--115.
\bibitem{Ka2}
R. M. Kashaev,
\textit{The quantum dilogarithm and Dehn twist in quantum Teichm\"uller theory},
Integrable Structures of Exactly Solvable Two-Dimension Models of Quantum Field Theory (Kiev, Ukraine, Sept. 25--30, 2000), NATO Sci. Ser. II Math. Phys. Chem., \textbf{35}, Kluwer, Dordrecht, (2001):211--221.

\bibitem{PT1}
  B. Ponsot, J. Teschner,
  \textit{Liouville bootstrap via harmonic analysis on a noncompact quantum group},
  arXiv:hep-th/9911110, (1999).
  
\bibitem{PT2}
  B. Ponsot, J. Teschner,
  \textit{Clebsch--Gordan and Racah--Wigner coefficients for a continuous series of representations of $\cU_q(\mathfrak{sl}(2,\R))$},
  Comm. Math. Phys., \textbf{224} (3), (2001):613--655.
 
  \bibitem{NT}
  I. Nidaiev, J. Teschner
  \textit{On the relation between the modular double of $\cU_q(\sl(2,\R))$ and the quantum Teichm\"{u}ller theory},
  arXiv:1302.3454  (2013).
  
\bibitem{Sch}
K. Schm\"{u}dgen,
\textit{Operator representations of $U_q(sl_2(\R))$},
Lett. Math. Phys., \textbf{37}, (1996):211--222.

\bibitem{SS1}
G. Schrader, A. Shapiro,
\textit{A cluster realization of $\cU_q(\sl_n)$ from quantum character varieties},
Invent. Math. \textbf{216}(3), (2019):799--846.
\bibitem{SS2}
G. Schrader, A. Shapiro,
\textit{Continuous tensor categories from quantum groups I:algebraic aspects},
arXiv:1708.08107, (2017).
\end{thebibliography}
\end{document}